\documentclass[11pt]{article}
\pdfoutput=1
\usepackage[margin=1in]{geometry}
\usepackage{setspace}
\usepackage{graphicx}
\usepackage{amsmath,amssymb,amsthm,bbm,mathrsfs,enumerate}
\usepackage[colorlinks=true,linkcolor=red,anchorcolor=blue,citecolor=blue,CJKbookmarks=True]{hyperref}
\usepackage{algorithm,algorithmic}
\usepackage{appendix}
\usepackage{booktabs}
\usepackage{threeparttable}  
\usepackage{multicol}
\usepackage{multirow}
\usepackage{caption}
\usepackage{subcaption}
\usepackage{mathtools}
\usepackage[normalem]{ulem}
\usepackage{cancel}
\usepackage{braket}
\usepackage[authoryear,round]{natbib}

\newtheorem{lemma}{Lemma}
\newtheorem{theorem}{Theorem}
\newtheorem{proposition}{Proposition}
\newtheorem{corollary}{Corollary}
\newtheorem{assumption}{Assumption}

\newtheorem{remark}{Remark}

\begin{document}

\allowdisplaybreaks
	
\vspace{-2.0in}
\title{\bf Data-Driven Brownian Reflection Control}

\author{
Guodong Pang\footnote{Department of Computational Applied Mathematics and Operations Research, George R. Brown
School of Engineering, Rice University, Houston, TX 77005; gdpang@rice.edu} \qquad
Dacheng Yao\footnote{State Key Laboratory of Mathematical Sciences, Academy of Mathematics and Systems Science, Chinese Academy of Sciences, Beijing, 100190, China; School of Mathematical Sciences, University of Chinese Academy of Sciences, Beijing, 100049, China; dachengyao@amss.ac.cn} \qquad  
Hao Yin\footnote{Academy of Mathematics and Systems Science, Chinese Academy of Sciences, Beijing, 100190, China; School of Mathematical Sciences, University of Chinese Academy of Sciences, Beijing, 100049, China; yinhao@amss.ac.cn}}

\date{\today}

\maketitle \thispagestyle{empty}

\begin{abstract}
We study a data-driven reflection control problem for a Brownian model with unknown drift and volatility. We first propose a learn-then-optimize (LTO) algorithm: it estimates the policy-relevant parameter during exploration, plugs the estimate into the optimality equation, and exploits the resulting policy---achieving an $O(\sqrt{T})$ finite-time expected regret bound. We further propose two algorithms, adaptive-updating (AU) and full-history adaptive-updating (AU-FH), which continuously update the estimator and reflecting level, attaining an improved $O(\log T)$ regret bound. Notably, AU-FH algorithms leverages all historical data, yielding better performance in numerical simulations. Our analysis decomposes regret into exploration, transient, and learning components. Transient regret from nonstationarity is bounded by the time-integrated deviation of the transition semigroup from stationarity evaluated on the holding cost, which can be further bounded via a Foster-Lyapunov inequality for exponential convergence of the controlled reflected Brownian motion (RBM). For learning regret, we establish local regularity properties together with consistency and mean-squared error bounds for the estimator, which control the stationary cost gap between the learned and optimal reflection policies. In addition, we leverage the monotonicity of the moving boundary Skorokhod map to derive moment bounds for the AU algorithms' switching states, via pathwise comparison with fixed-boundary RBMs.

\vspace{.25in}
\end{abstract}

\section{Introduction}\label{sec:1}
Brownian and diffusion models have been widely used as control formulations and tractable approximations for stochastic systems in operations research, including queueing networks, production-inventory systems, and revenue management operations (see, e.g., \citet{chen2001fundamentals}, \citet{harrison2013brownian}, and \citet{constantinides1978existence}). 
These models are often developed under complete knowledge of system primitives. In numerous practical applications, however, certain system primitives are unknown a priori; consequently, the controller must concurrently learn from operational data and regulate system behavior in real time. This, in turn, presents decision makers with the dual challenge of jointly estimating unknown parameters and implementing optimal control---balancing exploration (to refine estimation accuracy) and exploitation (to minimize operational costs).

Motivated by the gap between fully specified analytical models and data-driven implementation, recent research has increasingly focused on integrating model-based control with learning methods for stochastic systems. In this paper, we consider a data-driven variant of the Brownian production-inventory model studied by~\citet{yang2020optimality}---whereas their model operates under a full-information setting, our model relaxes this by assuming the drift and volatility to be unknown. We leverage the full-information optimal reflection control policy to design data-driven algorithms: these algorithms are not only interpretable and implementable but also provide strong performance guarantees. To evaluate these policies, we use the finite-time regret---defined as the expected difference between the cumulative cost of the learning-based policy over a finite horizon $T$ and that of the corresponding full-information optimal policy, with both policies evaluated under identical underlying stochastic realizations.
 
Since the optimality equation depends on the drift $\theta$ and volatility $\sigma$ through the composite parameter $\gamma:=-2\theta/\sigma^2$, we exploit this dimensionality reduction to design algorithms for policy-relevant parameter $\gamma$ rather than estimating $\theta$ and $\sigma$ separately. We first propose the learn-then-optimize (LTO) algorithm (Algorithm~\ref{alg:1}): it begins with an exploration policy, estimates $\gamma$ from trajectory data, computes the optimal reflecting level for this estimate, and then exploits the resulting policy. The inherent tradeoff between exploring diverse policies to learn unknown model primitives and exploiting the most promising policy based on historical observations mirrors the exploration-exploitation tradeoff in the classical multi-armed bandit (MAB) problem. However, a key observation is that every reflection control policy generates informative data in our model---thus, we propose adaptive-updating (AU) algorithm (Algorithm~\ref{alg:2}) and full-history adaptive-updating (AU-FH) algorithm (Algorithm~\ref{alg:3}) that integrate exploration and exploitation, continuously updating the estimator and the reflecting level while achieving stronger performance guarantees.

Our main contributions can be summarized as follows.
We use a model-based framework to design data-driven algorithms with finite-time regret guarantees for Brownian reflection control under unknown drift and volatility. For the LTO algorithm, we show that the regret bound is $O(\sqrt{T})$ (Theorem~\ref{thm:LTO4generalh}) when the exploration period is of the square-root order of the finite-time horizon. The AU  and AU-FH algorithms, by contrast, achieve an $O(\log T)$ regret (Theorem~\ref{thm:AU4generalh}). The AU-FH algorithm yields a slight performance improvement by incorporating the full historical data into estimation. 
To get these results, we decompose the regret into three parts: exploration regret, transient regret and learning regret (see Section~\ref{sec:3}). The analysis of transient and learning regret is pivotal to our results, as elaborated below.

Transient regret denotes the cumulative deviation of the expected running cost of the controlled process from the stationary cost associated with the policy implemented over each interval; it arises from the nonstationarity of the controlled process. 
We exploit the fact that the transient regret can be bounded by the exponential convergence bounds for the  transition law from the stationary distribution for an 
appropriately chosen Lyapunov function (see \eqref{eq:transient-rewrite}-\eqref{eq:exponential-ergodicity}). For this purpose, since the process evolves as a reflected Brownian motion (RBM) under a fixed reflection policy,
we adopt a modified Lyapunov function approach proposed by~\citet{Sarantsev2016Explicit,sarantsev2017reflected} to derive an ``explicit" exponential convergence rate to the stationary distribution for the RBM (Proposition~\ref{prop:exp-convergence-rate-rbm}).
This explicit result demonstrates how the transient regret bound depends on the parameters. Notably, the exponential convergence rate is independent of the reflection level, so this result can be uniformly applied to each fixed-policy interval to derive an $O(1)$ transient regret bound.
For the LTO algorithm (which involves only one policy update), the transient regret is bounded by a constant (Proposition~\ref{prop:transient-regret-bound-LTO}). For the AU and AU-FH algorithms, since the number of updates up to time $T$ is of order $\log T$, summing the constant transient bound over all update intervals yields an $O(\log T)$ transient regret (Proposition~\ref{prop:transient-regret-bound-AU}).

The learning regret term corresponds to the stationary cost gap between the learned reflection policy and the full-information optimal policy. Leveraging the ergodicity property (Lemma~\ref{lemma:ergodicity}), we construct an estimator via the time average of the controlled process and establish the consistency of this estimator (Lemma~\ref{lemma:consistency}). 
To bound the learning regret, we first establish the local regularity of the optimality equation (see Section~\ref{sec:5.1}), then derive a mean squared error (MSE) bound (see Section~\ref{sec:5.2}) for the time-average process
(which can be viewed as the reciprocal of the estimator). We then introduce a good event to characterize whether the estimator falls within the neighborhood of the true parameter, where the local regularity property holds. On the good event, we bound the learning regret using the quadratic cost bound (following from local regularity) and the estimator’s MSE bound. For the complement event, we control the learning regret via the uniform cost bounds and the probability of this event.
For the LTO algorithm, the learning regret is $O(1/\tau)$ (Proposition~\ref{prop:learning-regret-bound-LTO}), where $\tau$ denotes the exploration period. Optimizing the regret over $\tau$ reveals the exploration-exploitation tradeoff and yields an $O(\sqrt{T})$ regret bound.
For the AU and AU-FH algorithms integrating exploration and exploitation,the doubling interval strategy ensures the learning regret is bounded by a constant over each update interval, leading to an $O(\log T)$ learning regret bound (Proposition~\ref{prop:learning-regret-bound-AU} and Proposition~\ref{prop:learning-regret-bound-AUFH}).

It is worth noting that, in contrast to the LTO algorithm, the above transient and learning regret analysis for the AU and AU-FH algorithms face an additional technical challenge: both the state and the reflecting level at each update time are random and history-dependent. Thus, when applying the strong Markov property, we require an exponential bound for the random switching states in the transient regret analysis and a polynomial bound for those in the learning regret analysis, respectively.
To address this challenge, we leverage the monotonicity of the moving boundary Skorokhod map to dominate the resulting sample paths---a critical step that enables us to derive the desired bounds (see Section~\ref{sec:4.3}).

Numerical experiments show that the algorithms’ performance aligns with the theoretical analysis. We use REINFORCE (a classic model-free policy gradient algorithm in RL) as a benchmark to demonstrate the effectiveness of designing algorithms that fully exploit the problem structure. We also investigate how the diffusion scale influences the finite-time learning performance of the AU algorithm.

\subsection{Literature Review}\label{sec:1.1}
This paper connects to several interrelated strands of literature. Our work builds on the optimality of reflection control established by~\citet{yang2020optimality} in the full-information setting. More broadly, Brownian control models have been widely used as performance models and control approximations for queueing systems, stochastic networks, inventory and production systems; see~\citet{harrison2013brownian} for a systematic treatment. Classical contributions include singular control of Brownian models, such as~\citet{harrison1983instantaneous},~\citet{karatzas1983class},~\citet{menaldi2005some},~\citet{taksar1985average},~\citet{jack2006singular}, and~\citet{dai2013brownian}; impulse control of Brownian or diffusion models, such as~\citet{harrison1983impulse},~\citet{ormeci2008impulse},~\citet{dai2013brownian}, and~\citet{helmes2017continuous}; and drift-rate control of Brownian processing systems, such as~\citet{ata2005drift}. These studies establish the optimal control policy under known system parameters. In contrast, we develop a data-driven, model-based learning framework for Brownian reflection control.

Our work also relates to the growing literature on data-driven and learning-based continuous-time stochastic control. 
A closely related line of work focuses on data-driven control for diffusion processes.
\citet{christensen2023data} investigate data-driven optimal stopping for a general diffusion process, analyze estimator of the optimal stopping barrier, and derive non-asymptotic upper bounds on the simple regret.
\citet{christensen2023nonparametric} and~\citet{christensen2024learning} develop data-driven methods for impulse and singular control problems in ergodic diffusion settings. They employ nonparametric estimation for objective functions based on invariant distribution, and their performance guarantees are formulated in terms of sublinear long-run average regret per unit time.
For multidimensional reflection problems, \citet{christensen2024data} study reversible Langevin diffusions under normal reflection, where the control variable is the reflection domain and the full-information problem can be reformulated as a shape optimization problem. In the data-driven setting, they construct an optimal domain estimator via nonparametric invariant density estimation and obtain sublinear regret per unit time through an episodic learning scheme.
Our work differs from these papers in several key aspects. 
First, the source of information used for learning is fundamentally distinct. These studies typically assume that the underlying diffusion is ergodic, so their estimators based on invariant density can be estimated from uncontrolled trajectory data.  
In contrast, the uncontrolled Brownian motion with negative drift in our model has no stationary distribution, we generate informative data for estimating the policy-relevant parameter by imposing reflection control.
Second, the statistical structure of our learning problem is parametric (focusing on the policy-relevant parameter $\gamma$) rather than nonparametric.
Third, the regret criteria differ: we analyze finite-time regret, whereas they focus on long-run average regret per unit time.
Most closely related to our Brownian framework, \citet{ankirchner2024learning} study a drift rate learning problem for Brownian control with known variance coefficient, using the optimality of bang-bang policy under the full-information setting to design learning algorithms and derive expected finite-time regret bound for a quadratic cost function. In contrast, we focus on data-driven reflection control, accommodate an unknown volatility parameter, and explicitly analyze the regularity properties of the optimality equation.
Compared to the papers discussed above, our model setup also covers the case where the volatility term are unknown. The volatility parameter plays a critical role in our model:  
variations in $\sigma$ may rapidly alter the optimal policy and optimal cost (see Appendix~\ref{app:C.3}).
Several studies have developed model-free reinforcement learning (RL) methods for controlled diffusions; see, e.g., \citet{jia2022policy,jia2023q}. Related model-free RL approaches have also been extended to financial applications, including continuous-time optimal execution~\citep{wang2026reinforcement} and mean-variance portfolio selection~\citep{huang2024mean,gao2024reinforcement}.

Combining learning and control in queueing systems has received increasing attention; see~\citet{walton2021learning} for a comprehensive review. From a queueing perspective, the reflected Brownian model considered in this paper arises naturally as the heavy-traffic limit of a $G/G/1$ queue, and hence our learning problem can be viewed as a diffusion-scale analogue of learning in queueing systems. There is a stream of work applying RL to Markov decision process (MDP) structured problems in queueing control (see~\citet{liu2022rl,liu2019reinforcement}, \citet{dai2022queueing}, \citet{cohen2024learning}, and~\citet{mastropietro2025fast}). Recent works by \citet{ata2024singular,ata2025drift} and~\citet{zhong2025multilevel} consider the associated diffusion control problems in heavy traffic; their approaches, however, rely on computational methods or deep learning for high-dimensional problems, which differ from ours. Another line of work formulates queueing control as MAB problems; see, e.g., \citet{krishnasamy2021learning}, \citet{freund2024efficient}, \citet{zhalechian2023data}, \citet{jia2024online}, and \citet{xing2025online}. There also exist some works using online learning or statistical learning for queueing systems; see, e.g., \citet{hsu2022integrated}, \citet{zhong2024learning}, \citet{chen2024online}, \citet{zheng2025joint}, and~\citet{chen2026online}.

\subsection{Organization}\label{sec:1.2}
The rest of the paper is organized as follows. Section~\ref{sec:2} introduces the model and main results: we formulate our learning Brownian control problem and present some preliminaries and our main results. Section~\ref{sec:3} gives the sample-path construction regret decomposition for three proposed algorithms. Then, we show how to analyze the transient regret in Section~\ref{sec:4} and the learning regret in Section~\ref{sec:5}. In Section~\ref{sec:6}, we complete the proofs of our main results. In Section~\ref{sec:7}, we evaluate the performance of our proposed algorithms. Finally, Section~\ref{sec:8} concludes the paper with several possible future directions. Supporting proofs can be found in the appendix.

\section{Model and Main Results}\label{sec:2}
\subsection{Model}\label{sec:2.1}
We consider a data-driven learning-based control problem for a Brownian model with unknown drift and variance parameters.
Given an initial state $x\in\mathbb{R}$, the controlled state process $Z=\{Z(t):t\geq0\}$ evolves according to
\[
Z(t)=x+\theta t+\sigma B(t)+Y(t), \quad t\geq 0,
\]
where $\theta<0$ and $\sigma>0$, and $B = \{B(t) : t \geq 0\}$ is a standard Brownian motion defined on some filtered probability space $(\Omega,\{\mathcal{F}_t\},\mathcal{F},\mathbb{P})$.
The control process $Y=\{Y(t):t\geq0\}$ is admissible if it is $\{\mathcal{F}_t\}$-nonanticipative, nondecreasing, and satisfies $Y(0)\geq 0$.
We denote the set of all admissible controls by $\Pi$.
The objective is to design a data-driven admissible control policy $Y\in\Pi$, based on observations of the system state under unknown $\theta$ and $\sigma$, to minimize the long-run average cost
\[
AC(x,Y)=\limsup_{t\to\infty}\mathbb{E}_{x}\left[\frac{1}{t}\int_0^t h\left(Z(u)\right)\mathrm{d} u\right],
\]
where $\mathbb{E}_x$ is the expectation operator conditioning on the initial state $Z(0)=x$. In the full-information setting, this model has been studied by~\citet{yang2020optimality}.

The model is motivated by a production-inventory system in which demand arrives at rate $\lambda$ and production occurs at rate $\mu$, with $\lambda<\mu$. In this case, the net drift $\theta$ can be viewed as $\theta=\lambda-\mu<0$. The nondecreasing control process $Y$ represents cumulative upward adjustments that offset this negative drift, such as disposing of excess inventory or reducing the production rate. 
The state $Z(t)$ represents net inventory: positive values correspond to backlogged demand, whereas negative values correspond to on-hand inventory.
The holding cost function $h$ penalizes deviations of the state on both sides of the desired operating region. 

Let $\gamma:=-2\theta/\sigma^2$.
For unknown parameters $\theta$ and $\sigma$, we assume that the manager has a prior lower bound $\underline{\gamma}>0$ for the true parameter $\gamma$.
The holding cost function $h$ is assumed to satisfy the following conditions.
\begin{assumption}\label{ass:1}
    \begin{enumerate}[$(i)$]
        \item $h:\mathbb{R}\to[0,\infty)$ is continuous and strictly increasing on $[0,\infty)$, and is continuously differentiable with $h^{\prime}(x)<0$ on $(-\infty,0)$. 
        \item $\lim_{x\to-\infty}h(x)=+\infty$.
        \item There exists a constant $0<b<\underline{\gamma}/2$ such that $\limsup_{x\to+\infty}h(x)/e^{bx}<+\infty$.
    \end{enumerate}
\end{assumption}

To evaluate a data-driven control policy, we use the expected regret over a finite horizon $T$, defined by
\begin{align}\label{eq:regret-1}
    R(x,T):=\mathbb{E}_x\left[\int_0^T \left(h(Z(t))-h(Z^*(t))\right)\mathrm{d} t\right],
\end{align}
where $\left\{Z(t):t\geq0\right\}$ is the state process under the data-driven control policy $\left\{Y(t):t\geq 0\right\}$ and $\left\{Z^*(t):t\geq0\right\}$ is the optimally controlled process under known $\theta$ and $\sigma$.
Here, $\mathbb{E}_x$ denotes the expectation conditional on the same initial state $x$ for both processes.

\subsection{Optimal Reflection Control under Full Information}\label{sec:2.2}
To contextualize our results, we first outline the full-information findings below. We start with a pair of continuous $\mathcal{F}_t$-adapted processes $\left(Z_r^{s,y}(t),Y_r^{s,y}(t)\right),t\geq0$ that satisfy the following SDE
\begin{align}\label{eq:SDE-startfrom-s}
     \mathrm{d}Z_r^{s,y}(t)=\theta\mathrm{d}t+\sigma\mathrm{d}B(t)+ \mathrm{d}Y_r^{s,y}(t),\quad t\geq s,
\end{align}
with reflection at $r$, the initial condition $Z_r^{s,y}(s)=y$, and subject to
\begin{enumerate}
    \item[(a)] $Z_r^{s,y}(t)\geq r,\quad t\geq s$;
    \item[(b)] $Y_r^{s,y}$ is nondecreasing and $Y_r^{s,y}(s)=0$;
    \item[(c)] $\int_s^t \mathbf{1}_{\left\{Z_r^{s,y}(u)>r\right\}}\mathrm{d}Y_r^{s,y}(u)=0,\quad t\geq s$;
    \item[(d)] $Z_r^{s,y}(t)=y+\theta (t-s)+\sigma \left( B(t)-B(s)\right)+Y_r^{s,y}(t),t\geq s,\quad \text{a.s.}$.
\end{enumerate}
This reflected SDE \eqref{eq:SDE-startfrom-s} has a unique strong solution (see Theorem~$1.2.1$ of~\citet{pilipenko2014introduction}).

Consider a subclass of admissible control $\Pi^*\subseteq \Pi$ called reflection control (defined as follows).
For a fixed reflection level $r\in\mathbb{R}$, let $\left(Z_r^{0,x}(t),Y_r^{0,x}(t)\right),t\geq0$ be the solution to \eqref{eq:SDE-startfrom-s}---when the process starting from $s=0$ and $y=x$, we omit the superscript and simply denote the pair by $\left(Z_r(t),Y_r(t)\right)$.
Then $Y_r=\left\{Y_r(t):t\geq0\right\}\in\Pi^*$ is the minimal nondecreasing process that keeps the corresponding state process $Z_r=\left\{Z_r(t):t\geq0\right\}$ above $r$ almost surely. 
It is known that $Y_r$ and $Z_r$ can be explicitly expressed as follows, as derived from the Skorokhod problem:
\begin{align}
    &Y_r(t)=\sup_{0\leq u\leq t}\left(r-x-\theta u-\sigma B(u)\right)^+,\label{eq:Skorokhod-mapping-Y}\\
    &Z_r(t)=x+\theta t+\sigma B(t)+\sup_{0\leq u\leq t}\left(r-x-\theta u-\sigma B(u)\right)^+.\label{eq:Skorokhod-mapping-Z}
\end{align}
Furthermore, $Z_r(t)-r$ is a standard reflected Brownian motion (RBM) and $Z_r(+\infty)-r$ follows an exponential distribution with rate $\gamma=-2\theta/\sigma^2$ (refer to Section~$6.2$ of~\citet{chen2001fundamentals}). Here, $Z_r(+\infty)$ denotes a random variable having the stationary distribution of $Z_r$. We denote this stationary distribution by $\pi_r$, i.e.,
\begin{align}\label{eq:stationary-Zr}
\pi_r(\mathrm{d}z)=\gamma e^{-\gamma(z-r)}\mathbf{1}_{ [r,\infty)}(z)\mathrm{d}z.
\end{align}
Thus, under the reflection control $Y_r$, we have
\begin{align*}
    AC(x,Y_r)&=\int_{-\infty}^{\infty}h(z) \pi_r(\mathrm{d}z)=\gamma\int_r^{\infty}h(z)e^{-\gamma (z-r)}\mathrm{d} z.
\end{align*}

Notably, for the reflection control class, the long-run average cost only depends on parameters $\theta$ and $\sigma$ via $\gamma$. Define
\begin{equation}\label{eq:def-C}
C(u,r):=u\int_r^{\infty}h(z)e^{-u (z-r)}\mathrm{d} z.
\end{equation}
We have $AC(x,Y_r)=C(\gamma,r)$.
By \eqref{eq:def-C} and a change of variables, we also have
\[
C(u,r)=u \int_0^{\infty}h(r+z)e^{-u z}\mathrm{d}z.
\]
Then for $r>0$,
\begin{align*}
C(\gamma,r)&=\gamma\int_0^{\infty}h(r+z)e^{-\gamma z}\mathrm{d}z> \gamma\int_0^{\infty}h(z)e^{-\gamma z}\mathrm{d}z=C(\gamma,0),
\end{align*}
where 
the inequality follows from the fact that $h(x)$ is strictly increasing in $[0,\infty)$ in Assumption~\ref{ass:1} $(i)$.
Thus, it is not optimal to choose $r>0$. Taking partial derivative of $C(\gamma,r)$ with respect to $r\leq0$ in \eqref{eq:def-C}, we have $\partial_r C(\gamma,r)=\gamma\left(C(\gamma,r)-h(r)\right)$.
We will show that the optimal reflection level $r^*$ is the solution to the optimality equation
\begin{equation}\label{eq:optimal}
    C(\gamma,r)=h(r).
\end{equation}
Assumption~\ref{ass:1} ensures this equation has a unique, finite, negative solution $r^*$---specifically, it implies $C(\gamma,r)$ is unimodal in $r$ for $u=\gamma$ (see Lemma~\ref{lemma:exist-unique-finite-solution}).

Moreover, under the Assumption~\ref{ass:1} $(iii)$, for fixed parameter $\gamma\geq\underline{\gamma}$, we have $0<b< \gamma/2$. Theorem $1$ in~\citet{yang2020optimality} shows that the optimal reflection control $Y_{r^*}(t)$ is optimal over all admissible controls---that is, $AC(x,Y_{r^*})\leq AC(x,Y)$ for any admissible control $Y\in\Pi$ and from any initial state $x$.

\subsection{Data-Driven Adaptive Reflection Control}\label{sec:2.3}
Since the optimal policy is reflection control with reflecting level $r^*$ under Assumption~\ref{ass:1},  the regret in \eqref{eq:regret-1} becomes
\begin{equation}\label{eq:regret-2}
R(x,T)=\mathbb{E}_x\left[\int_0^T \left(h(Z(t))-h(Z_{r^*}(t))\right)\mathrm{d} t\right].
\end{equation}

In our learning framework, we aim to learn the optimal reflecting level $r^*$ characterized in the full-information setting. Since the optimality equation \eqref{eq:optimal} only depends on the composite parameter $\gamma$, the optimal reflecting level $r^*$ is determined by $\gamma$. 
Hence, we design learning algorithms for the parameter $\gamma$, rather than to identify $\theta$ and $\sigma$. This motivates a plug-in learning approach: estimate the effective parameter $\gamma$ from the controlled trajectory, and then compute the corresponding reflecting level via the full-information optimality equation based on the estimated parameter.
 
For $u\geq\underline{\gamma}$, define
\begin{align}\label{eq:F(u,r)}
F(u,r):=C(u,r)-h(r).
\end{align}
It is clear that $F(u,r)=0$ is precisely the optimality equation \eqref{eq:optimal} associated with the parameter $u$. Let $r(u)$ be the solution to this equation; then $r^*=r(\gamma)$.

Since $Z_r$ is ergodic with stationary distribution $\pi_r$, we have the following lemma, which enables us to construct an estimator for $\gamma$.
\begin{lemma}\label{lemma:ergodicity}
    For any $r\leq 0$,
    \[
    \lim_{T\to\infty}\frac{1}{T}\int_0^T \left(Z_r(t)-r\right)\mathrm{d}t=\int_0^{\infty}z\pi_r(\mathrm{d}z)=\frac{1}{\gamma}\quad\text{a.s.}
    \]
\end{lemma}

A natural idea leads to the learn-then-optimize (LTO) algorithm (see Algorithm~\ref{alg:1}). In the first step, we operate the system under a fixed reflection policy (with reflecting level $0$) for a learning period. Then by Lemma~\ref{lemma:ergodicity}, we construct an estimator $\hat{\gamma}$ for the true parameter $\gamma$ from the sample-path average of the resulting state process.
In the second step, we plug the estimator $\hat{\gamma}$ into the optimality equation
$F(\hat{\gamma},r)=0$ and obtain the corresponding optimal reflection level $r(\hat{\gamma})$. In the third step, we implement this learned policy (with reflecting level $r(\hat{\gamma})$) to minimize the cost over the remaining time horizon. We refer to $\tau$ as the length of the learning interval and denote by $\left\{Z^{\mathrm{LTO}}(t):t\geq0\right\}$ the state dynamics when the process is controlled by the LTO algorithm. Let $R^{\mathrm{LTO}}(x,T)$ be the regret of the LTO algorithm at any time $T\geq 0$, as defined in \eqref{eq:regret-2} with $Z(t)$ being replaced by $Z^{\mathrm{LTO}}(t)$.

\begin{algorithm}[htbp]
	\caption{Learn-then-optimize algorithm}
    \label{alg:1}
    \begin{algorithmic}[1] 
        \STATE Choose $\tau\in(0,\infty)$. Control the state with reflection level $0$ on $[0,\tau)$.
        \STATE At $\tau$, compute the estimator 
        \begin{equation}\label{eq:hat-gamma}
        \hat{\gamma}:=\left(\frac{1}{\tau}\int_0^{\tau}Z^{\mathrm{LTO}}(t)\mathrm{d} t\right)^{-1}\vee\underline{\gamma}.
        \end{equation}
        \STATE Compute $\hat{r}_{\tau}:=r(\hat{\gamma})$ by solving the optimality equation $F(\hat{\gamma},r)=0$.
        \STATE Control the state with reflection level $\hat{r}_{\tau}$.
    \end{algorithmic}
\end{algorithm}

\begin{theorem}\label{thm:LTO4generalh}
    There exists a  constant $C_1>0$, depending on primitive parameters of $h$, $\theta$, $\sigma$, $\underline{\gamma}$, and $x\geq0$, such that for all $T\geq 1$ the regret of the learn-then-optimize algorithm with learning interval length $\tau=\sqrt{T}$ satisfies
    \[
    R^{\mathrm{LTO}}(x,T)\leq C_1\left(1+\sqrt{T}\right).
    \]
\end{theorem}
The proof of Theorem~\ref{thm:LTO4generalh} is provided in Section~\ref{sec:6}. In fact, the initial exploration policy is not critical, as any reflection control can generate data to estimate $\gamma$ by Lemma~\ref{lemma:ergodicity}.
Theorem~\ref{thm:LTO4generalh} reveals that the exploration-exploitation tradeoff is mediated by the choice of the learning period $\tau$: a longer learning period improves the accuracy of $\hat{\gamma}$, but reduces the time available for exploiting $r(\hat{\gamma})$. In the context of this one-time estimation learning scheme for $\gamma$, the regret bound is optimized when the learning period is of order $\sqrt{T}$.

The LTO algorithm serves as a transparent benchmark for understanding the exploration-exploitation tradeoff. Since the exploration policy is not critical, the state process controlled by the reflecting level $r(\hat{\gamma})$ remains informative and can be used for further learning. Building on this insight, we next consider the adaptive-updating algorithm (see Algorithm~\ref{alg:2}) that integrates exploration and exploitation simultaneously over time.

Rather than estimating $\gamma$ once and keeping the resulting reflecting level fixed, the adaptive-updating (AU) algorithm allows the controller to continuously estimate $\gamma$ and update the policy over doubling intervals. This design does not depend on the horizon $T$ and achieves a better regret rate.

\begin{algorithm}[htbp]
	\caption{Adaptive-updating algorithm}
    \label{alg:2}
    \begin{algorithmic}[1] 
        \STATE Initialize the values $k=0$, $\hat{r}_0=0,\tau_0=0$.
        \STATE Set $\tau_{k}=2^{k}-1$.
        \STATE Control the state with reflection level $\hat{r}_k$ on $[\tau_k,\tau_{k+1})$.
        \STATE At $\tau_{k+1}$ compute the estimator 
        \begin{align}\label{eq:hat-gamma-k}
            \hat{\gamma}_{k+1}:=\left(\frac{1}{\Delta\tau_k}\int_{\tau_k}^{\tau_{k+1}}\left(Z^{\mathrm{AU}}(s)-\hat{r}_k\right)\mathrm{d}s\right)^{-1}\vee\underline{\gamma},
        \end{align}
        where $\Delta \tau_k=\tau_{k+1}-\tau_k$.
        \STATE Compute $\hat{r}_{k+1}:=r(\hat{\gamma}_{k+1})$ by solving the optimality equation $F(\hat{\gamma}_{k+1},r)=0$.
        \STATE $k\leftarrow k+1$ and go to Step $2$.
    \end{algorithmic}
\end{algorithm}

Noting that only the data from the previous interval is used in \eqref{eq:hat-gamma-k}, we also propose the full-history adaptive-updating (AU-FH) algorithm (see Algorithm~\ref{alg:3}) that leverages all historical data to construct the estimator. Theoretically, this algorithm still achieves a regret of order $O(\log T)$, and it demonstrates better performance in numerical experiments. Let $\left\{Z^{\mathrm{AU}}(t):t\geq0\right\}$ and $\left\{Z^{\mathrm{AU\text{-}FH}}(t):t\geq0\right\}$ denote the state processes when applying the AU and AU-FH algorithms, and $R^{\mathrm{AU}}(x,T)$ and $R^{\mathrm{AU\text{-}FH}}(x,T)$ be the regrets of the AU and AU-FH algorithms at any time $T\geq 0$ as defined in \eqref{eq:regret-2}, with $Z(t)$ being replaced by $Z^{\mathrm{AU}}(t)$ and $Z^{\mathrm{AU\text{-}FH}}(t)$, respectively.

\begin{algorithm}[htbp]
	\caption{Full-history adaptive-updating algorithm}
    \label{alg:3}
    \begin{algorithmic}[1] 
        \STATE Initialize the values $k=0$, $\tilde{r}_0=0,\tau_0=0$.
        \STATE Set $\tau_{k}=2^{k}-1$.
        \STATE Control the state with reflection level $\tilde{r}_k$ on $[\tau_k,\tau_{k+1})$.
        \STATE At $\tau_{k+1}$ compute the estimator 
        \begin{align}\label{eq:tilde-gamma-k}
            \tilde{\gamma}_{k+1}:=\left(\frac{1}{\tau_{k+1}}\sum_{i=0}^k\int_{\tau_i}^{\tau_{i+1}}\left(Z^{\mathrm{AU\text{-}FH}}(s)-\tilde{r}_i\right)\mathrm{d}s\right)^{-1}\vee\underline{\gamma},
        \end{align}
        where $\Delta \tau_k=\tau_{k+1}-\tau_k$.
        \STATE Compute $\tilde{r}_{k+1}:=r(\tilde{\gamma}_{k+1})$ by solving the optimality equation $F(\tilde{\gamma}_{k+1},r)=0$.
        \STATE $k\leftarrow k+1$ and go to Step $2$.
    \end{algorithmic}
\end{algorithm}
\begin{theorem}\label{thm:AU4generalh}
    There exists a constants $C_2,C_3>0$, depending on primitive parameters of $h$, $\theta$, $\sigma$, $\underline{\gamma}$, and $x\geq0$, such that the expected regrets of the $\mathrm{AU}$ and $\mathrm{AU}\text{-}\mathrm{FH}$ algorithms satisfy for all $T\geq 1$,
    \[
    R^{\mathrm{AU}}(x,T)\leq C_2\left(1+\log{T}\right),\quad R^{\mathrm{AU\text{-}FH}}(x,T)\leq C_3\left(1+\log{T}\right).
    \]
\end{theorem}

\section{Sample-Path Construction and Regret Decomposition}\label{sec:3}
In this subsection, we use the reflected SDE in \eqref{eq:SDE-startfrom-s} to construct the sample paths for the three algorithms. This sample path representation allows us to decompose the regret over the learning and updating interval into three parts: exploration regret, transient regret and learning regret (defined as follows).
\subsection{Learn-then-optimize  Algorithm}\label{sec:3.1}
By Algorithm~\ref{alg:1}, we control the state with reflection level $0$ over $[0,\tau)$, so we define $Z^{\mathrm{LTO}}(t):=Z_0(t)$ on $[0,\tau)$. Since $Z_0$ is a standard reflected Brownian motion (RBM), the occupation time formula gives
\[
\int_0^\tau \mathbf{1}_{\{Z_0(s)=0\}}\mathrm{d} \langle Z_0\rangle_s=\int_{\mathbb{R}}\mathbf{1}_{\{a=0\}}L_{\tau}^a(Z_0)\mathrm{d}a=0,
\]
where $\langle Z_0\rangle$ denotes the predictable quadratic variation and
\[
L_{\tau}^a(Z_0):=\lim_{\varepsilon\downarrow0}\frac{1}{2\varepsilon}\int_0^\tau \mathbf{1}_{\{\lvert Z_0(s)-a\rvert<\varepsilon\}}\mathrm{d} \langle Z_0\rangle_s
\] denotes the local time at level $a$.
Since $\mathrm{d}\langle Z_0\rangle_s=\sigma^2\mathrm{d}s$, substituting into the above equality yields $\int_0^\tau \mathbf{1}_{\{Z_0(s)=0\}}\mathrm{d}s=0$.
This implies that $Z_0(s)>0$ for almost every $s\in[0,\tau]$,
\[
\int_0^{\tau}Z^{\mathrm{LTO}}(t)\mathrm{d}t=\int_0^{\tau}Z_0(t)\mathrm{d}t>0.
\]
Thus the estimator $\hat{\gamma}$ defined in \eqref{eq:hat-gamma} is well-defined. Then the learned reflection level $\hat{r}_{\tau}=r(\hat{\gamma})<0\leq Z_0(\tau)$ since the optimal equation has a unique negative solution. Let $X:=Z_0(\tau)$; we therefore define $Z^{\mathrm{LTO}}(t):=Z_{\hat{r}_{\tau}}^{\tau, X}(t)$ on $[\tau,\infty)$.

The regret for LTO algorithm can be decomposed as follows: 

\begin{align}\label{eq:regret-decomposition-LTO}
    R^{\mathrm{LTO}}(x,T)&=\mathbb{E}_x\left[\int_0^T \left(h\left(Z^{\mathrm{LTO}}(t)\right)-h\left(Z_{r^*}(t)\right)\right)\mathrm{d}t \right]\nonumber\\
    &=\underbrace{\mathbb{E}_x\left[\int_0^{\tau} \left(h\left(Z_0(t)\right)-C(\gamma,0)\right)\mathrm{d}t \right]}_{=:R_{1}^{\mathrm{LTO}}}+\underbrace{\mathbb{E}_x\left[\int_0^{\tau} \left(C(\gamma,0)-C(\gamma,r^*)\right)\mathrm{d}t \right]}_{=:R_{2}^{\mathrm{LTO}}}\nonumber\\
    &\quad+\underbrace{\mathbb{E}_x\left[\int_{\tau}^T \left(h\left(Z_{\hat{r}_{\tau}}^{\tau, X}(t)\right)-C(\gamma,\hat{r}_{\tau})\right)\mathrm{d}t \right]}_{=:R_{3}^{\mathrm{LTO}}}+\underbrace{\mathbb{E}_x\left[\int_{\tau}^T \left(C(\gamma,\hat{r}_{\tau})-C(\gamma,r^*)\right)\mathrm{d}t \right]}_{=:R_{4}^{\mathrm{LTO}}}\nonumber\\
    &\quad+\underbrace{\mathbb{E}_x\left[\int_0^T \left(C(\gamma,r^*)-h\left(Z_{r^*}(t)\right)\right)\mathrm{d}t \right]}_{=:R_5^{\mathrm{LTO}}}.
\end{align}
$R_{1}^{\mathrm{LTO}}$ denotes the difference between the expected cost of the process using reflection control at level $0$ over $[0,\tau)$ and the corresponding stationary cost up to time $\tau$, which is referred to as the transient regret. $R_{2}^{\mathrm{LTO}}$ is the stationary cost gap between $0$ reflection control and optimal reflection control, which is referred to as the exploration regret. $R_{3}^{\mathrm{LTO}}$ is the transient regret associated with reflection level $\hat{r}_{\tau}$ on $[\tau,T)$. $R_{4}^{\mathrm{LTO}}$ denotes the stationary cost gap between the reflection control induced by $\hat{r}_{\tau}$ and $r^*$. We refer to $R_{4}^{\mathrm{LTO}}$ as the learning regret, since it captures the regret caused by implementing the learned policy instead of the optimal policy. $R_5^{\mathrm{LTO}}$ is the transient regret for optimal reflection level $r^*$.

\subsection{Adaptive-updating  Algorithm}\label{sec:3.2}
Define the updating grid by
\[
\tau_0=0,\quad\Delta\tau_k=2^{k}\quad k\in\mathbb{N}. 
\]
Then
\[
\tau_{k+1}=\tau_k+\Delta\tau_k=2^{k+1}-1\quad k\in\mathbb{N}.
\]
Let $I_k=[\tau_k,\tau_{k+1})$ denote the $k$-th interval, and initialize $\hat r_0=0$ and $X_0=x$. On $I_0$, we set $Z^{\mathrm{AU}}(t):=Z_0(t)$. As shown earlier, $\hat{\gamma}_1$ is well-defined and the learned reflection level $\hat{r}_1=r(\hat{\gamma}_1)<0\leq Z_0(\tau_1)$. So we let $X_1:=Z_0(\tau_1)$. 

We now recursively construct the sample path for the AU algorithm: given $\hat{r}_k,X_k$ for $k\geq1$, define 
\begin{equation}\label{eq:sample-path-AU}
Z^{\mathrm{AU}}(t):=Z^{\tau_k,X_k}_{\hat r_k}(t) \quad t\in I_k.
\end{equation}
Since on $I_k$, the state evolves as a RBM with reflection level $\hat{r}_k$ and initial state $X_k$, the estimator $\hat{\gamma}_{k+1}$ in \eqref{eq:hat-gamma-k} is well-defined and $\hat{r}_{k+1}<0$ is the unique negative solution of optimality equation $F(\hat{\gamma}_{k+1},r)=0$. It is possible that the new reflection level $\hat{r}_{k+1}$ is higher than the left limit of the current state $Z^{\mathrm{AU}}(\tau_{k+1}-)$; we therefore define the post-update state by
\begin{equation}\label{eq:poststate}
    X_{k+1}:=Z^{\mathrm{AU}}(\tau_{k+1}-)\vee \hat r_{k+1},
\end{equation}
and the control (denoted by $Y^{AU}(t)$ for AU algorithm) has an instantaneous jump at $\tau_{k+1}$:
\[
\Delta{Y}^{\mathrm{AU}}(\tau_{k+1})=\left(\hat{r}_{k+1}-Z^{\mathrm{AU}}(\tau_{k+1}-)\right)^+,
\]
where $\Delta{Y}^{\mathrm{AU}}(t):=Y^{\mathrm{AU}}(t)-Y^{\mathrm{AU}}(t-)$.
Since there is no impulse control cost in our setting, this jump only ensures that the next interval starts from an admissible state $X_{k+1}\geq \hat r_{k+1}$. This completes the recursion for $\hat{r}_{k+1}$ and $X_{k+1}$.
\begin{remark}
    From the sample path construction of the $\mathrm{AU}$ algorithm above, we see that the control is not a pure reflection control: it acts as a reflection control within each interval $I_k$, but may involve an impulse control at the interval switching points $\tau_{k+1}$.
\end{remark}

For a fixed $T\geq1$, let $N:=\max\{k:2^{k}-1\leq T\}=\lfloor \log_2 (T+1)\rfloor$, then $T<2^{N+1}-1$ and
\[
\tau_N=\sum_{n=0}^{N-1} \Delta \tau_n=\sum_{n=0}^{N-1} 2^n=2^N-1\leq T<2^{N+1}-1=\sum_{n=0}^{N} 2^n=\tau_{N+1}.
\]
Thus, we have

\begin{align*}
    R^{\mathrm{AU}}(x,T)&=\mathbb{E}_x\left[\int_0^T \left(h\left(Z^{\mathrm{AU}}(t)\right)-h\left(Z_{r^*}(t)\right)\right)\mathrm{d}t \right]\\
    &\leq\underbrace{\mathbb{E}_x\left[\sum_{k=0}^{N-1}\int_{\tau_k}^{\tau_{k+1}}  h\left(Z_{\hat{r}_k}^{\tau_k,X_k}(t)\right)-C(\gamma,\hat{r}_k)\mathrm{d}t\right]+\mathbb{E}_x\left[\int_{\tau_{N}}^T  h\left(Z_{\hat{r}_N}^{\tau_N,X_N}(t)\right)-C(\gamma,\hat{r}_{N})\mathrm{d}t\right]}_{=:R_1^{\mathrm{AU}}}\\
    &\quad+\underbrace{\mathbb{E}_x\left[\int_{\tau_0}^{\tau_{1}} \left\lvert C(\gamma,\hat{r}_0)-C(\gamma,r^*)\right\rvert\mathrm{d}t\right]}_{=:R_{2}^{\mathrm{AU}}}+\underbrace{\mathbb{E}_x\left[\sum_{k=1}^{N}\int_{\tau_k}^{\tau_{k+1}} \left\lvert C(\gamma,\hat{r}_k)-C(\gamma,r^*)\right\rvert\mathrm{d}t\right]}_{=:R_{3}^{\mathrm{AU}}}\\
    &\quad+\underbrace{\mathbb{E}_x\left[\int_0^T \left( C(\gamma,r^*)-h\left(Z_{r^*}(t)\right)\right)\mathrm{d}t \right]}_{=:R_4^{\mathrm{AU}}},
\end{align*}
where the inequality follows from $C(\gamma,\hat{r}_N)\geq C(\gamma,r^*)$ and $\tau_N\leq T<\tau_{N+1}$.
$R_1^{\mathrm{AU}}$ is the combination of the segment transient regret for policy $\hat{r}_k$ on $[\tau_k,\tau_{k+1})$ and residual transient regret on $[\tau_{N},T)$. 
$R_{2}^{\mathrm{AU}}$ is the exploration regret incurred on the initial interval, and $R_{3}^{\mathrm{AU}}$ is the cumulative learning regret over the subsequent updating intervals.
Finally, $R_4^{\mathrm{AU}}$ is the transient regret for optimal policy $r^*$.

\subsection{Full-history Adaptive-updating Algorithm}

The sample path construction of the AU-FH algorithm is almost identical to that of the AU algorithm, except for two differences: $\hat{r}_k$ is replaced with $\tilde{r}_k$ in each update, and the post-update state is denoted by $\tilde{X}_k$.
The corresponding regret decomposition is also completely analogous: we denote by $R_1^{\mathrm{AU\text{-}FH}}$, $R_2^{\mathrm{AU\text{-}FH}}$, $R_3^{\mathrm{AU\text{-}FH}}$, and $R_4^{\mathrm{AU\text{-}FH}}$ the terms obtained by replacing $\hat{r}_k, X_k$ with $\tilde{r}_k,\tilde{X}_k$ in $R_1^{\mathrm{AU}}$, $R_2^{\mathrm{AU}}$, $R_3^{\mathrm{AU}}$, and $R_4^{\mathrm{AU}}$, respectively.

\section{Transient Regret Analysis}\label{sec:4}
From our regret decomposition, we observe that $R_{1}^{\mathrm{LTO}}$, $R_{3}^{\mathrm{LTO}}$, $R_5^{\mathrm{LTO}}$, $R_1^{\mathrm{AU}}$, $R_4^{\mathrm{AU}}$, $R_1^{\mathrm{AU\text{-}FH}}$, and $R_4^{\mathrm{AU\text{-}FH}}$ are all transient regrets corresponding to specific reflected Brownian motions on a finite interval.
We first explain how the Foster-Lyapunov inequality for an appropriately chosen Lyapunov function can be used to bound the transient regret. We then establish the explicit rate of exponential convergence to the stationary distribution for RBMs, which provides the quantitative finite-time estimates required in the subsequent transient regret analysis.

\subsection{Explicit Exponential Convergence Rate for RBM }\label{sec:4.1}
For each fixed reflection level $r\leq0$ and $x\geq r$, the reflected Brownian motion $Z_r=\{Z_r(t):t\geq0\}$ defined in \eqref{eq:SDE-startfrom-s} (for $s=0$ and $y=x$) is a time-homogeneous continuous-time Markov process on $[r,\infty)$ starting from $x$. Let $\mathbb{P}_t^{(r)}(x,A):=\mathbb{P}\left(Z_r(t)\in A\right)$ for $A\in\mathcal{B}\left([r,\infty)\right)$ be the transition kernel of $Z_r$, where $\mathcal{B}\left([r,\infty)\right)$ is the Borel $\sigma$-field on $[r,\infty)$. The associated transition semigroup, still denoted by $\mathbb{P}_t^{(r)}$,  acts on any measurable function $f$ by
\begin{align}\label{eq:def-P}
\mathbb{P}_t^{(r)}f(x):=\int_{[r,\infty)} f(y)\mathbb{P}_t^{(r)}(x,\mathrm{d}y)=\mathbb{E}_x\left[f(Z_r(t))\right],
\end{align}
whenever the expectation is well-defined. Recalling the definition of $\pi_r$ in \eqref{eq:stationary-Zr}, we also define for a measurable function $f$, 
$
\pi_r(f):=\int_{\mathbb{R}} f(z)\pi_r(\mathrm{d}z)=\int_{[r,\infty)} f(z)\pi_r(\mathrm{d}z).
$
Then we can bound the transient regret of $Z_r$ over time $T$ by the time integral of the deviation between the transition semigroup $\mathbb{P}_t^{(r)}$ and the stationary distribution $\pi_r$ (as evaluated on the holding cost $h$). This bounding relation is formally expressed as:
\begin{align}\label{eq:transient-rewrite}
\mathbb{E}_x\left[\int_0^T\left(h\left(Z_r(t)\right)-C(\gamma,r)\right)\mathrm{d}t\right]&=\int_0^T\left( \int_{[r,\infty)}h(y)\mathbb{P}_t^{(r)}(x,\mathrm{d}y)-\int_{[r,\infty)}h(y)\pi_r(\mathrm{d}y)\right)\mathrm{d}t\nonumber\\
&\leq\int_0^T\left\lvert \mathbb{P}_t^{(r)}h(x)-\pi_r(h) \right\rvert\mathrm{d}t.
\end{align}

For a continuous-time Markov process $\{X(t):t\geq0\}$ starting from $x$ with transition kernel $\mathbb{P}_t$ and generator $\mathscr{A}$. The standard Foster-Lyapunov method (see~\citet{down1995exponential}) states that if we can find a suitable Lyapunov function $V$ and verify the exponential drift condition
\begin{align}\label{eq:exponential-drift}
\mathscr{A}V\leq-c_1V+c_2\mathbf{1}_C
\end{align}
for a petite set $C$ and constants $c_1,c_2>0$, then the process has exponential ergodicity with stationary distribution $\pi$. And we can obtain an exponential ergodicity bound of the form
\begin{align}\label{eq:exponential-ergodicity}
\left\|\mathbb{P}_t(x,\cdot)-\pi\right\|_{V}\leq M^{\prime} V(x)e^{-\rho t}
\end{align}
for some constants $M^{\prime}<\infty$ and $\rho>0$. Here, for a measurable function $V: [r,\infty)\to[1,\infty)$ and a Borel measure $\nu$ on $[r,\infty)$, we define the $V$-norm by $\|\nu\|_V:=\sup_{\lvert f\rvert\leq V}\left\lvert \int_{[r,\infty)} f\mathrm{d}\nu\right\rvert$.

Thus, we aim to establish \eqref{eq:exponential-ergodicity} for the RBM and use it to bound \eqref{eq:transient-rewrite}.
For the RBM with reflection level $r$, the generator is given by
\begin{align}\label{eq:generator}
\Gamma f(x)=\theta f'(x)+\frac{\sigma^2}{2}f''(x),\quad x>r,
\end{align}
for test functions $f$ satisfying the Neumann boundary condition $f^{\prime}(r)=0$. A natural candidate Lyapunov function is the exponential function $V_r(z):=e^{\beta(z-r)},z\geq r$ for some parameter $\beta>0$. This choice is motivated by the fact that direct computation of $\Gamma V_r(z)$ yields
\[
\Gamma V_r(z)=\left(\theta\beta+\frac{\sigma^2}{2}\beta^2\right)V_r(z), 
\]
which closely resembles the exponential drift condition \eqref{eq:exponential-drift} if $\theta\beta+\sigma^2\beta^2/2<0$. However, we cannot directly apply the generator $\Gamma$ to this function---since $V_r^{\prime}(r)=\beta\neq0$, violating the Neumann boundary condition. Following~\citet{Sarantsev2016Explicit,sarantsev2017reflected}, we resolve this issue by working with a modified Lyapunov function. Specifically, instead of imposing the drift condition globally, we only require the supermartingale property up to the first hitting time of the reflecting boundary. Define
\begin{align}\label{eq:tau-r}
\tau_r:=\inf\{t\geq 0: Z_r(t)=r\}.
\end{align}
A nondecreasing function $V:[r,\infty)\to[1,\infty)$ is called a modified Lyapunov function with Lyapunov constant $k>0$ if
\[
V(Z_r(t\wedge \tau_r))+k\int_0^{t\wedge \tau_r}V(Z_r(s))\mathrm{d}s,\quad t\geq 0
\]
is a supermartingale for every initial state $x\geq r$. 
We say that the process $Z_r$ is stochastically ordered if, for every $r\leq x\leq y$, one can construct on a common probability space two copies denoted by $Z_r^{x}$ and $Z_r^{y}$ (starting from $x$ and $y$ respectively), such that
\[
Z_r^{x}(t)\leq Z_r^{y}(t),\quad t\geq 0,\quad \text{a.s.}
\]
Choose $\beta:=\underline{\gamma}/2\in(b,\underline{\gamma})$. In the following lemma, we show that $V_r(x)=e^{\underline{\gamma}(x-r)/2}$ is a modified Lyapunov function and prove some properties of $Z_r$.
\begin{lemma}\label{lemma:rbm-ergodicity}
    For every $r\leq0$, the following hold.
    \begin{enumerate}[$(i)$]
    \item $Z_r$ is stochastically ordered;
    \item $V_r(z)=e^{\underline{\gamma}(z-r)/2}$ is a modified Lyapunov function with Lyapunov constant $\sigma^2\underline{\gamma}(2\gamma-\underline{\gamma})/8$;
    \item $\int_{[r,\infty)}V_r\mathrm{d}\pi_r=\frac{\gamma}{\gamma-\underline{\gamma}/2}<\infty$.
    \end{enumerate}
\end{lemma}
The proof of Lemma~\ref{lemma:rbm-ergodicity} is in Appendix~\ref{app:A.1}.  By Theorem $4.1$ in~\citet{Sarantsev2016Explicit}, the above lemma implies the following explicit rate of exponential convergence to the stationary distribution for RBMs.
\begin{proposition}\label{prop:exp-convergence-rate-rbm}
    For each $r\leq0$ and $x\geq r$,
    \begin{align}\label{eq:ergodicity-convergence}
    \left\|\mathbb{P}_t^{(r)}(x,\cdot)-\pi_r\right\|_{V_r}=\left[\frac{\gamma}{\gamma-\underline{\gamma}/2}+e^{\underline{\gamma}(x-r)/2}\right]e^{-\sigma^2\underline{\gamma}(2\gamma-\underline{\gamma})t/8},\quad t\geq0.
    \end{align}
\end{proposition}
Notably, the exponential convergence rate is determined by the Lyapunov constant, which depends only on $\gamma,\underline{\gamma},\sigma^2$. Thus, this rate is uniform for all $r\leq0$, whereas the prefactor (i.e., the front multiplicative coefficient) depends on $r$ through the initial displacement $x-r$.
This uniformity is critical to the analysis of our algorithms: although the switching policies and states in the algorithms are stochastic and history-dependent, the exponential convergence rate uniform in $r$ can be applied in each fixed-policy interval to control the transient regret. What remains for us is to address the initial displacement arising from each updates in the algorithms to determine the value of the prefactor.

\subsection{Transient Regret Bound for LTO Algorithm}\label{sec:4.2}
In this subsection, we give the transient regret bound for $R_{1}^{\mathrm{LTO}}$, $R_{3}^{\mathrm{LTO}}$ and $R_5^{\mathrm{LTO}}$ in Algorithm~\ref{alg:1}. 

As will be shown later, the solution of optimality equation lies in $(-K_h,0)$ for some constant $K_h>0$ under Assumption~\ref{ass:1} (see Lemma~\ref{lemma:exist-unique-finite-solution}). Now, we first show that there exists a constant $C_h$ such that $h(x)\leq C_h V_r(x),x\geq -K_h$ for each $r\in[-K_h,0]$.
By Assumption~\ref{ass:1} $(iii)$, there exists a constant $a>0$ such that $h(x)\leq ae^{bx}<ae^{\underline{\gamma}x/2}$ for $x\geq0$. Since $h(x)$ is continuous on $[-K_h,0]$, there exists a constant $H_{K}>0$ such that $h(x)\leq H_K$ for $-K_h\leq x\leq0$. Thus there exists a constant $C_h=\max\{a,H_K\}>0$ such that $h(x)\leq C_hV_r(x)$ for each $r\in[-K_h,0]$.

At the switching point $\tau$ in LTO algorithm, the initial state $X=Z^{\mathrm{LTO}}(\tau)=Z_0(\tau)$ is a random variable. To bound the transient regret, we also need to establish the exponential moment bound for $Z_0$ and use the strong Markov property of RBM. The proofs of these two lemmas are provided in Appendix~\ref{app:A.2}.
\begin{lemma}\label{lemma:exp-moment-bound-Z0}
For $x\geq 0$,
    \begin{align}\label{eq:exponential-moment-bound}
        \sup_{t\geq0}\mathbb{E}_x\left[e^{\underline{\gamma} Z_0(t)/2}\right]\leq 4+e^{\underline{\gamma}x/2}.
    \end{align}
\end{lemma}

\begin{lemma}\label{lemma:Markov-property-LTO}
$X,\hat{r}_{\tau}$ are $\mathcal{F}_{\tau}$-measurable random variables.
Moreover,
\[
\mathbb E_x\left[\int_{\tau}^{T}h\left(Z^{\tau,X}_{\hat{r}_{\tau}}(t)\right)\mathrm{d}t\Bigg|\mathcal F_{\tau}\right]=\int_{0}^{T-\tau}\mathbb{E}_X^{(\hat{r}_{\tau})}\left[h\left(Z_{\hat{r}_{\tau}}^{0,X}(s)\right)\right]\mathrm{d}s\quad a.s.,
\]
where $\mathbb{E}_X^{(\hat{r}_{\tau})}$ is the expectation conditioning on the random initial state $X$ and random reflection level $\hat{r}_{\tau}$.
\end{lemma}

\begin{proposition}[Transient regret bounds in LTO algorithm]\label{prop:transient-regret-bound-LTO}
For $x\geq 0$, we have
\begin{subequations}
\begin{align}
R_{1}^{\mathrm{LTO}}&\leq\frac{8C_h\left(2+e^{\underline{\gamma}x/2}\right)}{\sigma^2\underline{\gamma}(2\gamma-\underline{\gamma})};\label{eq:transient-regret-bound-LTO-1}\\
R_{3}^{\mathrm{LTO}}&\leq\frac{8C_h\left(2+e^{\underline{\gamma}K_h/2}\left(4+e^{\underline{\gamma}x/2}\right)\right)}{\sigma^2\underline{\gamma}(2\gamma-\underline{\gamma})};\label{eq:transient-regret-bound-LTO-2}\\
R_{5}^{\mathrm{LTO}}&\leq \frac{8C_h\left(2+e^{\underline{\gamma}(x+K_h)/2}\right)}{\sigma^2\underline{\gamma}(2\gamma-\underline{\gamma})}\label{eq:transient-regret-bound-LTO-3}.
\end{align}
\end{subequations}
\end{proposition}

\begin{proof}
    Proof of \eqref{eq:transient-regret-bound-LTO-1}: Since $h(x)\leq C_hV_0(x)$, we have
    \begin{align*}
        \left\lvert \mathbb{E}_x[h(Z_0(t))]-C(\gamma,0)\right\rvert&=\left\lvert \int_{[0,\infty)} h(z)\left(\mathbb{P}_t^{(0)}(x,\mathrm{d}z)-\pi_0(\mathrm{d}z)\right)\right\rvert\\
        &\leq C_h \left\|\mathbb{P}_t^{(0)}(x,\cdot)-\pi_0\right\|_{V_0}\\
        &\leq C_h\left[\frac{\gamma}{\gamma-\underline{\gamma}/2}+e^{\underline{\gamma}x/2}\right]e^{-\sigma^2\underline{\gamma}(2\gamma-\underline{\gamma})t/8}\\
        &\leq C_h\left[2+e^{\underline{\gamma}x/2}\right]e^{-\sigma^2\underline{\gamma}(2\gamma-\underline{\gamma})t/8},
    \end{align*}
    where the first inequality follows from the definition of $V$-norm, the second inequality is a consequence of \eqref{eq:ergodicity-convergence} for $r=0$, and the last inequality holds since $\frac{\gamma}{\gamma-\underline{\gamma}/2}\leq 2$. Integrating from $0$ to $\tau$, we obtain
    \begin{align*}
    R_{1}^{\mathrm{LTO}}&=\mathbb{E}_x\left[\int_0^{\tau} \left(h\left(Z_0(t)\right)-C(\gamma,0)\right)\mathrm{d}t \right] \leq \int_0^{\tau} \left\lvert\mathbb{E}_x\left[\left(h\left(Z_0(t)\right)\right]-C(\gamma,0)\right)\right\rvert\mathrm{d}t \\
    &\leq \int_0^{\tau}C_h\left[2+e^{\underline{\gamma}x/2}\right]e^{-\sigma^2\underline{\gamma}(2\gamma-\underline{\gamma})t/8}\mathrm{d}t\\
    &=\frac{8C_h\left(2+e^{\underline{\gamma}x/2}\right)}{\sigma^2\underline{\gamma}(2\gamma-\underline{\gamma})}\left(1-e^{-\sigma^2\underline{\gamma}(2\gamma-\underline{\gamma})\tau/8}\right)\leq\frac{8C_h\left(2+e^{\underline{\gamma}x/2}\right)}{\sigma^2\underline{\gamma}(2\gamma-\underline{\gamma})}.
    \end{align*}

    Proof of \eqref{eq:transient-regret-bound-LTO-2}: By the tower property of conditional expectation and Lemma~\ref{lemma:Markov-property-LTO}, we have
    \begin{align*}
    R_{3}^{\mathrm{LTO}}&=\mathbb{E}_x\left[\int_{\tau}^T \left(h\left(Z_{\hat{r}_{\tau}}^{\tau, X}(t)\right)-C(\gamma,\hat{r}_{\tau})\right)\mathrm{d}t \right]\\
    &=\mathbb{E}_x\left[\mathbb{E}_x\left[\int_{\tau}^T \left(h\left(Z_{\hat{r}_{\tau}}^{\tau, X}(t)\right)-C(\gamma,\hat{r}_{\tau})\right)\mathrm{d}t\bigg|\mathcal{F}_{\tau} \right]\right]\\
    &=\mathbb{E}_x\left[\int_{0}^{T-\tau}\left(\mathbb{E}_X^{(\hat{r}_{\tau})}\left[h\left(Z_{\hat{r}_{\tau}}^{0,X}(s)\right)\right]-C(\gamma,\hat{r}_{\tau})\right)\mathrm{d}s \right].
    \end{align*}
    The above argument in $(i)$, replacing $0$ and $x$ with $\hat{r}_{\tau}$ and $X$, implies
    \begin{align*}
    R_{3}^{\mathrm{LTO}}&\leq\mathbb{E}_x\left[\int_0^{T-\tau} C_h\left[2+e^{\underline{\gamma}(X+K_h)/2}\right]e^{-\sigma^2\underline{\gamma}(2\gamma-\underline{\gamma})t/8}\mathrm{d}t\right]\\
    &=\frac{8C_h\left(2+e^{\underline{\gamma}K_h/2}\mathbb{E}_x\left[e^{\underline{\gamma}Z_0(\tau)/2}\right]\right)}{\sigma^2\underline{\gamma}(2\gamma-\underline{\gamma})}\left(1-e^{-\sigma^2\underline{\gamma}(2\gamma-\underline{\gamma})(T-\tau)/8}\right).
    \end{align*}
    Then by Lemma~\ref{lemma:exp-moment-bound-Z0}, we have
    \begin{align*}
    R_{3}^{\mathrm{LTO}}&\leq\frac{8C_h\left(2+e^{\underline{\gamma}K_h/2}\left(4+e^{\underline{\gamma}x/2}\right)\right)}{\sigma^2\underline{\gamma}(2\gamma-\underline{\gamma})}\left(1-e^{-\sigma^2\underline{\gamma}(2\gamma-\underline{\gamma})(T-\tau)/8}\right),
    \end{align*}
    which implies \eqref{eq:transient-regret-bound-LTO-2}.

    Proof of \eqref{eq:transient-regret-bound-LTO-3}: The proof follows similarly to part $(i)$: we only need to replace $0$ with $r^*$ (noting that $r^*\in[-K_h,0]$) and adjust the integral to $[0,T]$. We have
    \begin{align*}
        R_{5}^{\mathrm{LTO}}&=\mathbb{E}_x\left[\int_0^T \left(C(\gamma,r^*)-h\left(Z_{r^*}(t)\right)\right)\mathrm{d}t \right]\leq \int_0^T \left\lvert\mathbb{E}_x\left[\left(h\left(Z_{r^*}(t)\right)\right]-C(\gamma,r^*)\right)\right\rvert\mathrm{d}t \\
        &\leq \int_0^TC_h\left[2+e^{\underline{\gamma}(x+K_h)/2}\right]e^{-\sigma^2\underline{\gamma}(2\gamma-\underline{\gamma})t/8}\mathrm{d}t =\frac{8C_h\left(2+e^{\underline{\gamma}(x+K_h)/2}\right)}{\sigma^2\underline{\gamma}(2\gamma-\underline{\gamma})}\left(1-e^{-\sigma^2\underline{\gamma}(2\gamma-\underline{\gamma})T/8}\right),
    \end{align*}    
    which implies \eqref{eq:transient-regret-bound-LTO-3}.
\end{proof}

\subsection{Moving Boundary Skorokhod Mapping}\label{sec:4.3}
From the above transient regret analysis for LTO algorithm, the random initial state at time $\tau$ is exactly $Z_0(\tau)$, thus we can apply Lemma~\ref{lemma:exp-moment-bound-Z0} to obtain the exponential moment bound. However, the AU algorithm involves random switching states, which requires additional handling. In the following, we leverage the monotonicity of moving boundary Skorokhod mapping to bound the sample path of the AU algorithm from above by $Z_0$.

An important observation is that the sample path construction in \eqref{eq:sample-path-AU} and \eqref{eq:poststate} can be viewed as a moving boundary Skorokhod mapping. Equivalently, defining the piecewise-constant boundary
\[
\hat{r}(t):=\sum_{k=0}^{\infty} \hat{r}_k\mathbf{1}_{[\tau_k,\tau_{k+1})}(t),
\]
the AU sample path $\{Z^{\mathrm{AU}}:t\geq0\}$ and the corresponding control process $\{Y^{\mathrm{AU}}(t):t\geq0\}$ admit the Skorokhod representation $Z_{\hat{r}}(t) = x + \theta t + \sigma B(t) + Y_{\hat{r}}(t)$, where $Y_{\hat{r}}$ is the minimal nondecreasing càdlàg (right-continuous with left limits) process such that $Z_{\hat{r}}(t) \geq \hat{r}(t),t \geq 0$ and is given explicitly by $Y_{\hat{r}}(t) = \sup_{0 \leq s \leq t} \left(\hat r(s) - x - \theta s - \sigma B(s)\right)^+$.
The following lemma shows the monotonicity of the moving boundary Skorokhod mapping. Since the reflection control we consider all have reflecting level less than or equal to zero, we use $(Y_i,Z_i)$ for $i=1,2$ here to denote the new definitions.
\begin{lemma}\label{lemma:moving-boundary-Skorokhod}
    Let $b_1,b_2:[0,\infty)\to\mathbb{R}$ be càdlàg functions with $b_1(t)\leq b_2(t),t\geq0$.
    For $x\geq b_2(0)$ and $i=1,2$, define
    \[
    Y_i(t):=\sup_{0\leq s\leq t}\left(b_i(s)-x-\theta s-\sigma B(s)\right)^+,\quad Z_i(t):=x+\theta t+\sigma B(t)+Y_i(t).
    \]
    Then for all $t\geq0$,
    \[
    Y_1(t)\leq Y_2(t),\quad Z_1(t)\leq Z_2(t)\quad \text{a.s.}
    \]
\end{lemma}
Using Lemma~\ref{lemma:moving-boundary-Skorokhod}, we directly have the following corollary, which bounds the random states in AU algorithm by $Z_0(\tau_k)$.

\begin{corollary}\label{Cor:1}
    For $x\geq0$, we have $Z_{\hat{r}}(t)\leq Z_0(t),t\geq0$ almost surely.
    Consequently, for $k\geq 1$, the post-update state defined in \eqref{eq:poststate} satisfies
    \[
    X_k\leq Z_0(\tau_k),\quad \text{a.s.}
    \]
\end{corollary}
The proofs of Lemma~\ref{lemma:moving-boundary-Skorokhod} and Corollary~\ref{Cor:1} are given in Appendix~\ref{app:A.3}.
\begin{remark}\label{remark:AU-FH}
    For $\mathrm{AU}$-$\mathrm{FH}$ algorithm, just as with the $\mathrm{AU}$ algorithm, we have $\tilde{X}_k\leq Z_0(\tau_k)$ a.s.
\end{remark}

\subsection{Transient Regret Bound for AU and AU-FH} Algorithms\label{sec:4.4}
Similar to Lemma~\ref{lemma:Markov-property-LTO}, we have the following lemma.
\begin{lemma}\label{lemma:Markov-property-AU}
$X_k,\hat{r}_k$ are $\mathcal{F}_{\tau_k}$-measurable random variables.
Moreover,
\[
\mathbb E_x\left[\int_{\tau_k}^{\tau_{k+1}}h\left(Z^{\tau_k,X_k}_{\hat{r}_k}(t)\right)\mathrm{d}t\Bigg|\mathcal F_{{\tau}_k}\right]=\int_{0}^{\Delta\tau_k}\mathbb{E}_{X_k}^{(\hat{r}_k)}\left[h\left(Z_{\hat{r}_k}^{0,X_k}(s)\right)\right]\mathrm{d}s\quad a.s.,
\]
where $\mathbb{E}_{X_k}^{(\hat{r}_k)}$ is the expectation conditioning on the random initial state $X_k$ and random reflection level $\hat{r}_k$.
\end{lemma}
Combining all the ingredients above, we have the following transient regret bound for AU and AU-FH algorithms.

\begin{proposition}[Transient regret bounds in AU and AU-FH algorithms]\label{prop:transient-regret-bound-AU}
For $x\geq 0$, the following hold.
\begin{subequations}
\begin{align}
\max\left\{R_{1}^{\mathrm{AU}},R_1^{\mathrm{AU\text{-}FH}}\right\}&\leq\frac{8C_h \left(2+e^{\underline{\gamma} K_h/2}\left(4+e^{\underline{\gamma}x/2}\right)\right)}{\sigma^2\underline{\gamma}(2\gamma-\underline{\gamma})}\left(\log_2 T +2\right);\label{eq:transient-regret-bound-AU-1}\\
\max\left\{R_{4}^{\mathrm{AU}}, R_4^{\mathrm{AU\text{-}FH}}\right\}&\leq \frac{8C_h\left(2+e^{\underline{\gamma}(x+K_h)/2}\right)}{\sigma^2\underline{\gamma}(2\gamma-\underline{\gamma})}.\label{eq:transient-regret-bound-AU-2}
\end{align}
\end{subequations}
\end{proposition}
\begin{proof}
   We only prove \eqref{eq:transient-regret-bound-AU-1} and \eqref{eq:transient-regret-bound-AU-2} for the AU algorithm; the corresponding results for the AU-FH algorithm are similar.
    Proof of \eqref{eq:transient-regret-bound-AU-1}: By the tower property and Lemma~\ref{lemma:Markov-property-AU}, we have for each $k$,
    \begin{align*}
        \mathbb{E}_x\left[\int_{\tau_k}^{\tau_{k+1}} \left( h\left( Z_{\hat{r}_k}^{\tau_k,X_k}(t)\right)-C(\gamma,\hat{r}_k)\right)\mathrm{d}t\right]&=\mathbb{E}_x\left[\mathbb{E}_x\left[\int_{\tau_k}^{\tau_{k+1}} \left( h\left( Z_{\hat{r}_k}^{\tau_k,X_k}(t)\right)-C(\gamma,\hat{r}_k)\right)\mathrm{d}t\bigg|\mathcal{F}_{\tau_k}\right]\right]\\
        &\leq\mathbb{E}_x\left[\int_{0}^{\Delta\tau_k}\left\lvert \mathbb{E}_{X_k}^{(\hat{r}_k)}\left[h\left(Z_{\hat{r}_k}^{0,X_k}(s)\right)\right]-C(\gamma,\hat{r}_k)\right\rvert\mathrm{d}s \right]\\
        &\leq \frac{8C_h\left(2+e^{\underline{\gamma}K_h/2}\mathbb{E}_x\left[e^{\underline{\gamma}X_k/2}\right]\right)}{\sigma^2\underline{\gamma}(2\gamma-\underline{\gamma})}\left(1-e^{-\sigma^2\underline{\gamma}(2\gamma-\underline{\gamma})\Delta\tau_k/8}\right)\\
        &\leq\frac{8C_h\left(2+e^{\underline{\gamma}K_h/2}\left(4+e^{\underline{\gamma}x/2}\right)\right)}{\sigma^2\underline{\gamma}(2\gamma-\underline{\gamma})}.
    \end{align*}
    And similarly, we obtain
    \begin{align*}
        \mathbb{E}\left[\int_{\tau_{N}}^T \left( h\left(Z_{\hat{r}_N}^{\tau_N,X_N}(t)\right)-C(\gamma,\hat{r}_{N})\right)\mathrm{d}t\right]\leq \frac{8C_h\left(2+e^{\underline{\gamma}K_h/2}\left(4+e^{\underline{\gamma}x/2}\right)\right)}{\sigma^2\underline{\gamma}(2\gamma-\underline{\gamma})}.
    \end{align*}
    Thus, by summing over $k=0$ to $N-1$ and adding the residual term, we have
    \begin{align*}
    R_{1}^{\mathrm{AU}}&=\mathbb{E}_x\left[\sum_{k=0}^{N-1}\int_{\tau_k}^{\tau_{k+1}} \left( h\left(Z_{\hat{r}_k}^{\tau_k,X_k}(t)\right)-C(\gamma,\hat{r}_k)\right)\mathrm{d}t\right]+\mathbb{E}\left[\int_{\tau_{N}}^T \left( h\left(Z_{\hat{r}_N}^{\tau_N,X_N}(t)\right)-C(\gamma,\hat{r}_{N})\right)\mathrm{d}t\right]\\
    &\leq \frac{8C_h\left(2+e^{\underline{\gamma}K_h/2}\left(4+e^{\underline{\gamma}x/2}\right)\right)}{\sigma^2\underline{\gamma}(2\gamma-\underline{\gamma})}(N+1)\\
    &\leq \frac{8C_h\left(2+e^{\underline{\gamma}K_h/2}\left(4+e^{\underline{\gamma}x/2}\right)\right)}{\sigma^2\underline{\gamma}(2\gamma-\underline{\gamma})} (\log_2(T+1)+1).
    \end{align*}
    Since for $T\geq 1$, $\log_2(T+1)=\log_2T+\log_2\left(1+\frac{1}{T}\right)\leq \log_2T+1$, we obtain the desired result.

    The proof of \eqref{eq:transient-regret-bound-AU-2} follows from the same argument as that of \eqref{eq:transient-regret-bound-LTO-3}.
\end{proof}

\section{Learning Regret Analysis}\label{sec:5}
In this section, we analyze $R_{4}^{\mathrm{LTO}}$, $R_{3}^{\mathrm{AU}}$ and $R_3^{\mathrm{AU\text{-}FH}}$, proceeding in two steps.
First, we study the interrelationship between the cost function $C(u,r)$, the reflection level $r$, and the learned parameter (denoted by $u$). Local regularity results imply that, in a neighborhood of the true parameter $\gamma$, the reflection level mapping $r(u)$ is Lipschitz continuous in $u$, and the cost function $C(u,r)$ is bounded by a quadratic function in $r$.
Second, we examine the efficiency of our estimator---specifically, its convergence rate to the true parameter $\gamma$. We establish a mean squared error (MSE) bound, which allows us to derive both the convergence rate and the probability that the estimated parameter lies outside the aforementioned neighborhood.

\subsection{Local Regularity Property}\label{sec:5.1}
In this part, we only show the key results; their proofs are provided in Appendix~\ref{app:B.1}.
We first need the continuity and differentiability of the functions $C(u,r)$ and $F(u,r)$.
\begin{lemma}\label{lemma:regularity}
    \begin{enumerate}[$(i)$]
    \item $C(u,r)$ is jointly continuous in $(u,r)$ on $[\underline{\gamma},\infty)\times\mathbb{R}$. Furthermore, $F(u,r)$ is jointly continuous on $[\underline{\gamma},\infty)\times\mathbb{R}$.
    \item $\partial_uF(u,r)=\partial_uC(u,r)$ is jointly continuous on $[\underline{\gamma},\infty)\times\mathbb{R}$.
    \item $\partial_rC(u,r)=u\left(C(u,r)-h(r)\right)$  and $\partial_rC(u,r)$ is jointly continuous on $[\underline{\gamma},\infty)\times\mathbb{R}$.
    \item $\partial_r F(u,r)=\partial_rC(u,r)-h^{\prime}(r)$ is jointly continuous on $[\underline{\gamma},\infty)\times(-\infty,0)$.
    \end{enumerate}
\end{lemma}
We mainly justify the interchange of limits and derivatives with respect to the integral in Lemma~\ref{lemma:regularity}.

\begin{lemma}\label{lemma:exist-unique-finite-solution}
    There exists a constant $K_h>0$ such that for all $u\geq \underline{\gamma}$, the optimal equation $F(u,r)=0$ has a unique negative solution $r(u)\in(-K_h,0)$, such that
    \begin{align}\label{eq:optimal-r-function}
    C\left(u,r(u)\right)< C(u,r),\quad \forall r\in\mathbb{R}\setminus\{r(u)\}.
    \end{align}
\end{lemma}

\begin{lemma}\label{lemma:finite-stationary-cost}
    For each $r\in[-K_h,0]$, there exists a constant $M$ such that $C(\gamma,r)\leq M$.
\end{lemma}
Although Lemma~\ref{lemma:exist-unique-finite-solution} shows that the optimality equation $F(u,r)=0$ has a unique solution in $(-K_h,0)$ for all $u\geq\underline{\gamma}$, it is not enough to ensure the regularity required for the regret analysis. Indeed, when $u\to\infty$, $r(u)$ may tend to $0$, in which case it may result in
\[
\partial_r F(u,r(u))=uF(u,r(u))-h^{\prime}(r(u))=-h^{\prime}(r(u))\to-h^{\prime}(0).
\]
Assumption~\ref{ass:1} $(i)$ does not give the regularity of $h$ at $0$. Instead of establishing the global regularity, we only focus on the local regularity around the true parameter $\gamma$.
Formally, we have the following proposition.
\begin{proposition}\label{prop:local-lipschitz}
    There exist constants $\varepsilon_0>0$, $\delta_0>0$ and $L_0>0$ such that for all $u_1,u_2\in (\gamma-\varepsilon_0,\gamma+\varepsilon_0)\cap[\underline{\gamma},\infty)$, $r(u_1)$ and $r(u_2)$ lie in $\left(r^*-\delta_0,r^*+\delta_0\right)\subseteq(-K_h,0)$ and
    \begin{align}\label{eq:Lip-r-u}
    \lvert r(u_1)-r(u_2)\rvert\leq L_0\lvert u_1-u_2\rvert.
    \end{align}
\end{proposition}

This proposition tells us the Lipschitz continuity of $r(u)$ around $\gamma$. 
\begin{proposition}\label{prop:costgap}
    There exists a constant $L_1$ such that for each $r\in[r^*-\delta_0,r^*+\delta_0]$,
    \begin{align}\label{eq:quadratic-C-r}
        0\leq C(\gamma,r)-C(\gamma,r^*)\leq L_1(r-r^*)^2.
    \end{align}
\end{proposition}
This proposition tells us that the stationary cost gap in the neighborhood of the optimal $r^*$ is bounded by a quadratic function in $r$. Define the good event 
\begin{align}\label{eq:good-event}
G_{\varepsilon_0}:=\{u\geq\underline{\gamma}:\lvert u-\gamma\rvert\leq \varepsilon_0\}.
\end{align}
Then Proposition~\ref{prop:local-lipschitz} and Proposition~\ref{prop:costgap} hold on $G_{\varepsilon_0}$.

\subsection{Consistency and Mean Squared Error Bound for the Estimator}\label{sec:5.2}
Based on the ergodicity property in Lemma~\ref{lemma:ergodicity}, it is natural to use $\left(\frac{1}{\tau}\int_0^{\tau}Z_0(t)\mathrm{d}t\right)^{-1}$ to approximate $\gamma$ in LTO algorithm. Since we have the prior lower bound $\underline{\gamma}>0$ for $\gamma$, we let 
\[
\hat{\gamma}=\hat{\gamma}(\tau):=\left(\frac{1}{\tau}\int_0^{\tau}Z_0(t)\mathrm{d}t\right)^{-1}\vee\underline{\gamma}
\]
to be the estimator for $\gamma$. From this definition, $\hat{\gamma}$ explicitly depends on the choice of the exploration period $\tau$.
The estimator $\hat{\gamma}_k$ defined in \eqref{eq:hat-gamma-k} for AU algorithm and the estimator $\tilde{\gamma}_k$ defined in \eqref{eq:tilde-gamma-k} for the AU-FH algorithm follow the same design idea. Since $\hat{\gamma}_k$ is constructed based on the interval $[\tau_{k-1},\tau_k)$, during which the state process evolves as an RBM with reflecting level $\hat{r}_{k-1}$, we let
\[
\hat{\gamma}_{k}:=\left(\frac{1}{\Delta\tau_{k-1}}\int_{\tau_{k-1}}^{\tau_{k}}\left(Z_{\hat{r}_{k-1}}^{\tau_{k-1},X_{k-1}}(t)-\hat{r}_{k-1}\right)\mathrm{d}t\right)^{-1}\vee \underline{\gamma}.
\]
Since $\tilde{\gamma}_k$ is constructed based on the interval $[0,\tau_k)$, we let
\begin{align}\label{eq:tilde-gamma-k-rewrite}
\tilde{\gamma}_{k}:=\left(\frac{1}{\tau_{k}}\sum_{i=0}^{k-1}\int_{\tau_i}^{\tau_{i+1}}\left(Z_{\tilde{r}_{i-1}}^{\tau_{i-1},\tilde{X}_{i-1}}(s)-\tilde{r}_{i-1}\right)\mathrm{d}s\right)^{-1}\vee\underline{\gamma}.
\end{align}
For $k\geq1$, let $$\tilde{q}_{k}:=\frac{1}{\Delta\tau_{k-1}}\int_{\tau_{k-1}}^{\tau_{k}}\left(Z_{\tilde{r}_{k-1}}^{\tau_{k-1},\tilde{X}_{k-1}}(s)-\tilde{r}_{k-1}\right)\mathrm{d}s$$ and $$\tilde{q}_{k}^{\mathrm{FH}}:=\frac{\sum_{i=0}^{k-1}\Delta\tau_i \tilde{q}_{i+1}}{\sum_{i=0}^{k-1}\Delta\tau_i},$$
then $\tilde{\gamma}_{k}$ defined in \eqref{eq:tilde-gamma-k-rewrite}  can rewrite as $\tilde{\gamma}_{k}=\frac{1}{\tilde{q}_{k}^{\mathrm{FH}}}\vee\underline{\gamma}$.

\begin{lemma}[Consistency of the estimators]\label{lemma:consistency}
    The following hold.
    \begin{enumerate} 
    \item [(i)] 
    The estimator $\hat{\gamma}$ in $\mathrm{LTO}$ algorithm satisfies $\hat{\gamma}\to\gamma$ as $\tau\to \infty$.
    \item [(ii)]
    The estimator $\hat{\gamma}_k$ in $\mathrm{AU}$ algorithm satisfies $\hat{\gamma}_k\to\gamma$ as $k\to \infty$.
    \item [(iii)]
    The estimator $\tilde{\gamma}_k$ in $\mathrm{AU}\text{-}\mathrm{FH}$ algorithm satisfies $\tilde{\gamma}_k\to\gamma$ as $k\to \infty$.
    \end{enumerate}
\end{lemma}
The proofs in this subsection are  provided in Appendix~\ref{app:B.2}.
In the following, we prove the MSE bound for these time-average estimators.
We first establish a uniform polynomial moment bound for the reflected Brownian motion $Z_r$ defined in \eqref{eq:SDE-startfrom-s}.

\begin{lemma}\label{lemma:Zr-polynomial-bound}
    For $p\geq1$ and $p\in\mathbb{N}$, 
    \[
    \sup_{t\geq 0}\mathbb{E}_x\left[\left\lvert Z_r(t)\right\rvert^p\right]\leq 2^{2p-2}\left(\lvert r\rvert^p+(x-r)^p+\frac{p!}{\gamma^p}\right). 
    \] 
\end{lemma}
We have the following lemma, which provides an MSE bound for the reciprocal of the estimator without considering the prior lower bound.
\begin{lemma}\label{lemma:MSE-bound}
    There exists a constant $C_3<\infty$ (depending on $\theta$, $\sigma$, and $K_h$) such that for each $r\in[-K_h,0]$, $x\geq r$ and $T\geq 1$,
    \begin{align}\label{eq:MSE-bound}
        \mathbb{E}_x\left[\left\lvert \frac{1}{T}\int_0^T \left(Z_r(t)-r\right)\mathrm{d}t-\frac{1}{\gamma}\right\rvert^2\right]\leq \frac{C_3(1+\lvert x\rvert^4)}{T}.
    \end{align}
\end{lemma}

Recall $G_{\varepsilon_0}$ in \eqref{eq:good-event}. We have the following results for  the estimator $\hat{\gamma}$.

\begin{lemma}\label{lemma:learning-two-results}
The following hold:
\begin{subequations}
\begin{align}
\mathbb{P}_x(\hat{\gamma}\in G_{\varepsilon_0}^c)&\leq\frac{\gamma^2(\gamma+\varepsilon_0)^2C_3(1+\lvert x\rvert^4)}{\varepsilon_0^2\tau};\label{eq:probability-bad-event}\\
\mathbb{E}_x\left[\left\lvert\hat{\gamma}-\gamma\right\rvert^2\mathbf{1}_{\{\hat{\gamma} \in G_{\varepsilon_0}\}}\right]&\leq\frac{\gamma^2(\gamma+\varepsilon_0)^2C_3(1+\lvert x\rvert^4)}{\tau}\label{eq:convergence-rate-hat-gamma}.
\end{align}
\end{subequations}
\end{lemma}

\eqref{eq:probability-bad-event} characterizes the probability that the estimator $\hat{\gamma}$ in LTO algorithm falls outside $G_{\varepsilon_0}$. This probability decreases as $\tau$ increases. 
\eqref{eq:convergence-rate-hat-gamma} is the MSE bounds for the estimator $\hat{\gamma}$ in a neighborhood of $\gamma$.
Combining \eqref{eq:convergence-rate-hat-gamma}, Proposition~\ref{prop:local-lipschitz} and Proposition~\ref{prop:costgap}, we can directly bound the learning regret over $G_{\varepsilon_0}$. 

Notice that $\hat{\gamma}_k$ in \eqref{eq:hat-gamma-k} comes from the time average process over $[\tau_{k-1},\tau_k)$. Since $X_{k-1}$ and $\hat{r}_{k-1}$ are measurable with respect to $\mathcal{F}_{\tau_{k-1}}$, by Lemma~\ref{lemma:MSE-bound},
we obtain \[
\mathbb{E}_x\left[\left\lvert\frac{1}{\Delta_{\tau_{k-1}}}\int_{\tau_{k-1}}^{\tau_k}\left(Z_{\hat{r}_{k-1}}^{\tau_{k-1},X_{k-1}}(t)-\hat{r}_{k-1}\right)\mathrm{d}t-\frac{1}{\gamma}\right\rvert^2\bigg|\mathcal{F}_{\tau_{k-1}}\right]\leq \frac{C_3\left(1+\lvert X_{k-1}\rvert^4\right)}{\Delta\tau_{k-1}}.
\]
Similar to \eqref{eq:probability-bad-event} and \eqref{eq:convergence-rate-hat-gamma}, we also have the following lemma.

\begin{lemma}\label{lemma:learning-two-results-AU}
For $k\geq1$, the following hold:
\begin{subequations}
\begin{align}
\mathbb{P}_x(\hat{\gamma}_k\in G_{\varepsilon_0}^c|\mathcal{F}_{\tau_{k-1}})&\leq\frac{\gamma^2(\gamma+\varepsilon_0)^2C_3(1+\lvert X_{k-1}\rvert^4)}{\varepsilon_0^2\Delta\tau_{k-1}};\label{eq:probability-bad-event-AU}\\
\mathbb{E}_x\left[\left\lvert\hat{\gamma}_k-\gamma\right\rvert^2\mathbf{1}_{\{\hat{\gamma}_k \in G_{\varepsilon_0}\}}\big|\mathcal{F}_{\tau_{k-1}}\right]&\leq\frac{\gamma^2(\gamma+\varepsilon_0)^2C_3(1+\lvert X_{k-1}\rvert^4)}{\Delta \tau_{k-1}}\label{eq:convergence-rate-hat-gamma-AU}.
\end{align}
\end{subequations}
\end{lemma}

To bound $\mathbb{E}_x\left[\lvert X_{k-1}\rvert^4\right],k\geq1$, we have the following lemma.
\begin{lemma}\label{lemma:bound-X-4}
    \begin{align}\label{eq:polynomial-bound-X}
    \mathbb{E}_x\left[\lvert X_{k-1}\rvert^4\right]\leq 2^3\left(K_h^4+2^6\left(x^4+\frac{4!}{\gamma^4}\right)\right).
\end{align}
\end{lemma}

Now we establish the MSE bound for $\tilde{q}_{k+1}^{\mathrm{FH}}$. By Remark~\ref{remark:AU-FH} and using results analogous to Lemma~\ref{lemma:MSE-bound} and Lemma~\ref{lemma:bound-X-4} for AU-FH, we can obtain the existence of constant $C_4>0$ (depending on $\theta$, $\sigma$, $K_h$, and $x$) such that for all $i\geq1$,
\begin{align}\label{eq:MSE-q-i}
    \mathbb{E}_x\left[\left(\tilde{q}_{i}-\frac{1}{\gamma}\right)^2\right]\leq\frac{C_4}{\Delta\tau_{i-1}}.
\end{align}
Combining \eqref{eq:MSE-q-i} and Minkowski inequality, we have the following lemma.
\begin{lemma}\label{lemma:MSE-bound-4-AUFH}
    There exists a constant $C_5>0$ (depending on $\theta$, $\sigma$, $K_h$, and $x$) such that for $k\geq 1$,
    \begin{align}\label{eq:MSE-q-FH}
    \mathbb{E}_x\left[\left(\hat{q}_{k}^{FH}-\frac{1}{\gamma}\right)^2\right]\leq\frac{C_5}{\tau_k}.
    \end{align}
\end{lemma}

\subsection{Learning Regret Bound}\label{sec:5.3}
Combining the results in Sections~\ref{sec:5.1} and~\ref{sec:5.2}, we present the learning regret bounds for proposed algorithms in this subsection. Recall the learning regret in LTO algorithm
\[
R_{4}^{\mathrm{LTO}}=\mathbb{E}_x\left[\int_{\tau}^T \left(C(\gamma,\hat{r}_{\tau})-C(\gamma,r^*)\right)\mathrm{d}t \right].
\]

\begin{proposition}[Learning regret bound in LTO algorithm]\label{prop:learning-regret-bound-LTO}
    For $x\geq0$,
    \begin{align}\label{eq:R22-LTO-bound}
        R_{4}^{\mathrm{LTO}}\leq L_1 L_0^2(T-\tau)\frac{\gamma^2(\gamma+\varepsilon_0)^2C_3(1+x^4)}{\tau}+2M(T-\tau)\frac{\gamma^2(\gamma+\varepsilon_0)^2C_3(1+x^4)}{\varepsilon_0^2\tau}.
    \end{align}
\end{proposition}
\begin{proof}
    Decompose the learning regret by event $\{\hat{\gamma}\in G_{\varepsilon_0}\}$, we have
    \begin{align*}
        R_{4}^{\mathrm{LTO}}&=\mathbb{E}_x\left[\int_{\tau}^T \left(C(\gamma,\hat{r}_{\tau})-C(\gamma,r^*)\right)\mathrm{d}t \right]\\
        &=\mathbb{E}_x\left[\int_{\tau}^T \left(C(\gamma,\hat{r}_{\tau})-C(\gamma,r^*)\right)\mathrm{d}t\cdot\mathbf{1}_{\left\{\hat{\gamma}\in G_{\varepsilon_0}\right\}} \right]+\mathbb{E}_x\left[\int_{\tau}^T \left(C(\gamma,\hat{r}_{\tau})-C(\gamma,r^*)\right)\mathrm{d}t \cdot\mathbf{1}_{\left\{\hat{\gamma}\in G_{\varepsilon_0}^c\right\}}\right].
    \end{align*}
    By Propositions~\ref{prop:local-lipschitz} and~\ref{prop:costgap} and \eqref{eq:convergence-rate-hat-gamma}, the first term is bounded by
    \[
    \int_{\tau}^T L_1 L_0^2\mathbb{E}_x\left[\left\lvert\hat{\gamma}-\gamma\right\rvert^2\mathbf{1}_{\{\hat{\gamma} \in G_{\varepsilon_0}\}}\right]\mathrm{d} t\leq L_1 L_0^2(T-\tau)\frac{\gamma^2(\gamma+\varepsilon_0)^2C_3(1+x^4)}{\tau}.
    \]
    By Lemma~\ref{lemma:finite-stationary-cost} and \eqref{eq:probability-bad-event}, the second term is bounded by
    \[
    \int_{\tau}^T2M \mathbb{P}_x(\hat{\gamma}\in G_{\varepsilon_0}^c)\mathrm{d}t\leq2M(T-\tau)\frac{\gamma^2(\gamma+\varepsilon_0)^2C_3(1+x^4)}{\varepsilon_0^2\tau}.
    \]
    Thus, we complete the proof.
\end{proof}

Recall the learning regret in AU algorithm
\[
R_{3}^{\mathrm{AU}}=\mathbb{E}_x\left[\sum_{k=1}^{N}\int_{\tau_k}^{\tau_{k+1}} \left\lvert C(\gamma,\hat{r}_k)-C(\gamma,r^*)\right\rvert\mathrm{d}t\right].
\]

\begin{proposition}[Learning regret bound in AU algorithm]\label{prop:learning-regret-bound-AU}
    \begin{align}\label{eq:R2-AU-bound}
        R_{3}^{\mathrm{AU}}\leq 2\gamma^2(\gamma+\varepsilon_0)^2C_3\left(1+2^3K_h^4+2^9\left(x^4+\frac{4!}{\gamma^4}\right)\right)\left(L_1L_0^2+\frac{2M}{\varepsilon_0^2}\right)\left(\log_2T+1\right).
    \end{align}
\end{proposition}
\begin{proof}
    Decomposing the learning regret by event $\{\hat{\gamma}_k\in G_{\varepsilon_0}\}$, we have for $k\geq 1$,
    \begin{align*}
        \mathbb{E}_x\left[\int_{\tau_k}^{\tau_{k+1}} \left\lvert C(\gamma,\hat{r}_k)-C(\gamma,r^*)\right\rvert\mathrm{d}t\right]&=\mathbb{E}_x\left[\int_{\tau_k}^{\tau_{k+1}} \left\lvert C(\gamma,\hat{r}_k)-C(\gamma,r^*)\right\rvert\mathrm{d}t\cdot\mathbf{1}_{\left\{\hat{\gamma}_k\in G_{\varepsilon_0}\right\}} \right]\\
        &\quad+\mathbb{E}_x\left[\int_{\tau_k}^{\tau_{k+1}} \left\lvert C(\gamma,\hat{r}_k)-C(\gamma,r^*)\right\rvert\mathrm{d}t\cdot\mathbf{1}_{\left\{\hat{\gamma}_k\in G_{\varepsilon_0}^c\right\}} \right]\,. 
    \end{align*}
    By the tower property, the first term equals
    \begin{align*}
        \mathbb{E}_x\left[\mathbb{E}_x\left[\int_{\tau_k}^{\tau_{k+1}} \left\lvert C(\gamma,\hat{r}_k)-C(\gamma,r^*)\right\rvert\mathrm{d}t\cdot\mathbf{1}_{\left\{\hat{\gamma}_k\in G_{\varepsilon_0}\right\}}\bigg|\mathcal{F}_{\tau_{k-1}} \right]\right]\,.
    \end{align*}
    By Proposition~\ref{prop:local-lipschitz} and~\ref{prop:costgap}, \eqref{eq:convergence-rate-hat-gamma-AU} and \eqref{eq:polynomial-bound-X}, this can be bounded by
    \begin{align*}
   &  \int_{\tau_k}^{\tau_{k+1}} L_1 L_0^2\mathbb{E}_x\left[\mathbb{E}_x\left[\left\lvert\hat{\gamma}_k-\gamma\right\rvert^2\mathbf{1}_{\{\hat{\gamma}_k \in G_{\varepsilon_0}\}}\big|\mathcal{F}_{\tau_{k-1}}\right]\right]\mathrm{d} t \\
   &\leq L_1 L_0^2\Delta\tau_k\frac{\gamma^2(\gamma+\varepsilon_0)^2C_3\left(1+\mathbb{E}_x\left[\left\lvert X_{k-1}\right\rvert^4\right]\right)}{\Delta\tau_{k-1}}\\
    &\leq2L_1 L_0^2\gamma^2(\gamma+\varepsilon_0)^2C_3\left(1+2^3K_h^4+2^9\left(x^4+\frac{4!}{\gamma^4}\right)\right)\,.
    \end{align*}

    The second term equals
    \begin{align*}
        \mathbb{E}_x\left[\mathbb{E}_x\left[\int_{\tau_k}^{\tau_{k+1}} \left\lvert C(\gamma,\hat{r}_k)-C(\gamma,r^*)\right\rvert\mathrm{d}t\cdot\mathbf{1}_{\left\{\hat{\gamma}_k\in G_{\varepsilon_0}^c\right\}}\bigg|\mathcal{F}_{\tau_{k-1}} \right]\right]\,. 
    \end{align*}
    By   Lemma~\ref{lemma:finite-stationary-cost} and \eqref{eq:probability-bad-event-AU}, it can be bounded by
    \begin{align*}
    \int_{\tau_k}^{\tau_{k+1}} 2M\mathbb{E}_x\left[\mathbb{P}_x\left(\hat{\gamma}_k \in G_{\varepsilon_0}^c|\mathcal{F}_{\tau_{k-1}}\right)\right]\mathrm{d} t&\leq 2M\Delta\tau_k\frac{\gamma^2(\gamma+\varepsilon_0)^2C_3\left(1+\mathbb{E}_x\left[\left\lvert X_{k-1}\right\rvert^4\right]\right)}{\varepsilon_0^2\Delta\tau_{k-1}}\\
    &\leq\frac{4M}{\varepsilon_0^2}\gamma^2(\gamma+\varepsilon_0)^2C_3\left(1+2^3K_h^4+2^9\left(x^4+\frac{4!}{\gamma^4}\right)\right)\,. 
    \end{align*}
 
    Summing over $k=1$ to $N$, we have
    \begin{align*}
        R_{3}^{\mathrm{AU}}&=\mathbb{E}_x\left[\sum_{k=1}^{N}\int_{\tau_k}^{\tau_{k+1}} \left\lvert C(\gamma,\hat{r}_k)-C(\gamma,r^*)\right\rvert\mathrm{d}t\right]\nonumber\\
        &\leq 2\gamma^2(\gamma+\varepsilon_0)^2C_3\left(1+2^3K_h^4+2^9\left(x^4+\frac{4!}{\gamma^4}\right)\right)\left(L_1L_0^2+\frac{2M}{\varepsilon_0^2}\right)N\\
        &\leq 2\gamma^2(\gamma+\varepsilon_0)^2C_3\left(1+2^3K_h^4+2^9\left(x^4+\frac{4!}{\gamma^4}\right)\right)\left(L_1L_0^2+\frac{2M}{\varepsilon_0^2}\right)\left(\log_2T+1\right),
    \end{align*}
    where the last inequality follows from $N=\lfloor \log_2 (T+1)\rfloor\leq  \log_2 (T+1)\leq \log_2T+1$.
\end{proof}
For the AU-FH algorithm, Lemma~\ref{lemma:MSE-bound-4-AUFH}, together with the same good event decomposition argument as in Proposition~\ref{prop:learning-regret-bound-AU}, implies that the expected learning regret over each interval $I_k$ is bounded by a constant multiple of $\frac{\Delta\tau_k}{\tau_k}\leq 2$. Summing over the updating intervals, we obtain the following result.
\begin{proposition}[Learning regret bound in AU-FH algorithm]\label{prop:learning-regret-bound-AUFH}
\begin{align}\label{eq:R2-AUFH-bound}
R_{3}^{\mathrm{AU\text{-}\mathrm{FH}}}
\leq
2\gamma^2(\gamma+\varepsilon_0)^2C_5
\left(
L_1L_0^2+\frac{2M}{\varepsilon_0^2}
\right)
\left(\log_2T+1\right).
\end{align}
\end{proposition}

\section{Proofs of the Main Results}\label{sec:6}
\begin{proof}[Proof of Theorem~\ref{thm:LTO4generalh}]
    By Lemma~\ref{lemma:finite-stationary-cost}, the exploration regret satisfies
    \[
    R_{2}^{\mathrm{LTO}}=\mathbb{E}_x\left[\int_0^{\tau} \left(C(\gamma,0)-C(\gamma,r^*)\right)\mathrm{d}t \right]\leq 2M\tau.
    \]
    The bounds for the terms $R_{1}^{\mathrm{LTO}}$, $R_{3}^{\mathrm{LTO}}$ and $R_{5}^{\mathrm{LTO}}$ can be found in Proposition~\ref{prop:transient-regret-bound-LTO}. $R_4^{\mathrm{LTO}}$ is bounded in Proposition~\ref{prop:learning-regret-bound-LTO}.
    Combining the above results, we have
    \begin{align*}
        R^{\mathrm{LTO}}(x,T)&\leq \frac{8C_h\left(2+e^{\underline{\gamma}x/2}\right)}{\sigma^2\underline{\gamma}(2\gamma-\underline{\gamma})}+2M\tau+\frac{8C_h\left(2+e^{\underline{\gamma}K_h/2}\left(4+e^{\underline{\gamma}x/2}\right)\right)}{\sigma^2\underline{\gamma}(2\gamma-\underline{\gamma})}+\frac{8C_h\left(2+e^{\underline{\gamma}(x+K_h)/2}\right)}{\sigma^2\underline{\gamma}(2\gamma-\underline{\gamma})}\\
        &\quad+L_1 L_0^2(T-\tau)\frac{\gamma^2(\gamma+\varepsilon_0)^2C_3(1+x^4)}{\tau}+2M(T-\tau)\frac{\gamma^2(\gamma+\varepsilon_0)^2C_3(1+x^4)}{\varepsilon_0^2\tau}.
    \end{align*}
    Rearranging the terms in $R^{\mathrm{LTO}}(x,T)$, we find that  there exist constants $C_6,C_7,C_8>0$ such that
    \[
    R^{\mathrm{LTO}}(x,T)\leq C_6+C_7\tau+C_8\frac{T}{\tau}.
    \]
    To minimize the regret bound, we optimize over $\tau$: the term  $C_7\tau + C_8\frac{T}{\tau}$ is minimized when $\tau = \sqrt{T}$, leading to the best regret order $O(\sqrt{T})$. Equivalently, there exists a constant $C_1>0$ such that
    \[
    R^{\mathrm{LTO}}(x,T)\leq C_1\left(1+\sqrt{T}\right). \qedhere
    \]
\end{proof}

\begin{proof}[Proof of Theorem~\ref{thm:AU4generalh}]
    For exploration regret $R_{2}^{\mathrm{AU}}$, we have
    \begin{align}\label{eq:exploration-regret-bound-AU}
    R_{2}^{\mathrm{AU}}=\mathbb{E}_x\left[\int_{\tau_0}^{\tau_{1}} \left\lvert C(\gamma,\hat{r}_0)-C(\gamma,r^*)\right\rvert\mathrm{d}t\right]\leq 2M\Delta\tau_0=2M.
    \end{align}

Recalling the results in Proposition~\ref{prop:transient-regret-bound-AU} and Proposition~\ref{prop:learning-regret-bound-AU}, we have the existence of $C_2>0$ such that
\begin{align*}
    R^{\mathrm{AU}}(x,T)&\leq \frac{8C_h \left(2+e^{\underline{\gamma} K_h/2}\left(4+e^{\underline{\gamma}x/2}\right)\right)}{\sigma^2\underline{\gamma}(2\gamma-\underline{\gamma})}\left(\log_2 T +2\right)+2M\\
    &\quad+2\gamma^2(\gamma+\varepsilon_0)^2C_3\left(1+2^3K_h^4+2^9\left(x^4+\frac{4!}{\gamma^4}\right)\right)\left(L_1L_0^2+\frac{2M}{\varepsilon_0^2}\right)\left(\log_2T+1\right)\\
    &\quad+ \frac{8C_h\left(2+e^{\underline{\gamma}(x+K_h)/2}\right)}{\sigma^2\underline{\gamma}(2\gamma-\underline{\gamma})}\\
    &\leq C_2(1+\log T).
\end{align*}
For AU-FH algorithm, since $R_2^{\mathrm{AU}\text{-}\mathrm{FH}}\leq 2M$, combining Proposition~\ref{prop:transient-regret-bound-AU} and Proposition~\ref{prop:learning-regret-bound-AUFH}, we have the existence of $C_3>0$ such that
\[
R^{\mathrm{AU}\text{-}\mathrm{FH}}(x,T)\leq C_3(1+\log T).
\]
This completes the proof.
\end{proof}

\section{Simulation-Based Numerical Results}\label{sec:7}
In this section, we conduct simulation experiments to illustrate the finite-time performance of the proposed learning algorithms.  
We first compare the expected finite-time regret of the three learning algorithms with a benchmark RL algorithm (REINFORCE) under four representative holding cost functions including absolute-value, quadratic, and exponential functions.
Then, we study the effect of diffusion scale on AU regret.

For four holding cost functions, we obtain the corresponding  optimal reflection levels as follows:
\begin{enumerate}[$(i)$]
    \item $h(x)=\lvert x\rvert$: in this case $C(u,r)=\frac{2 e^{ur}}{u}-r-\frac{1}{u}$ and $r(u)=-\frac{\log 2}{u}$.
    \item $h(x)=x^2$: in this case $C(u,r)=r^2+\frac{2r}{u}+\frac{2}{u^2}$ and $r(u)=-\frac{1}{u}$.
    \item $h(x)=e^{b\lvert x\rvert},0<b<\underline{\gamma}/2$: in this case $C(u,r)=\frac{u}{u+b}e^{-br}+\left(\frac{u}{u-b}-\frac{u}{u+b}\right)e^{ur}$ and $r(u)=\log\left(\frac{1}{2}-\frac{b}{2u}\right)/(u+b)$.
    \item $h(x)=1-e^{-\lvert x\rvert}$: in this case $C(u,r)=1+\frac{2u}{u^2-1}e^{ur}-\frac{u}{u-1}e^r$ for $u\neq1$, $C(1,r)=1+(r-\frac{1}{2})e^r$ and $r(u)=\frac{1}{1-u}\log \left(\frac{2u}{u+1}\right)$ for $u\neq1$, $r(1)=-\frac{1}{2}$.
\end{enumerate}

Unless otherwise specified, we set the initial state to $x=0.2$. The true parameters are chosen as $\theta=-1,\sigma=1$ so that $\gamma=2$. We choose $b=0.5$ in the exponential holding cost function, and the prior lower bound is set to $\underline\gamma=1.1$ to ensure $b<\underline{\gamma}/2$. For other cost functions, we set $\underline{\gamma}=0.1$.

\subsection{Expected Regret Demonstration and Benchmark Comparison}\label{sec:7.1}
For each cost function, we simulate $2000$ independent Brownian sample paths over the time horizon $T=500$. We compute the regret processes of the LTO, AU and AU-FH algorithms and then average over all replications to estimate the expected regret. We also adopted REINFORCE, a classical model-free policy-gradient RL algorithm, as a benchmark (see Appendix~\ref{app:C.2} for details). Figures~\ref{fig:regret-abs-2000},~\ref{fig:regret-qua-2000}, \ref{fig:regret-exp-2000} and~\ref{fig:regret-bound-2000} report the cumulative regret curves.

The simulation results are qualitatively consistent with the theoretical bounds. The expected regret of the LTO algorithm exhibits a square-root growth pattern, while the expected regrets of the AU and AU-FH algorithms grow much more slowly, in line with the logarithmic regret bound.

The comparison highlights the benefit of exploiting the problem structure in algorithm design: our model-based algorithms consistently outperform the generic REINFORCE benchmark. The improvement of AU over LTO reflects the advantage of repeatedly updating the estimator using informative data. Moreover, AU-FH further yields a slight performance improvement over AU by incorporating the full history of observations into parameter estimation.

\begin{figure}[htbp]
    \centering
    \begin{subfigure}[b]{0.45\textwidth}
        \centering
        \includegraphics[width=\textwidth]{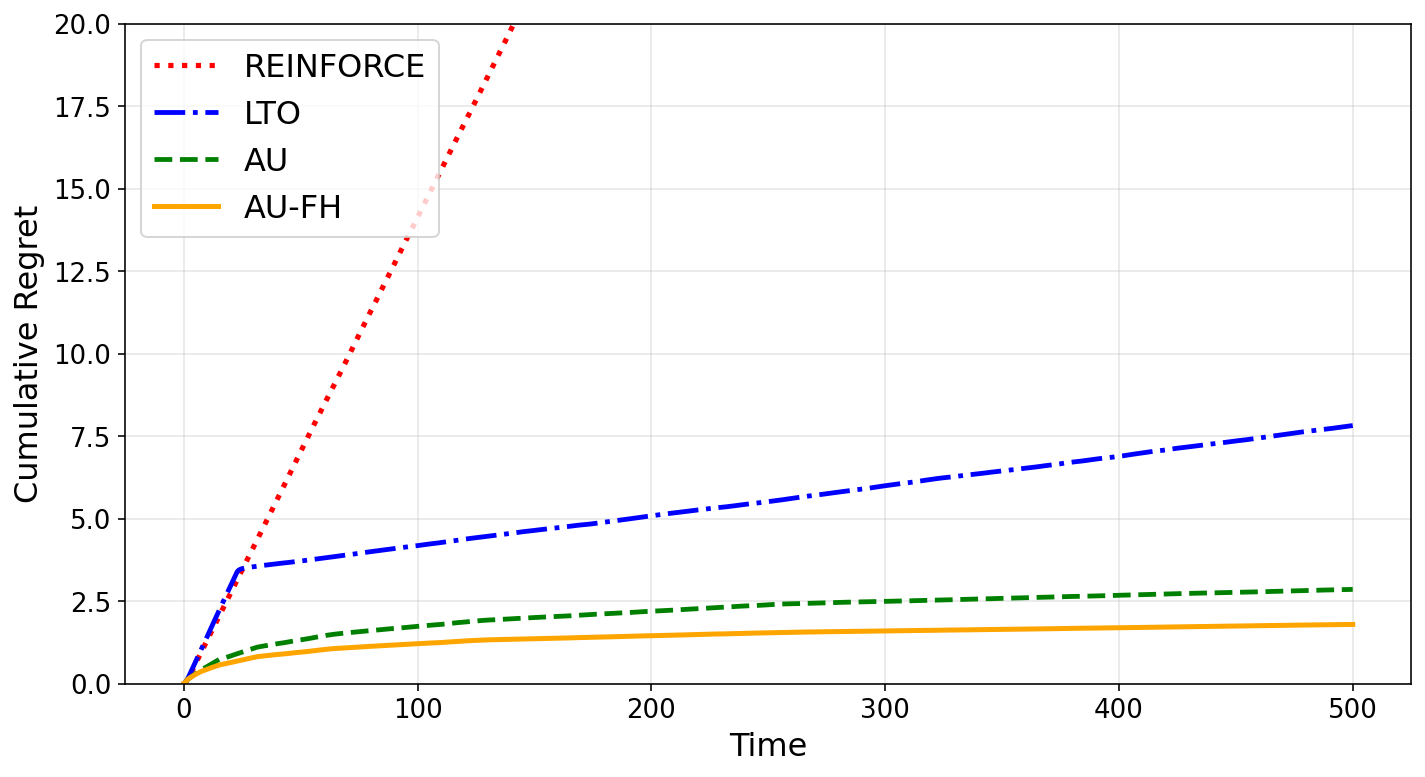}
        \caption{$h(x)=\lvert x\rvert$}
        \label{fig:regret-abs-2000}
    \end{subfigure}
    \hfill
    \begin{subfigure}[b]{0.45\textwidth}
        \centering
        \includegraphics[width=\textwidth]{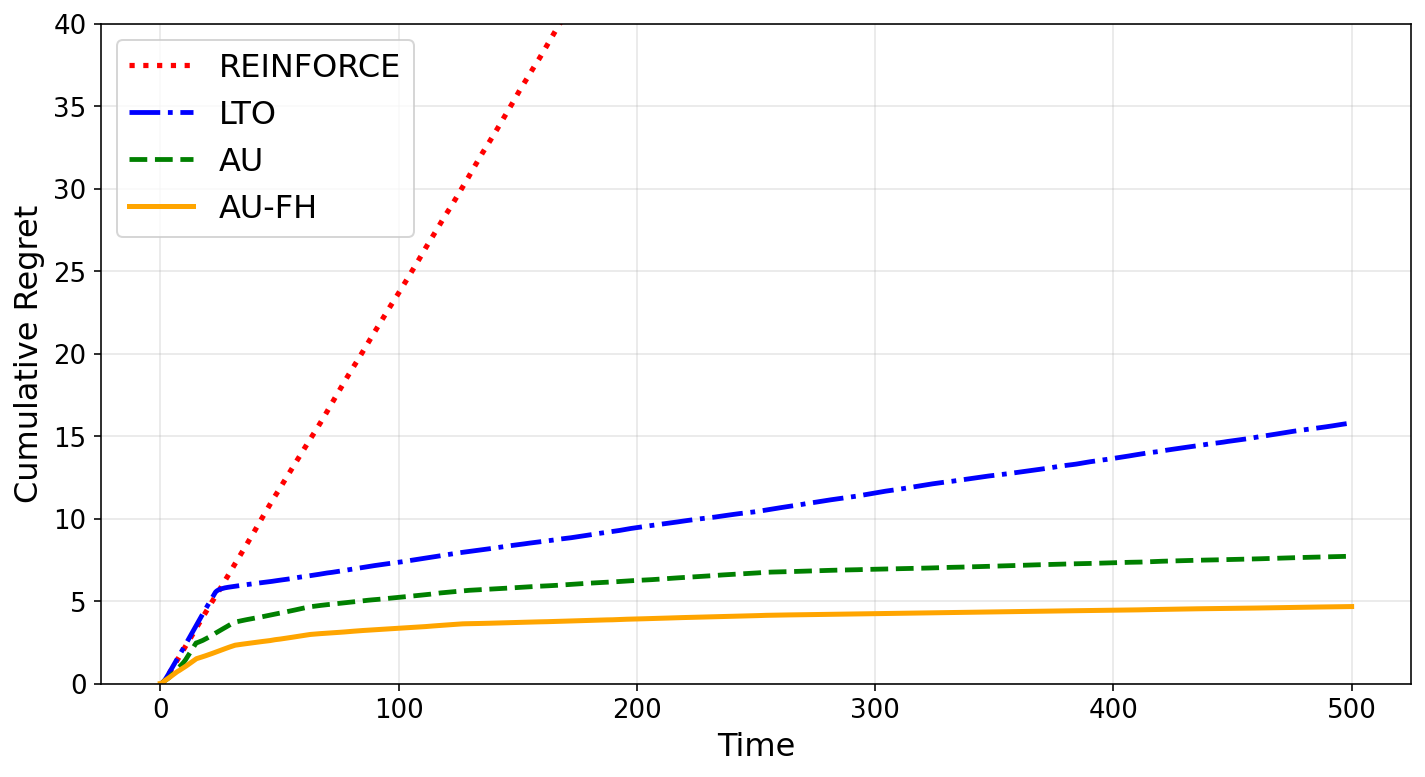}
        \caption{$h(x)=x^2$}
        \label{fig:regret-qua-2000}
    \end{subfigure} 
    \\
    \begin{subfigure}[b]{0.45\textwidth}
        \centering
        \includegraphics[width=\textwidth]{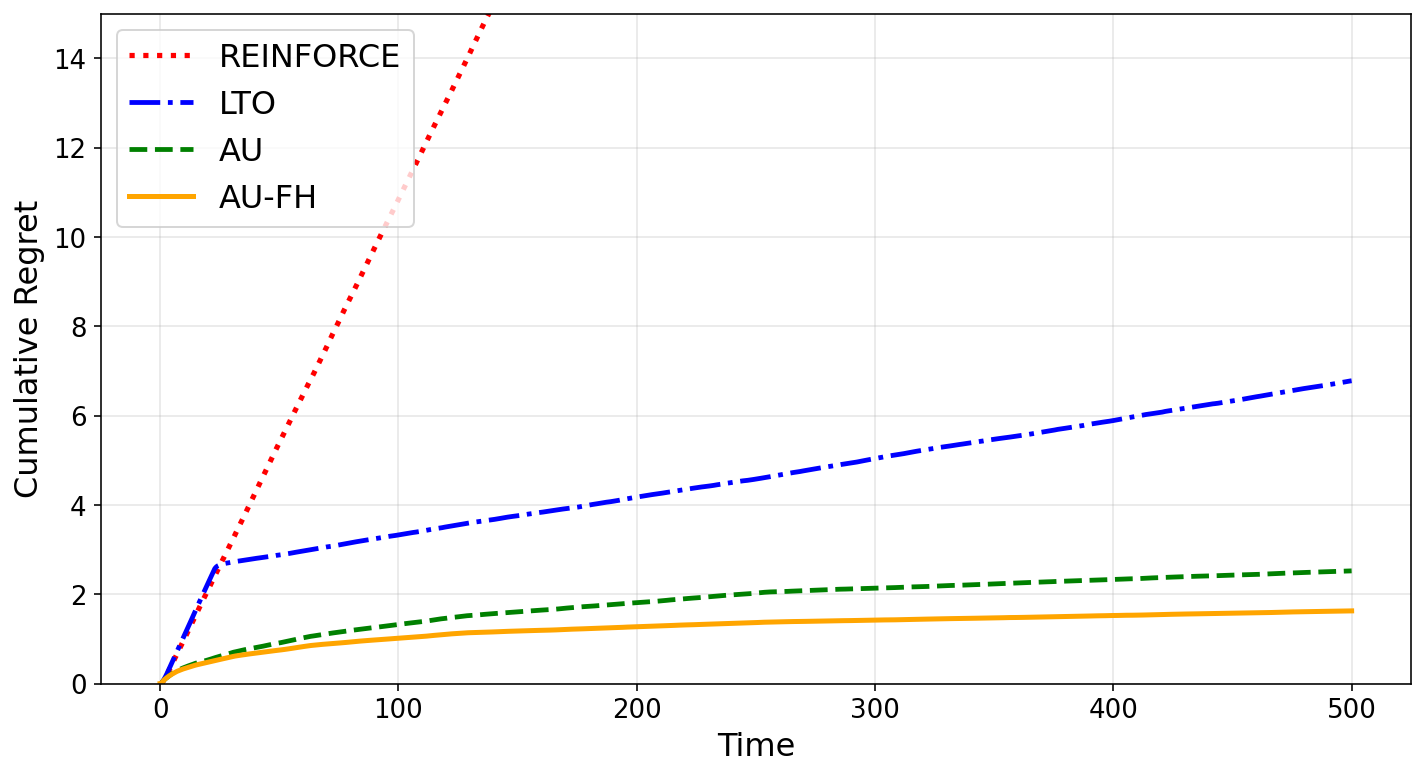}
        \caption{$h(x)=e^{0.5\lvert x\rvert}$}
        \label{fig:regret-exp-2000}
    \end{subfigure}
    \hfill
    \begin{subfigure}[b]{0.45\textwidth}
        \centering
        \includegraphics[width=\textwidth]{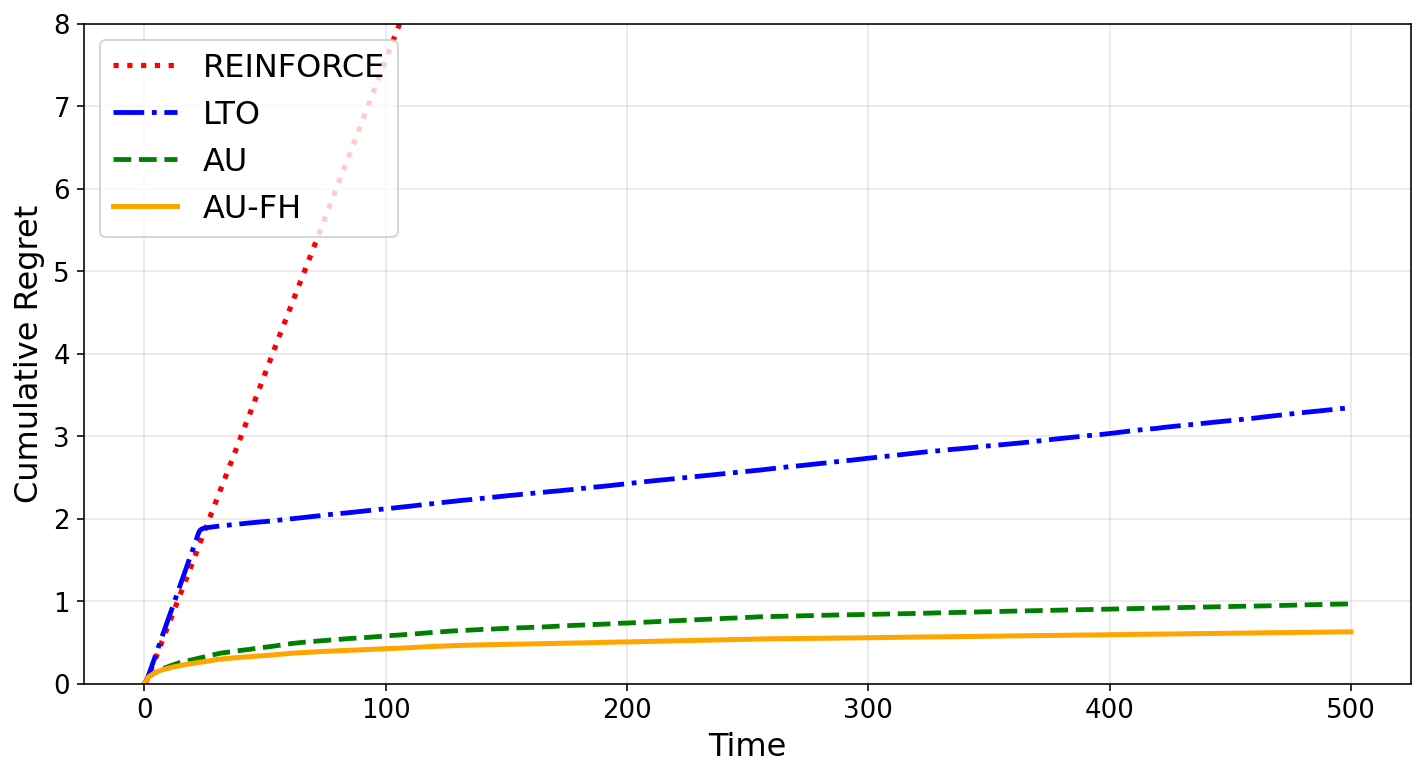}
        \caption{$h(x)=1-e^{-\lvert x\rvert}$}
        \label{fig:regret-bound-2000}
    \end{subfigure}
    \caption{Regret comparison of LTO, AU, and AU-FH algorithm  with REINFORCE algorithm under various cost functions}
\end{figure}

\subsection{Impact of Diffusion Scale on AU Regret}\label{sec:7.2}
To further investigate how the diffusion scale influences the finite-time learning performance of the AU algorithm, we vary $\sigma$ and adjust $\theta$ accordingly while fixing the policy-relevant parameter $\gamma = -2\theta/\sigma^2$. Under this setup, both the long-run average holding cost and the full-information optimal reflecting level remain unchanged---this design allows us to isolate the impact of diffusion scale variations on learning performance against a consistent benchmark.

\begin{figure}[htbp]
    \centering
    \includegraphics[scale=0.45]{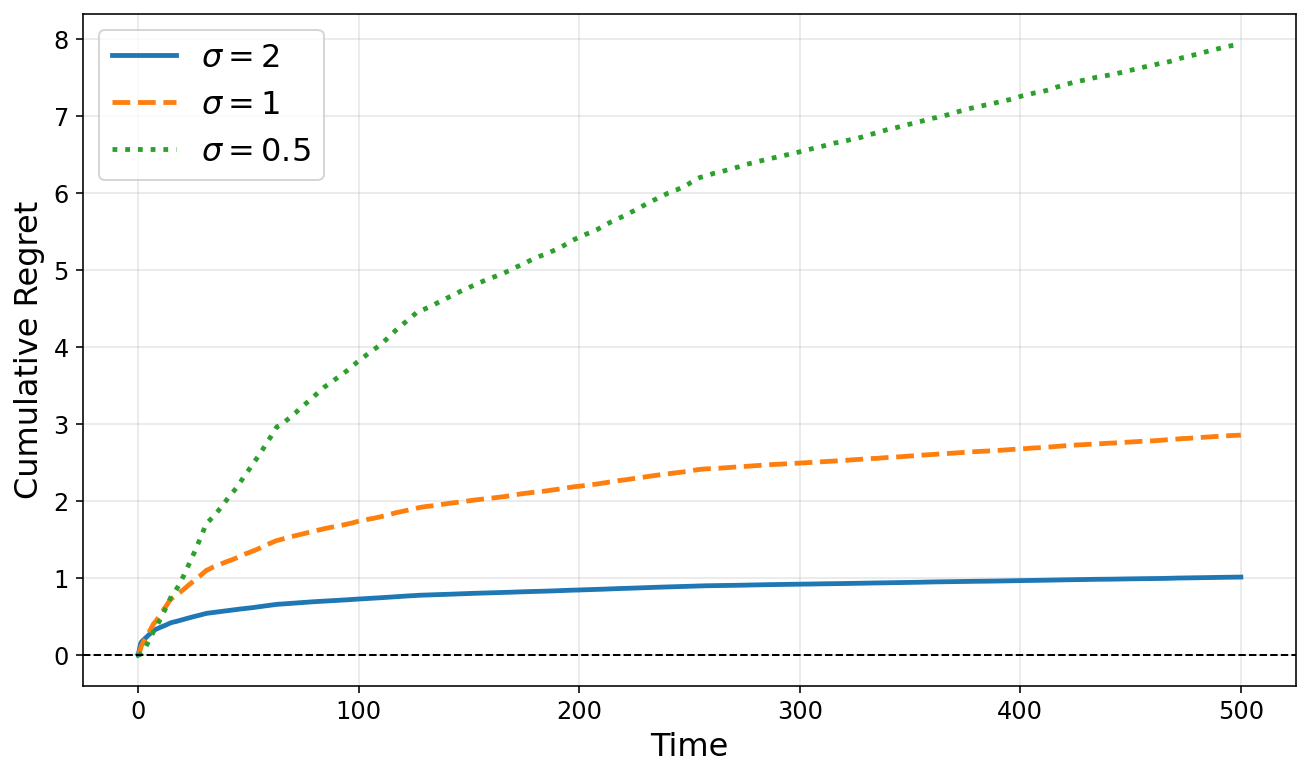}
    \caption{Effect of diffusion scale on AU regret}
    \label{fig:change-sigma-abs}
\end{figure}

Figure~\ref{fig:change-sigma-abs} compares the expected regret of the AU algorithm under several parameter settings with the same value $\gamma=2$ for $h(x)=|x|$ (The figures for other cost functions and their impact on AU-FH regret can be found in Appendix~\ref{app:C.4}). The results show that the expected regret decreases more rapidly as $\sigma$ increases.
This phenomenon can be understood through the generator of the RBM in \eqref{eq:generator}. 
We have
\[
\Gamma_{\sigma} f(z)=\sigma^2\Gamma_1^{(\gamma)}f(z), \quad \text{with}\quad  \Gamma_1^{(\gamma)}f(z)=\frac{1}{2}\left(f^{\prime\prime}(z)-\gamma f^{\prime}(z)\right).
\]
Hence, for fixed $\gamma$, increasing $\sigma$ accelerates the intrinsic time scale of the reflected diffusion by a factor  $\sigma^2$. Over the same physical time horizon, the process converges faster to the stationary distribution and the time-average estimator of $\gamma$ is based on more informative observations. This leads to faster stabilization of the AU policy and smaller finite-time expected regret. We also plot the impact of varying only $\sigma$ on the optimal reflection level and optimal long-run average cost, as shown in Appendix~\ref{app:C.3}.

\section{Concluding Remarks}\label{sec:8}
In this paper, we have developed three data-driven algorithms for the Brownian reflection control problem and establish their corresponding finite-time regret bounds. Based on the regret decomposition, we show how to use the Foster-Lyapunov inequality to bound the transient regret and derive several key ingredients (local regularity property, consistency and MSE bounds for the estimator) to control the learning regret. Furthermore, we leverage the monotonicity of the moving boundary Skorokhod mapping to bound the sample path of the AU and AU-FH algorithms. In numerical experiments, we demonstrate the advantages of the proposed algorithms by fully leveraging the problem structure.

Several directions remain open for future research. First, for reflection control, it would be natural to extend the present framework to a two-sided control framework. We expect not only our algorithm design principle but also the Foster-Lyapunov method employed in our transient regret analysis to remain implementable.
Second, high-dimensional (reflected) Brownian control problems are substantially more challenging: on one hand, characterizing optimal policy is analytically intractable; on the other hand, designing high-dimensional learning algorithms and conducting regret analysis are more difficult.
Third, it would be interesting to incorporate explicit control cost into the model, with the aim of conducting finite-time regret analysis. 
Finally, more flexible adaptive mechanisms may also be considered.
The AU and AU-FH algorithms in our paper use a deterministic doubling schedule, which has two useful analytical consequences: the number of policy updates up to time $T$ is only $O(\log T)$, 
and the learning regret on each interval is bounded by a constant, which are critical to yield the $O(\log T)$ regret bound. 
Beyond such deterministic schedules, one may consider data-dependent updating or stopping rules, where the updating is adaptive to the information generated by the data.

From an application perspective, our study can be viewed as a data-driven optimal reflection control in Brownian production-inventory models with unknown system primitives.
A few papers have studied related learning problems under different information structures for Brownian inventory model. \citet{cadenillas2013optimal} consider a Brownian demand process with Markov-modulated drift and variance parameters, and derive an optimal production policy when the underlying environment is not directly observable. \citet{federico2023two} use Bayesian updating to learn an unknown drift and manage the inventory through two-sided singular control.
In contrast, our work develops a model-based learning framework for Brownian inventory control and provides a finite-time regret analysis. More broadly, our data-driven model-based learning approach offers a promising framework for designing learning algorithms and analyzing regret in unknown environments for other operations research models---including queueing control, production-inventory systems, storage systems, and revenue management operations---that are described by Brownian or diffusion models and admit structured optimal policies (e.g., threshold policy, singular control, impulse control) under full-information settings.

\newpage 
\bibliographystyle{plainnat}
\bibliography{reference}

\newpage
\appendix
\section{Appendix: Proofs in Section~\ref{sec:4}}\label{app:A}
\subsection{Proof in Section~\ref{sec:4.1}}\label{app:A.1}
\begin{proof}[Proof of Lemma~\ref{lemma:rbm-ergodicity}]
    $(i)$: Recall the Skorokhod mapping from $x+\theta t+\sigma B(t)$ to $(Y_r(t),Z_r(t))$ in \eqref{eq:Skorokhod-mapping-Y} and \eqref{eq:Skorokhod-mapping-Z}. For $r\leq x\leq y$, we construct $Z_r^{0,x}$ and $Z_r^{0,y}$ defined in \eqref{eq:SDE-startfrom-s} driven by the same Brownian motion. By the monotonicity of the Skorokhod mapping, we obtain
    \[
    Z_r^{0,x}(t)\leq Z_r^{0,y}(t),\quad t\geq 0,\quad \text{a.s.}
    \]
    which implies $Z_r$ is stochastically ordered.

    $(ii)$: For $t<\tau_r$, by It\^o's formula,
    \[
    \mathrm{d} V_r(Z_r(t))=V_r^{\prime}(Z_r(t))\mathrm{d} Z_r(t)+\frac{1}{2}V_r^{\prime\prime}(Z_r(t))\mathrm{d}\braket{Z_r}_t.
    \]
    Here $\braket{Z_r}$ denotes the predictable quadratic variation of $Z_r$.
    Since $\mathrm{d}\braket{Z_r}_t=\sigma^2 \mathrm{d} t$, we have
    \[
    \mathrm{d}V_r(Z_r(t))=\left(\frac{\theta\underline{\gamma}}{2}+\frac{\sigma^2\underline{\gamma}^2}{8}\right)V_r(Z_r(t))\mathrm{d}t+\frac{\sigma\underline{\gamma}}{2} V_r(Z_r(t))\mathrm{d}B(t)+\frac{\underline{\gamma}}{2}  V_r(Z_r(t))\mathrm{d}Y_r(t).
    \]
    By the definition of $\tau_r$ in \eqref{eq:tau-r}, $Z_r(t)>r$, hence  $\mathrm{d}Y_r(t)=0$ on $\left[0,\tau_r\right)$. Thus, 
    \[
    V_r\left(Z_r(t\wedge\tau_r)\right)=V_r(Z_r(0))+\int_0^{t\wedge \tau_r} \left(\frac{\theta\underline{\gamma}}{2}+\frac{\sigma^2\underline{\gamma}^2}{8}\right)V_r(Z_r(s))\mathrm{d}s+\frac{\sigma\underline{\gamma}}{2}\int_0^{t\wedge \tau_r}  V_r(Z_r(s))\mathrm{d}B(s).
    \]
    Rewriting the above equality, we have
    \[
    V_r\left(Z_r(t\wedge\tau_r)\right)+\left(-\frac{\theta\underline{\gamma}}{2}-\frac{\sigma^2\underline{\gamma}^2}{8}\right)\int_0^{t\wedge\tau_r}V_r(Z_r(s))\mathrm{d}s=V_r(Z_r(0))+\frac{\sigma\underline{\gamma}}{2}\int_0^{t\wedge \tau_r} V_r(Z_r(s))\mathrm{d}B(s).
    \]
    The right-hand side is a local martingale and the left-hand side is nonnegative since $V_r\geq1$ and
    \[
    -\frac{\theta\underline{\gamma}}{2}-\frac{\sigma^2\underline{\gamma}^2}{8}=\frac{\sigma^2\underline{\gamma}(2\gamma-\underline{\gamma})}{8}>0.
    \]
    Therefore, $V_r\left(Z_r(t\wedge\tau_r)\right)+\left(\sigma^2\underline{\gamma}(2\gamma-\underline{\gamma})/8\right)\int_0^{t\wedge\tau_r}V_r(Z_r(s))\mathrm{d}s$
    is a supermartingale with Lyapunov constant $\sigma^2\underline{\gamma}(2\gamma-\underline{\gamma})/8$.

    $(iii)$: Recalling $\pi_r$ in \eqref{eq:stationary-Zr}, we have
    \[
    \int_{[r,\infty)}V_{r}\mathrm{d}\pi_r=\int_r^\infty e^{\underline{\gamma}(y-r)/2}\gamma e^{-\gamma (y-r)}\mathrm{d}y=\frac{\gamma}{\gamma-\underline{\gamma}/2}<\infty. \qedhere
    \]
\end{proof}

\subsection{Proofs in Section~\ref{sec:4.2}}\label{app:A.2}
\begin{proof}[Proof of Lemma~\ref{lemma:exp-moment-bound-Z0}]
    The result in \eqref{eq:ergodicity-convergence} for $r=0$ implies that
    \begin{align}\label{eq:ergodicity-convergence-Z0}
    \left\|\mathbb{P}_t^{(0)}(x,\cdot)-\pi_0\right\|_{V_0}\leq \left[2+e^{\underline{\gamma} x/2}\right]e^{-\sigma^2\underline{\gamma}(2\gamma-\underline{\gamma})t/8},\quad\quad t\geq0
    \end{align}
    since $\int_{[0,\infty)}V_0\mathrm{d}\pi_0=\frac{\gamma}{\gamma-\underline{\gamma}/2}\leq 2$ and $V_0(z)=e^{\underline{\gamma} z/2}$.
    
    Then we have
    \begin{align*}
    \mathbb{E}_x\left[ e^{\underline{\gamma} Z_0(t)/2}\right]&=\int_{[0,\infty)}e^{\underline{\gamma}z/2}\mathbb{P}_t^{(0)}(x,\mathrm{d}z)\\
    &=\int_{[0,\infty)}V_0(z)\pi_0(\mathrm{d}z)+\int_{[0,\infty)}V_0(z)\left(\mathbb{P}_t^{(0)}(x,\mathrm{d}z)-\pi_0(\mathrm{d}z)\right)\\
    &\leq 2+\left\|\mathbb{P}_t^{(0)}(x,\cdot)-\pi_0\right\|_{V_0}\\
    &\leq 4+e^{\underline{\gamma} x/2}
    \end{align*}
    where the first inequality follows from the definition of $V$-norm and the second inequality follows from \eqref{eq:ergodicity-convergence-Z0}.
    The above inequality holds for all $t\geq0$, so we obtain
    \[
    \sup_{t\geq0}\mathbb{E}_x\left[e^{\underline{\gamma} Z_0(t)/2}\right]\leq 4+e^{\underline{\gamma}x/2}. \qedhere
    \]
\end{proof}

\begin{proof}[Proof of Lemma~\ref{lemma:Markov-property-LTO}]
We first note that the pre-specified exploration period $\tau$ is deterministic. By the sample path construction in Section~\ref{sec:3.1}, $\hat{r}_{\tau}$ is computed from the observations before time $\tau$, and $X=Z_0(\tau)$. Therefore both $\hat{r}_{\tau}$ and
$X$ are $\mathcal F_{\tau}$-measurable.

For $s\geq 0$, define the shifted Brownian motion
\[
B^{\tau}(s):=B(\tau+s)-B(\tau).
\]
By the strong Markov property of Brownian motion, $B^{\tau}$ is a standard Brownian motion independent of $\mathcal F_{\tau}$.

Conditional on $\mathcal F_{\tau}$, the pair $\left(\hat{r}_{\tau},X\right)$ is fixed. On the interval $[\tau,T]$, the process satisfies
\[
Z^{\tau,X}_{\hat{r}_{\tau}}(\tau+s)=X+\theta s+\sigma B^{\tau}(s)+Y^{\tau,X}_{\hat{r}_{\tau}}(s)
\qquad 0\leq s<T-\tau.
\]
Hence, conditional on $\mathcal F_{\tau}$,
\[
\left\{Z^{\tau,X}_{\hat{r}_{\tau}}(\tau+s):0\leq s<T-\tau\right\}
\]
has the same law as a reflected Brownian motion with reflection level $\hat{r}_{\tau}$ starting from $X$. Therefore, for every $s\geq 0$,
\[
\mathbb E_x\left[h\left(Z^{\tau,X}_{\hat{r}_{\tau}}(\tau+s)\right)\middle|\mathcal F_{\tau}
\right]=\mathbb{E}_X^{\hat{r}_{\tau}}\left[h\left(Z_{\hat{r}_{\tau}}^{0,X}(s)\right)\right].
\]
Integrating over $s\in[0,T-\tau]$ gives the desired identity.
\end{proof}

\subsection{Proofs in Section~\ref{sec:4.3}}\label{app:A.3}
\begin{proof}[Proof of Lemma~\ref{lemma:moving-boundary-Skorokhod}]
    Since $b_1(t)\leq b_2(t),t\geq0$ we have for each $s\leq t$,
    \[
    \left(b_1(s)-x-\theta s-\sigma B(s)\right)^+\leq\left(b_2(s)-x-\theta s-\sigma B(s)\right)^+.
    \]
    Taking the supremum over $s\in[0,t]$ gives $Y_1(t)\leq Y_2(t)$, hence
    \[
    Z_1(t)=x+\theta t+\sigma B(t)+Y_1(t)\leq x+\theta t+\sigma B(t)+Y_2(t)=Z_2(t).
    \]
    This proves the result.
\end{proof}

\begin{proof}[Proof of Corollary~\ref{Cor:1}]
    Since the optimality equation only has a negative solution, we have $\hat{r}_k\leq0$ for each $k\in\mathbb{N}$.
    Then using Lemma~\ref{lemma:moving-boundary-Skorokhod} with $b_1(t)=\hat{r}(t)\leq0$ and $b_2(t)=0$ for all $t\geq0$ and each $\omega\in\Omega$, we complete the proof. Since $X_k=Z_{\hat{r}}(\tau_k)$, we have $X_k\leq Z_0(\tau_k)$ almost surely.
\end{proof}

\section{Proofs in Section~\ref{sec:5}}\label{app:B}

\subsection{Proofs in Section~\ref{sec:5.1}}\label{app:B.1}

\begin{proof}[Proof of Lemma~\ref{lemma:regularity}]
    $(i)$: Recall the definition of $C$ in \eqref{eq:def-C} and change of variables formula gives
    \[
    C(u,r)=\int_0^{\infty}u h(r+z)e^{-uz}\mathrm{d}z.
    \]
    Fix $(u_0,r_0)\in[\underline{\gamma},\infty)\times\mathbb{R}$, and let $(u_n,r_n)\to(u_0,r_0)$, with $u_n\geq \underline{\gamma}$ and $r_n\in\mathbb{R}$. Then there exists a $N_0$ such that for all $n\geq N_0$
    \[
    \frac{u_0}{2}\leq u_n\leq 2u_0,\quad r_0-1\leq r_n\leq r_0+1 .
    \]
    
    Let $L:=\max\{0,1-r_0\}$. For $0\leq z\leq L$ and $n\geq N_0$, the arguments $r_n+z\in[r_0-1,r_0+1+L]$, hence by the continuity of $h$,
    \[
    H:=\sup_{[r_0-1,r_0+1+L]}h(x)<\infty.
    \]
    Thus we have
    \begin{align}\label{eq:bound-for-yleqL}
        u_nh(r_n+z)e^{-u_nz}\leq2u_0H.
    \end{align}
    
    For $z\geq L$ and $n\geq N_0$, we have
    \[
    r_n+z\geq r_0-1+z\geq0.
    \]
    Assumption~\ref{ass:1} $(iii)$ ensures the existence of $a>0$ such that
    \begin{align*}
        h(r_n+z)\leq ae^{b(r_n+z)}\leq ae^{b(r_0+1)}e^{bz}.
    \end{align*}
    So
    \begin{align}\label{eq:bound-for-ygeqL}
        u_nh(r_n+z)e^{-u_nz}\leq2u_0ae^{b(r_0+1)}e^{-(u_n-b)z}\leq2u_0ae^{b(r_0+1)}e^{-(u_0/2-b)z}.
    \end{align}
    
    Combining \eqref{eq:bound-for-yleqL} and \eqref{eq:bound-for-ygeqL}, we obtain for $n\geq N_0$,
    \begin{align}\label{eq:bound-for-all}
    u_n h(r_n+z)e^{-u_n z}\leq 2u_0H\mathbf 1_{\{0\leq z\leq L\}}+2 u_0a e^{b(r_0+1)} e^{-(u_0/2-b) z}\mathbf 1_{\{z>L\}}.
    \end{align}
    Since $u_0\geq \underline{\gamma}$ and $b<\underline{\gamma}/2$, we have $u_0/2>b$ and thus the right-hand side is integrable on $[0,\infty)$. Moreover, for each fixed $z\geq0$, since $u_n\to u_0$, $r_n\to r_0$, and $h$ is continuous, we have
    \[
    u_n h(r_n+z)e^{-u_n z}\to u_0 h(r_0+z)e^{-u_0z} \quad \text{as} \quad n\to\infty.
    \]
    By the dominated convergence theorem,
    \[
    \lim_{n\to\infty}\int_0^\infty u_nh(r_n+z)e^{-u_nz}\mathrm{d}z=\int_0^\infty u_0h(r_0+z)e^{-u_0z}\mathrm{d}z.
    \]
    That is,
    \[
    C(u_n,r_n)\to C(u_0,r_0) \quad \text{as} \quad n\to\infty.
    \]
    Since $(u_0,r_0)$ is arbitrary, $C$ is jointly continuous on $[\underline{\gamma},\infty)\times\mathbb{R}$. Recall the definition of $F(u,r)$ in \eqref{eq:F(u,r)}, we directly have $F$ is jointly continuous on $[\underline{\gamma},\infty)\times\mathbb{R}$.

    $(ii)$: Differentiating $C(u,r)=u\int_0^\infty h(r+z)e^{-uz}\mathrm{d}z$ with respect to $u$ gives
    \[
    \partial_u C(u,r)=\int_0^\infty h(r+z)e^{-uz}(1-uz)\mathrm{d}z .
    \]
    We justify this expression and its joint continuity by dominated convergence. Fix the same $u_n$, $r_n$, $u_0$, $r_0$, $N_0$, $L$ and $H$ as above in $(i)$, for $0\leq z\leq L$ and $n\geq N_0$,
    \[
    \left|h(r_n+z)e^{-u_nz}(1-u_nz)\right|\leq H(1+2u_0L).
    \]
    For $z>L$ and $n\geq N_0$, we use the exponential bound in Assumption~\ref{ass:1} $(iii)$ to get
    \[
    \begin{aligned}
    \left|h(r_n+z)e^{-u_nz}(1-u_nz)\right|&\leq a e^{b(r_0+1)} e^{bz} e^{-u_nz}(1+u_nz) \\
    &\leq a e^{b(r_0+1)}(1+2u_0z)e^{-(u_0/2-b)z}.
    \end{aligned}
    \]
    The right-hand side of the above inequality is integrable over $(L,\infty)$. Hence by the dominated convergence theorem, we have 
    \[
        \partial_u C(u_n,r_n)\to\partial_u C(u_0,r_0)\quad \text{as} \quad n\to\infty.
    \]
    This proves that $\partial_u C(u,r)$ exists and is jointly continuous on $[\underline{\gamma},\infty)\times\mathbb{R}$. Since $F(u,r)=C(u,r)-h(r)$,
    we also have
    \[
    \partial_u F(u,r)=\partial_u C(u,r),
    \]
    and $\partial_u F$ is jointly continuous on
    $[\underline{\gamma},\infty)\times\mathbb{R}$.

    $(iii)$: Recall that in \eqref{eq:def-C},
    \[
    C(u,r)=u e^{ur}\int_r^\infty h(z)e^{-uz}\mathrm{d}z .
    \]
    Differentiating this expression with respect to $r$ and using the continuity of $h$ yields
    \begin{align*}
    \partial_r C(u,r)&=u^2 e^{ur}\int_r^\infty h(z)e^{-uz}\mathrm{d}z-u h(r) \\
    &=u\left(C(u,r)-h(r)\right).
    \end{align*}
    Hence, 
    \[
    \partial_r C(u,r)=uF(u,r).
    \]
    By the joint continuity of $C$ and the continuity of $h$, it follows that $\partial_r C(u,r)$ is jointly continuous on $[\underline{\gamma},\infty)\times\mathbb{R}$.

    $(iv)$: For $r<0$, since $h$ is continuously differentiable on $(-\infty,0)$, we have
    \[
    \partial_r F(u,r)=\partial_r C(u,r)-h'(r)=u\bigl(C(u,r)-h(r)\bigr)-h'(r).
    \]
    Therefore $\partial_r F$ is jointly continuous on $[\underline{\gamma},\infty)\times(-\infty,0)$. This proves the lemma. 
\end{proof}

\begin{proof}[Proof of Lemma~\ref{lemma:exist-unique-finite-solution}]
    Recall the definition of $C(u,r)$ in \eqref{eq:def-C}. Let $\xi_u\sim\text{Exp}(u)$. We have for $r\geq0$,
    \begin{align*}
        C(u,r)&=u\int_0^{\infty}h(z+r)e^{-u z}\mathrm{d}z=\mathbb{E}\left[h(r+\xi_u)\right]>h(r).
    \end{align*}
    The inequality holds since $h(x)$ is strictly increasing in $[0,\infty)$. Thus $F(u,r)=0$ does not have non-negative solution for $u\geq\underline{\gamma}$.   
    
    By Assumption~\ref{ass:1} $(iii)$, there exists some constant $a>0$ such that $h(x)\leq ae^{bx},x\geq 0$.
    Consequently, for every $u\geq \underline{\gamma}$,
    \begin{align*}
    u\int_0^\infty h(x)e^{-ux}\mathrm{d}x&\leq a u\int_0^\infty e^{-(u-b)x}\mathrm{d}x= \frac{au}{u-b}\leq \frac{a\underline{\gamma}}{\underline{\gamma}-b},
    \end{align*}
    where the last inequality follows from $0<b<\underline{\gamma}/2$.
    Hence
    \[
    M_0:=\sup_{u\geq \underline{\gamma}}u\int_0^\infty h(x)e^{-ux}\mathrm{d}x<\infty.
    \]
    Now fix $r<0$. For
    \[
    C(u,r)=u e^{ur}\int_r^\infty h(z)e^{-uz}\mathrm{d}z,
    \]
    splitting the integral at $0$, we obtain
    \[
    C(u,r)=u e^{ur}\int_r^0 h(z)e^{-uz}\mathrm{d}z+u e^{ur}\int_0^\infty h(z)e^{-uz}\mathrm{d}z.
    \]
    For $z\in[r,0)$, since $h'(x)<0$ on $(-\infty,0)$, we have $h(z)\leq h(r)$. Therefore, the first term
    \[
    u e^{ur}\int_r^0 h(z)e^{-uz}\mathrm{d}z\leq u e^{ur}h(r)\int_r^0 e^{-uz}\mathrm{d}z=h(r)(1-e^{ur}).
    \]
    The second term satisfies
    \[
    u e^{ur}\int_0^\infty h(z)e^{-uz}\,\mathrm{d}z\leq  M_0 e^{ur}.
    \]
    Combining the two estimates gives
    \[
    C(u,r)\leq h(r)(1-e^{ur})+M_0 e^{ur}.
    \]
    Hence,
    \[
    F(u,r)=C(u,r)-h(r)\leq e^{ur}\left(M_0-h(r)\right).
    \]

    By Assumption~\ref{ass:1} $(ii)$,
    \[
    h(r)\to+\infty,\quad \text{as}\quad r\to-\infty.
    \]
    Thus there exists $K_h>0$ such that for $r\leq -K_h$,
    \[
    h(r)>M_0.
    \]
    Consequently, for every $u\geq \underline{\gamma}$ and every $r\leq -K_h$,
    \[
    F(u,r)<0.
    \]
    Therefore no solution of $F(u,r)=0$ can lie in $(-\infty,-K_h]$. Moreover, for each $u\geq\underline{\gamma}$, since $F(u,0)>0$, $F(u,-K_h)<0$ and $F(u,r)$ is jointly continuous,  there exists a solution $\breve{r}$ in $(-K_h,0)$ such that $F(u,\breve{r})=0$.

    It remains to prove uniqueness. For fixed $u\geq\underline{\gamma}$, define
    \[
    G(u,r):=e^{-ur}F(u,r),\quad r<0.
    \]
    Using $\partial_r F(u,r)=uF(u,r)-h'(r)$, we obtain
    \[
    \partial_r G(u,r)=e^{-ur}\left(\partial_rF(u,r)-uF(u,r)\right)=-e^{-ur}h'(r)>0,
    \]
    because $h'(r)<0$ on $(-\infty,0)$. Hence $G(u,\cdot)$ is strictly increasing on $(-\infty,0)$. Since $e^{-ur}>0$, $G(u,\cdot)$ and $F(u,\cdot)$ have the same zeros. Therefore $F(u,r)=0$ has at most one solution on $(-K_h,0)$. So $\breve{r}$ is unique and it equals $r(u)$ by definition.
    The above analysis implies that for fixed $u\geq\underline{\gamma}$,
    \begin{align*}
        F(u,r)>0,\quad r(u)<r<0,
    \end{align*}
    and
    \begin{align*}
        F(u,r)<0,\quad r<r(u).
    \end{align*}
    Thus by $\partial_rC(u,r)=uF(u,r)$, we have for $r<0$
    \[
    C\left(u,r(u)\right)<C(u,r).
    \]
    By the continuity of $C(u,r)$ in $r$, we also have
    \[
    C\left(u,r(u)\right)<C(u,0).
    \]
    For $r>0$,
    \begin{align*}
    C(u,r)&=u\int_0^{\infty}h(r+y)e^{-u y}\mathrm{d}y>u\int_0^{\infty}h(y)e^{-u y}\mathrm{d}y=C(u,0),
    \end{align*}
    Thus
    \[
    C\left(u,r(u)\right)<C(u,r),\quad r\in\mathbb{R}. \qedhere
    \]
\end{proof}

\begin{proof}[Proof of Lemma~\ref{lemma:finite-stationary-cost}]
By Lemma~\ref{lemma:regularity}, $C(u,r)$ is jointly continuous on $[\underline{\gamma},\infty)\times\mathbb{R}$. In particular, for the true
parameter $\gamma$, the map
\[
    r\mapsto C(\gamma,r)
\]
is continuous on the compact interval $[-K_h,0]$. Therefore,
\[
    M:=\max_{r\in[-K_h,0]} C(\gamma,r)<\infty .
\]
Hence, for every $r\in[-K_h,0]$,
\[
    C(\gamma,r)\leq M.
\]
This proves the result.
\end{proof}

\begin{proof}[Proof of Proposition~\ref{prop:local-lipschitz}]
Recall that $r^*=r(\gamma)$ and $F(\gamma,r^*)=0$. By Lemma~\ref{lemma:regularity}, $F$ is continuously differentiable in a neighborhood of $(\gamma,r^*)$. Moreover,
\[
\partial_r F(\gamma,r^*)= \gamma\left(C(\gamma,r^*)-h(r^*)\right)-h'(r^*)= -h'(r^*)>0,
\]
where the second equality follows from the optimality equation
$C(\gamma,r^*)=h(r^*)$. 
By the implicit function theorem (IFT), there exist $\varepsilon_0>0$,
$\delta_0>0$, and a continuously differentiable function
$u\mapsto r(u)$ on
\[
U_0:=(\gamma-\varepsilon_0,\gamma+\varepsilon_0)\cap[\underline\gamma,\infty)
\]
such that
\[
F\left(u,r(u)\right)=0,\quad 
r(u)\in (r^*-\delta_0,r^*+\delta_0).
\]
By Lemma~\ref{lemma:exist-unique-finite-solution}, for each $u\geq \underline\gamma$, the equation $F(u,r)=0$
has a unique negative solution. Hence the local solution given by the IFT coincides with the globally defined solution $r(u)$.

Let $m_0:=-h'(r^*)>0$. Shrinking $\varepsilon_0$ if necessary, we may assume that
\[
\partial_r F(u,r)\ge \frac{m_0}{2},
\qquad
(u,r)\in U_0\times [r^*-\delta_0,r^*+\delta_0].
\]
Define
\[
M_1:=\sup_{(u,r)\in U_0\times [r^*-\delta_0,r^*+\delta_0]}
\left\lvert\partial_u F(u,r)\right\rvert. 
\]
Then $M_1<\infty$ since $\partial_u F$ is jointly continuous. 
Differentiating $F(u,r(u))=0$ gives
\[
r'(u)=-\frac{\partial_u F\left(u,r(u)\right)}{\partial_r F\left(u,r(u)\right)},
\]
and therefore
\[
|r'(u)|\le \frac{2M_1}{m_0},\qquad u\in U_0.
\]
Thus, for any $u_1,u_2\in U_0$, the mean value theorem yields
\[
|r(u_1)-r(u_2)|
\le
\frac{2M_1}{m_0}|u_1-u_2|.
\]
\eqref{eq:Lip-r-u} follows by taking $L_0:=2M_1/m_0$.
\end{proof}

\begin{remark}
    In Proposition~\ref{prop:local-lipschitz}, the IFT is applied in a neighborhood of the true parameter $\gamma$. Strictly speaking, if $\gamma=\underline{\gamma}$, the parameter $\gamma$ lies on the boundary of the domain $[\underline{\gamma},\infty)$, and hence the IFT cannot be applied directly on this half-line. This technical issue can be resolved by slightly enlarging the parameter domain. Indeed, by Assumption~\ref{ass:1} $(iii)$, there exists a constant $\eta_0>0$ such that
    \[
    2b<\underline{\gamma}-\eta_0 .
    \]
    Therefore, the dominated-convergence arguments in the proof of Lemma~\ref{lemma:regularity} remain valid with $[\underline{\gamma}-\eta_0,\infty)$ in place of $[\underline{\gamma},\infty)$. Consequently, the regularity properties of $C$ and $F$ hold on this enlarged domain. In particular, even when $\gamma=\underline{\gamma}$, the function $F$ is continuously differentiable in a full open neighborhood of $(\gamma,r^*)$, after choosing $\varepsilon_0<\eta_0$ if necessary. Thus, the IFT can still be applied in the proof of Proposition~\ref{prop:local-lipschitz}.
\end{remark}

\begin{proof}[Proof of Proposition~\ref{prop:costgap}]
By Lemma~\ref{lemma:exist-unique-finite-solution}, $r^*=r(\gamma)$ is the unique global minimizer of $C(\gamma,\cdot)$. Hence, for every $r\in \mathbb R$,
\[
C(\gamma,r)-C(\gamma,r^*)\geq 0.
\]

It remains to prove the quadratic upper bound. Since
$[r^*-\delta_0,r^*+\delta_0]\subseteq (-K_h,0)$, Lemma~\ref{lemma:regularity} implies that $C(\gamma,\cdot)$ is twice continuously differentiable on this interval, with
\[
\partial_{rr}C(\gamma,r)=\gamma\left(\partial_rC(\gamma,r)-h^{\prime}(r)\right)=\gamma\,\partial_r F(\gamma,r).
\]
Therefore,
\[
\bar L:=\sup_{s\in [r^*-\delta_0,r^*+\delta_0]}|\partial_{rr}C(\gamma,s)|<\infty .
\]
Moreover,
\[
\partial_r C(\gamma,r^*)=\gamma\left(C(\gamma,r^*)-h(r^*)\right)=0.
\]
By Taylor's formula, for each $r\in [r^*-\delta_0,r^*+\delta_0]$, there exists $\tilde r$ between $r$ and $r^*$ such that
\[
C(\gamma,r)-C(\gamma,r^*)=\frac12\,\partial_{rr}C(\gamma,\tilde r)(r-r^*)^2.
\]
Consequently,
\[
C(\gamma,r)-C(\gamma,r^*)\leq\frac{\bar L}{2}(r-r^*)^2.
\]
Taking $L_1:=\bar L/2$ proves \eqref{eq:quadratic-C-r}.
\end{proof}

\subsection{Proofs in Section~\ref{sec:5.2}}\label{app:B.2}
\begin{proof}[Proof of Lemma~\ref{lemma:consistency}]
    By Lemma~\ref{lemma:ergodicity},
    \[
    \lim_{\tau\to\infty}\frac{1}{\tau}\int_0^{\tau} Z_0(t)\mathrm{d}t=\frac{1}{\gamma}\quad \text{a.s.}
    \]
    Since $\gamma\geq\underline{\gamma}>0$ and
    \[
    \lim_{\tau\to\infty}\left(\frac{1}{\tau}\int_0^{\tau} Z_0(t)\mathrm{d}t\right)^{-1}=\gamma\quad \text{a.s.},
    \]
    we then have $\hat{\gamma}\to\gamma$ as $\tau\to\infty$. Replace the reflection level with $\hat{r}_k$ and notice that $\Delta \tau_{k-1}\to\infty$ as $k\to\infty$, we have $\hat{\gamma}_k\to\gamma$ as $k\to\infty$. 

    To show $\tilde{\gamma}_k\to\gamma$ as $k\to\infty$, we first recall that the definition of $\tilde{q}_{k}$ and $\tilde{q}_{k}^{\mathrm{FH}}$:
    \[
    \tilde{q}_{k}:=\frac{1}{\Delta\tau_{k-1}}\int_{\tau_{k-1}}^{\tau_{k}}\left(Z_{\tilde{r}_{k-1}}^{\tau_{k-1},\tilde{X}_{k-1}}(s)-\tilde{r}_{k-1}\right)\mathrm{d}s,\quad\tilde{q}_{k}^{\mathrm{FH}}:=\frac{\sum_{i=0}^{k-1}\Delta\tau_i \tilde{q}_{i+1}}{\sum_{i=0}^{k-1}\Delta\tau_i}.
    \]
    By Lemma~\ref{lemma:ergodicity}, we have $\tilde{q}_{k}\to1/\gamma$ almost surely as $k\to\infty$. Rewrite $\tilde{q}_{k}^{\mathrm{FH}}$, we have
    \[
    \tilde{q}_{k}^{\mathrm{FH}}=\frac{\sum_{i=0}^{k-1}\Delta\tau_i \tilde{q}_{i+1}}{\sum_{i=0}^{k-1}\Delta\tau_i}=\frac{\sum_{i=0}^{k-1}\Delta\tau_i \tilde{q}_{i+1}}{\tau_{k}}=\sum_{i=0}^{k-1}w_{k,i}\tilde{q}_{i+1},
    \]
    where $w_{k,i}:=\frac{\Delta\tau_i}{\tau_{k}}=\frac{2^i}{2^{k}-1}\geq0$ and $\sum_{i=0}^{k-1}w_{k,i}=1$. 
    Thus, the Toeplitz lemma implies that
    \[
    \tilde{q}_{k}^{\mathrm{FH}}\to\frac{1}{\gamma},\quad \text{a.s.}\quad as\quad k\to\infty,
    \]
    which yields that
    \[
    \tilde{\gamma}_{k}=\frac{1}{\tilde{q}_{k}^{\mathrm{FH}}}\vee\underline{\gamma}\to\gamma,\quad \text{a.s.}\quad as\quad k\to\infty. \qedhere
    \]
\end{proof}

\begin{proof}[Proof of Lemma~\ref{lemma:Zr-polynomial-bound}]
    We first consider $r=0$, and claim that
    \begin{align}\label{eq:tilde-Z0}
    \sup_{t\geq 0}\mathbb{E}\left[\left\lvert Z_0^{0,0}(t)\right\rvert^p\right]\leq\frac{p!}{\gamma^p}.
    \end{align}
    Let $Z_0^{0,\pi_0}(t)$ be another copy, driven by the same Brownian motion starting from stationary distribution $Z_0^{0,\pi_0}(0)\sim \mathrm{Exp}(\gamma)$. Then for every $t\geq 0$, $Z_0^{0,\pi_0}(t)\sim \text{Exp}(\gamma)$.
    By the monotonicity of the Skorokhod mapping and $Z_0^{0,\pi_0}(t)\geq0$, we obtain
    \[
    \sup_{t\geq 0}\mathbb{E}\left[\left\lvert Z_0^{0,0}(t)\right\rvert^p\right]\leq\sup_{t\geq 0}\mathbb{E}\left[\left\lvert Z_0^{0,\pi_0}(t)\right\rvert^p\right]=\mathbb{E}\left[\left\lvert Z_0^{0,\pi_0}(0)\right\rvert^p\right]=\frac{p!}{\gamma^p}.
    \]

    Now we turn to $Z_r(t)$. Let $\tilde{Z}(t):=Z_r(t)-r$ with initial $\tilde{Z}(0)=x-r\geq0$. Then $\{\tilde{Z}(t):t\geq 0\}$ is a standard RBM starting from $x-r\geq0$. We compare this process with $\left\{Z_0^{0,0}(t):t\geq0\right\}$ driven by the same Brownian motion.    

    Since the Skorokhod mapping is $1$-Lipschitz continuous, we have 
   \[
   \sup_{0\leq s\leq t}\left\lvert\tilde{Z}(s)-Z_0^{0,0}(s)\right\rvert\leq x-r,
   \]
   which implies 
   \[
   \tilde{Z}(t)\leq Z_0^{0,0}(t)+x-r,\quad \text{a.s.}
   \]
    Since $x^p$ is a convex function on $[0,\infty)$, from Jensen's inequality we have for any $a_1,a_2>0$
    \[
    (a_1+a_2)^p\leq 2^{p-1}\left(a_1^p+a_2^p\right). 
    \]
    Then, we have 
    \begin{align}\label{eq:tilde-Z}
        \mathbb{E}\left[\left\lvert\tilde{Z}(t)\right\rvert^p\right]\leq\mathbb{E}\left[\left\lvert Z_0^{0,0}(t)+x-r\right\rvert^p\right]\leq2^{p-1}\left(\mathbb{E}\left[\left\lvert Z_0^{0,0}(t)\right\rvert^p\right]+(x-r)^p\right).
    \end{align}
    Thus,
    \begin{align*}
        \sup_{t\geq 0}\mathbb{E}_x[\lvert Z_r(t)\rvert^p]\leq 2^{p-1}\left(\lvert r\rvert^p+\sup_{t\geq 0}\mathbb{E}\left[\left\lvert \tilde{Z}(t)\right\rvert^p\right]\right)\leq 2^{2p-2}\left(\lvert r\rvert^p+(x-r)^p+\frac{p!}{\gamma^p}\right).
    \end{align*}
    The last inequality follows from \eqref{eq:tilde-Z0} and \eqref{eq:tilde-Z}.
\end{proof}

\begin{proof}[Proof of Lemma~\ref{lemma:MSE-bound}]
    Notice that 
    \[
    f(z)=\frac{(z-r)^2}{2\theta}
    \]
    is the solution to the Poisson equation
    \[
    \Gamma f(z)=\left(z-r-\frac{1}{\gamma}\right),\quad z>r,
    \]
    where the generator $\Gamma$ is defined in \eqref{eq:generator}. 
    Indeed, we have
    \[
    f^{\prime}(z)=\frac{z-r}{\theta},\quad f^{\prime\prime}(z)=\frac{1}{\theta},
    \]
    which leads to 
    \[
    \Gamma f(z)=\theta \frac{z-r}{\theta}+\frac{\sigma^2}{2}\frac{1}{\theta}=z-r-\frac{1}{\gamma},
    \]
    and $f^{\prime}(r)=0$.

    Applying It\^o's formula to $f$, we obtain
    \[
    \int_0^T \left(Z_r(t)-r-\frac{1}{\gamma}\right)\mathrm{d} t=f(Z_r(T))-f(Z_r(0))-\sigma\int_0^T f^{\prime}(Z_r(t))\mathrm{d} B(t)-\int_0^T f^{\prime}(Z_r(t))\mathrm{d} Y_r(t).
    \]
    Since $f^{\prime}(r)=0$ and $Y_r(t)$ only increases on $Z_r(t)=r$, we have $\int_0^T f^{\prime}(Z_r(t))\mathrm{d} Y_r(t)=0$. Dividing by $T$ and taking absolute value on both sides, we arrive at
    \[
    \left\lvert \frac{1}{T}\int_0^T \left(Z_r(t)-r\right)\mathrm{d}t-\frac{1}{\gamma}\right\rvert=\frac{1}{T}\left\lvert f(Z_r(T))-f(Z_r(0))-\sigma\int_0^T f^{\prime}(Z_r(t))\mathrm{d} B(t)\right\rvert.
    \]
    Squaring both sides and taking the expectation, we obtain
    \begin{align*}
    \mathbb{E}_x\left[\left\lvert \frac{1}{T}\int_0^T \left(Z_r(t)-r\right)\mathrm{d}t-\frac{1}{\gamma}\right\rvert^2\right]&=\frac{1}{T^2}\mathbb{E}_x\left[\left\lvert f(Z_r(T))-f(Z_r(0))-\sigma\int_0^T f^{\prime}(Z_r(t))\mathrm{d}B(t)\right\rvert^2\right]\\
    &\leq\frac{1}{T^2}\mathbb{E}_x\left[\left( \left\lvert f(Z_r(T))-f(Z_r(0))\right\rvert+\left\lvert\sigma\int_0^T f^{\prime}(Z_r(t))\mathrm{d} B(t)\right\rvert\right)^2\right]\,.
    \end{align*}
    Therefore, by $(a_1+a_2)^2\leq 2(a_1^2+a_2^2)$, we have
    \begin{align}\label{eq:estimate-estimator}
    \mathbb{E}_x\left[\left\lvert \frac{1}{T}\int_0^T \left(Z_r(t)-r\right)\mathrm{d}t-\frac{1}{\gamma}\right\rvert^2\right]\leq\frac{2}{T^2}\mathbb{E}_x\left[\left\lvert f(Z_r(T))-f(x)\right\rvert^2\right]+\frac{2\sigma^2}{T^2}\mathbb{E}_x\left[\left\lvert \int_0^T f^{\prime}(Z_r(t))\mathrm{d}B(t)\right\rvert^2\right].
    \end{align}
    We estimate the two terms on the right-hand side of \eqref{eq:estimate-estimator} separately. 
    For the first term, we have
    \begin{align}\label{eq:estimator-first-term}
        \mathbb{E}_x\left[\left\lvert f(Z_r(T))-f(x)\right\rvert^2\right]&\leq2\left(\mathbb{E}_x\left[\left\lvert f(Z_r(T))\right\rvert^2\right]+\left\lvert f(x)\right\rvert^2\right)\nonumber\\
        &\leq 2\left(\mathbb{E}_x\left[\frac{\left(\lvert Z_r(T)\rvert-r\right)^4}{4\theta^2}\right]+\frac{(\lvert x\rvert-r)^4}{4\theta^2}\right)\nonumber\\
        &\leq2\left(\frac{2}{\theta^2}\left(\mathbb{E}_x\left[\lvert Z_r(T)\rvert^4\right]+r^4\right)+\frac{(\lvert x\rvert-r)^4}{4\theta^2}\right)\nonumber\\
        &\leq\frac{4}{\theta^2}\left(2^6\left(r^4+(\lvert x\rvert-r)^4+\frac{4!}{\gamma^4}\right)+r^4\right)+\frac{\left(\lvert x\rvert-r\right)^4}{2\theta^2},
    \end{align}
    where the first and third inequalities follow from Jensen's inequality, the second inequality follows from the definition of $f(z)$ and $(x-r)^4\leq \left(\lvert x\rvert-r\right)^4$, and the last inequality follows from Lemma~\ref{lemma:Zr-polynomial-bound}.

    For the second term, by It\^o isometry, we have
    \begin{align}\label{eq:estimator-second-term}
        \mathbb{E}_x\left[\left\lvert \int_0^T f^{\prime}(Z_r(t))\mathrm{d}B(t)\right\rvert^2\right]&=\int_0^T\mathbb{E}_x\left[\left\lvert  f^{\prime}(Z_r(t))\right\rvert^2\right]\mathrm{d} t\nonumber\\
        &\leq \int_0^T\mathbb{E}_x\left[\frac{\left(\left\lvert  Z_r(t)\right\rvert-r\right)^2}{\theta^2}\right]\mathrm{d} t\nonumber\\
        &\leq \int_0^T\left(\frac{2}{\theta^2}\mathbb{E}_x\left[\left\lvert  Z_r(t)\right\rvert^2\right]+\frac{2r^2}{\theta^2}\right)\mathrm{d} t\nonumber\\
        &\leq T\left(\frac{8}{\theta^2}\left(r^2+(\lvert x\rvert-r)^2+\frac{2}{\gamma^2}\right)+\frac{2r^2}{\theta^2}\right).
    \end{align}
    Thus, since $T\geq1$ and $r\in[-K_h,0]$, substituting \eqref{eq:estimator-first-term} and \eqref{eq:estimator-second-term} into \eqref{eq:estimate-estimator}, we obtain
    \begin{align*}
        \mathbb{E}_x\left[\left\lvert \frac{1}{T}\int_0^T \left(Z_r(t)-r\right)\mathrm{d}t-\frac{1}{\gamma}\right\rvert^2\right]&\leq \frac{2}{T^2}\left(\frac{4}{\theta^2}\left(2^6\left(r^4+(\lvert x\rvert-r)^4+\frac{4!}{\gamma^4}\right)+r^4\right)+\frac{\left(\lvert x\rvert-r\right)^4}{2\theta^2}\right)\\
        &\quad+\frac{2\sigma^2}{T^2}T\left(\frac{8}{\theta^2}\left(r^2+(\lvert x\rvert-r)^2+\frac{2}{\gamma^2}\right)+\frac{2r^2}{\theta^2}\right)\\
        &\leq \frac{2}{T}\left(\frac{4}{\theta^2}\left(2^6\left(K_h^4+(\lvert x\rvert+K_h)^4+\frac{4!}{\gamma^4}\right)+K_h^4\right)+\frac{\left(\lvert x\rvert+K_h\right)^4}{2\theta^2}\right)\\
        &\quad+\frac{2\sigma^2}{T}\left(\frac{8}{\theta^2}\left(K_h^2+(\lvert x\rvert+K_h)^2+\frac{2}{\gamma^2}\right)+\frac{2K_h^2}{\theta^2}\right)\\
        &\leq \frac{C_3(1+\lvert x\rvert^4)}{T}
    \end{align*}
    for some constant $C_3$ (depending only on $\theta$, $\sigma$, and $K_h$).
\end{proof}

\begin{proof}[Proof of Lemma~\ref{lemma:learning-two-results}]
Let $\eta_{\varepsilon_0}:=\frac{\varepsilon_0}{\gamma(\gamma+\varepsilon_0)}$. We first claim that
\[
\left\{ \hat{\gamma}:\left|\frac{1}{\hat{\gamma}}-\frac{1}{\gamma}\right|\leq \eta_{\varepsilon_0}\right\}\subseteq \{\hat{\gamma}\in G_{\varepsilon_0}\}.
\]
Suppose that $\left\lvert\frac{1}{\hat{\gamma}}-\frac{1}{\gamma}\right\rvert\leq \eta_{\varepsilon_0}$. Since $\eta_{\varepsilon_0}<1/\gamma$, we have $\frac{1}{\hat{\gamma}}\geq \frac{1}{\gamma}-\eta_{\varepsilon_0}>0$ and hence
\[
\hat{\gamma}\leq \frac{\gamma}{1-\gamma\eta_{\varepsilon_0}}.
\]
Therefore,
\[
\lvert\hat{\gamma}-\gamma\rvert=\gamma\hat{\gamma}\left\lvert\frac{1}{\hat{\gamma}}-\frac{1}{\gamma}\right\rvert
\leq\gamma\cdot \frac{\gamma}{1-\gamma\eta_{\varepsilon_0}}\cdot \eta_{\varepsilon_0}=\frac{\gamma^2\eta_{\varepsilon_0}}{1-\gamma\eta_{\varepsilon_0}}.
\]
By the definition of $\eta_{\varepsilon_0}$,
\[
\frac{\gamma^2\eta_{\varepsilon_0}}{1-\gamma\eta_{\varepsilon_0}}=\varepsilon_0.
\]
Hence $\lvert\hat{\gamma}-\gamma\rvert\leq \varepsilon_0$, which means $\hat{\gamma}\in G_{\varepsilon_0}$.

Recalling $\hat{\gamma}$ in \eqref{eq:hat-gamma}, we have
\[
\frac{1}{\hat{\gamma}}=\left(\frac{1}{\tau}\int_0^{\tau}Z_0(t)\mathrm{d} t\right)\wedge\frac{1}{\underline{\gamma}}.
\]
Since $\gamma\geq\underline{\gamma}$, we have $1/\gamma\leq 1/\underline{\gamma}$, and 
\begin{align}\label{eq:fraction-hat-gamma}
\left\lvert \frac{1}{\hat{\gamma}}-\frac{1}{\gamma}\right\rvert=\left\lvert \left(\frac{1}{\tau}\int_0^{\tau}Z_0(t)\mathrm{d} t\right)\wedge\frac{1}{\underline{\gamma}}-\frac{1}{\gamma}\right\rvert\leq\left\lvert\frac{1}{\tau}\int_0^{\tau}Z_0(t)\mathrm{d} t-\frac{1}{\gamma}\right\rvert.
\end{align}

Next, by Chebyshev's inequality,
\[
\mathbb{P}_x(\hat{\gamma} \in G_{\varepsilon_0}^c)\leq\mathbb{P}_x\left(\left\lvert\frac{1}{\hat{\gamma}}-\frac1\gamma\right\rvert>\eta_{\varepsilon_0}\right)\leq\frac{1}{\eta_{\varepsilon_0}^2}\mathbb{E}_x\left[\left\lvert\frac{1}{\hat{\gamma}}-\frac1\gamma\right\rvert^2\right]\leq\frac{\gamma^2(\gamma+\varepsilon_0)^2C_3(1+\lvert x\rvert^4)}{\varepsilon_0^2\tau}.
\]
The last inequality follows from \eqref{eq:fraction-hat-gamma} and Lemma~\ref{lemma:MSE-bound}.

Finally, on the good event $\{\hat\gamma \in G_{\varepsilon_0}\}$, we have $\hat{\gamma}\leq \gamma+\varepsilon_0$. Therefore,
\[
\left\lvert\hat{\gamma}-\gamma\right\rvert=\gamma\hat{\gamma}\left\lvert\frac{1}{\hat{\gamma}}-\frac{1}{\gamma}\right\rvert\leq\gamma(\gamma+\varepsilon_0)\left\lvert\frac{1}{\hat{\gamma}}-\frac1\gamma\right\rvert.
\]
We then obtain
\[
\mathbb{E}_x\left[\lvert\hat\gamma-\gamma\rvert^2\mathbf{1}_{\{\hat{\gamma} \in G_{\varepsilon_0}\}}\right]\leq\gamma^2(\gamma+\varepsilon_0)^2
\mathbb{E}_x\left[\left\lvert\frac{1}{\hat{\gamma}}-\frac1\gamma\right\rvert^2\right].
\]
Using the same bound above, we complete the proof.
\end{proof}

\begin{proof}[Proof of Lemma~\ref{lemma:bound-X-4}]
We use the fact that by Corollary~\ref{Cor:1}, 
\[
-K_h\leq X_{k-1}\leq Z_0(\tau_{k-1}),\quad \text{a.s.}
\]
Thus, by Jensen's inequality for function $x^4$, we have
\[
\lvert X_{k-1}\rvert^4\leq 2^3\left(K_h^4+\left(Z_0(\tau_{k-1})\right)^4\right)
\]
and by Lemma~\ref{lemma:Zr-polynomial-bound}, we get for any $k\geq1$,
\begin{align*}
    \mathbb{E}_x\left[\lvert X_{k-1}\rvert^4\right]&\leq2^3\left(K_h^4+\mathbb{E}_x\left[ \lvert Z_0(\tau_{k-1})\rvert^4\right]\right)\nonumber\leq2^3\left(K_h^4+\sup_{t\geq0}\mathbb{E}_x\left[ \lvert Z_0(t)\rvert^4\right]\right)\nonumber\\
    &\leq 2^3\left(K_h^4+2^6\left(x^4+\frac{4!}{\gamma^4}\right)\right).\qedhere
\end{align*}
\end{proof}

\begin{proof}[Proof of Lemma~\ref{lemma:MSE-bound-4-AUFH}]
    Recall the $L^2$-norm for a random variable $\eta$: 
    $\|\eta\|_{L^2}:=\left(\mathbb{E}_x\left[\eta^2\right]\right)^{\frac{1}{2}}.$
    Let $\tilde{e}_i:=\tilde{q}_i-\frac{1}{\gamma}$. Then we have
    \[
    \tilde{q}_k^{\mathrm{FH}}-\frac{1}{\gamma}=\frac{1}{\tau_k}\sum_{i=0}^{k-1}\Delta\tau_i\tilde{q}_{i+1}-\frac{1}{\gamma}=\frac{1}{\tau_k}\sum_{i=0}^{k-1}\Delta\tau_i\tilde{e}_{i+1}.
    \]
    Thus, 
    \begin{align*}
        \left\|\tilde{q}_{k}^{\mathrm{FH}}-\frac{1}{\gamma}\right\|_{L^2}=\left(\mathbb{E}_x\left[\left(\tilde{q}_{k}^{\mathrm{FH}}-\frac{1}{\gamma}\right)^2\right]\right)^{\frac{1}{2}}=\mathbb{E}_x\left[\left(\frac{1}{\tau_k}\sum_{i=0}^{k-1}\Delta\tau_i\tilde{e}_{i+1}\right)^2\right]=\frac{1}{\tau_k}\left\|\sum_{i=0}^{k-1}\Delta\tau_i\tilde{e}_{i+1}\right\|_{L^2}.
    \end{align*}
    By Minkowski inequality, we have
    \[
    \left\|\sum_{i=0}^{k-1}\Delta\tau_i\tilde{e}_{i+1}\right\|_{L^2}\leq\sum_{i=0}^{k-1}\left\|\Delta\tau_i\tilde{e}_{i+1}\right\|_{L^2}.
    \]
    So we obtain
    \begin{align*}
        \left\|\tilde{q}_{k}^{\mathrm{FH}}-\frac{1}{\gamma}\right\|_{L^2}\leq\frac{1}{\tau_k}\sum_{i=0}^{k-1}\Delta\tau_i\left\|\tilde{e}_{i+1}\right\|_{L^2}\leq\frac{1}{\tau_k}\sum_{i=0}^{k-1}\Delta\tau_i\sqrt{\frac{C_4}{\Delta\tau_i}}=\frac{\sqrt{C_4}}{\tau_k}\sum_{i=0}^{k-1}\sqrt{\Delta\tau_i},
    \end{align*}
    where the second inequality follows from \eqref{eq:MSE-q-i}. Since $\sum_{i=0}^{k-1}\sqrt{\Delta\tau_i}=\frac{2^{k/2}-1}{\sqrt{2}-1}$, we have
    \begin{align*}
        \mathbb{E}_x\left[\left(\tilde{q}_{k}^{\mathrm{FH}}-\frac{1}{\gamma}\right)^2\right]&\leq\frac{C_4}{(2^k-1)^2}\left(\frac{2^{k/2}-1}{\sqrt{2}-1}\right)^2\leq \frac{C_4}{(2^k-1)^2}\frac{2^k}{\left(\sqrt{2}-1\right)^2} \\
        & \leq \frac{C_4}{\left(\sqrt{2}-1\right)^2}\frac{2}{2^k-1}=\frac{2C_4}{\left(\sqrt{2}-1\right)^2}\frac{1}{\tau_k}, 
    \end{align*}
    where the last inequality follows from $2^k\leq 2(2^k-1)$ for $k\geq1$. \eqref{eq:MSE-q-FH} holds with $C_5=\frac{2C_4}{\left(\sqrt{2}-1\right)^2}$.
\end{proof}

\section{Supplementary Numerical Experiments}\label{app:C}
In this section, we provide some additional numerical experiment results. We first present a sample-path comparison among the full-information optimal policy, the LTO algorithm, the AU algorithm, and the AU-FH algorithm driven by the same Brownian motion. 
Next, we introduce the REINFORCE algorithm mentioned in Section~\ref{sec:7.1}. Finally, we supplement figures showing the impact of the diffusion scale on AU and AU-FH regrets.

\subsection{Sample-path Illustration}\label{app:C.1}
We simulate three controlled processes for cost function $h(x)=x^2$:  the reflected Brownian motion under the optimal reflecting level $r^*$, the process controlled by the LTO algorithm, the process controlled by the AU algorithm, and the process controlled by the AU-FH algorithm in Figures~\ref{fig:sample-opt}, \ref{fig:sample-lto}, \ref{fig:sample-au} and~\ref{fig:sample-aufh}, respectively. The optimal policy reflects the process at the optimal boundary $r^*$. The LTO algorithm first reflects the process at the exploratory boundary $0$, obtains an estimate $\hat{\gamma}$ of the true parameter $\gamma$, and then switches to the plug-in boundary $r(\hat{\gamma})$. The AU and AU-FH algorithms repeatedly update the estimate and the corresponding reflecting boundary along the doubling schedule. In the LTO, AU, and AU-FH plots, different colors indicate switching reflection regimes along the sample path.

\begin{figure}[htbp]
    \centering
    \begin{subfigure}[b]{0.49\textwidth}
        \centering
        \includegraphics[width=\textwidth]{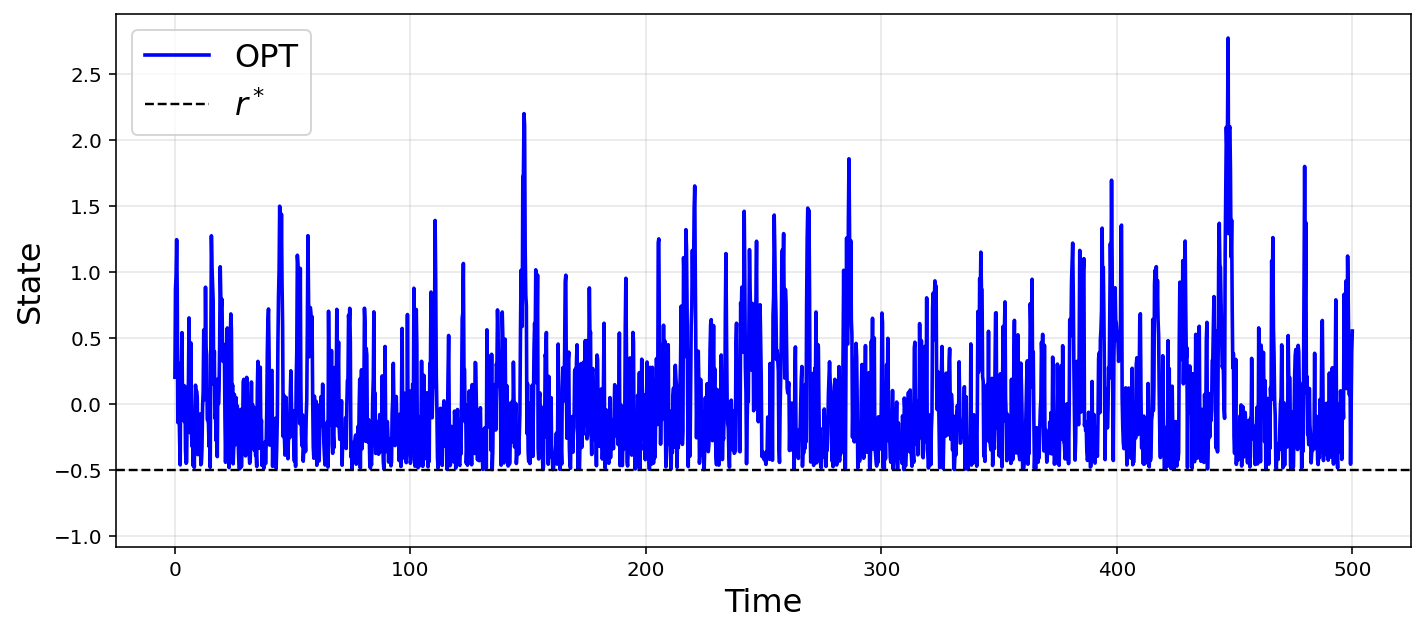}
        \caption{Optimal policy sample path}
        \label{fig:sample-opt}
    \end{subfigure}
    \hfill
    \begin{subfigure}[b]{0.49\textwidth}
        \centering
        \includegraphics[width=\textwidth]{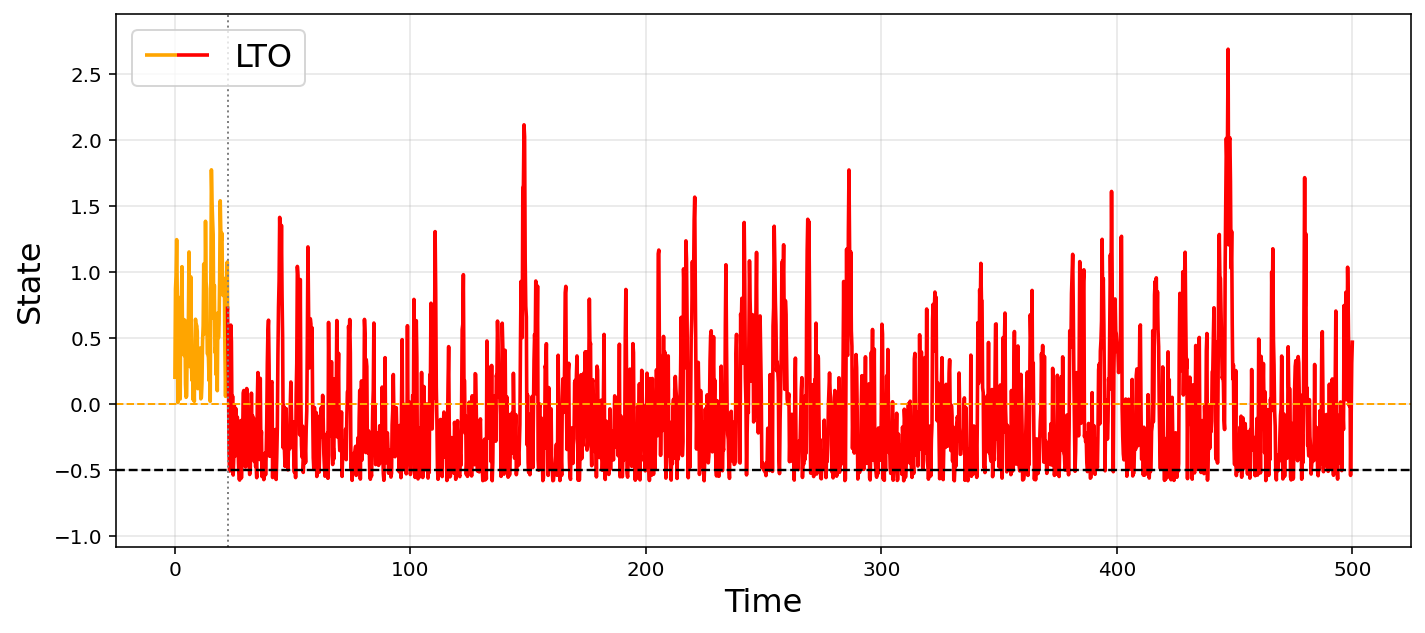}
        \caption{LTO sample path}
        \label{fig:sample-lto}
    \end{subfigure}
    \\
    \begin{subfigure}[b]{0.49\textwidth}
        \centering
        \includegraphics[width=\textwidth]{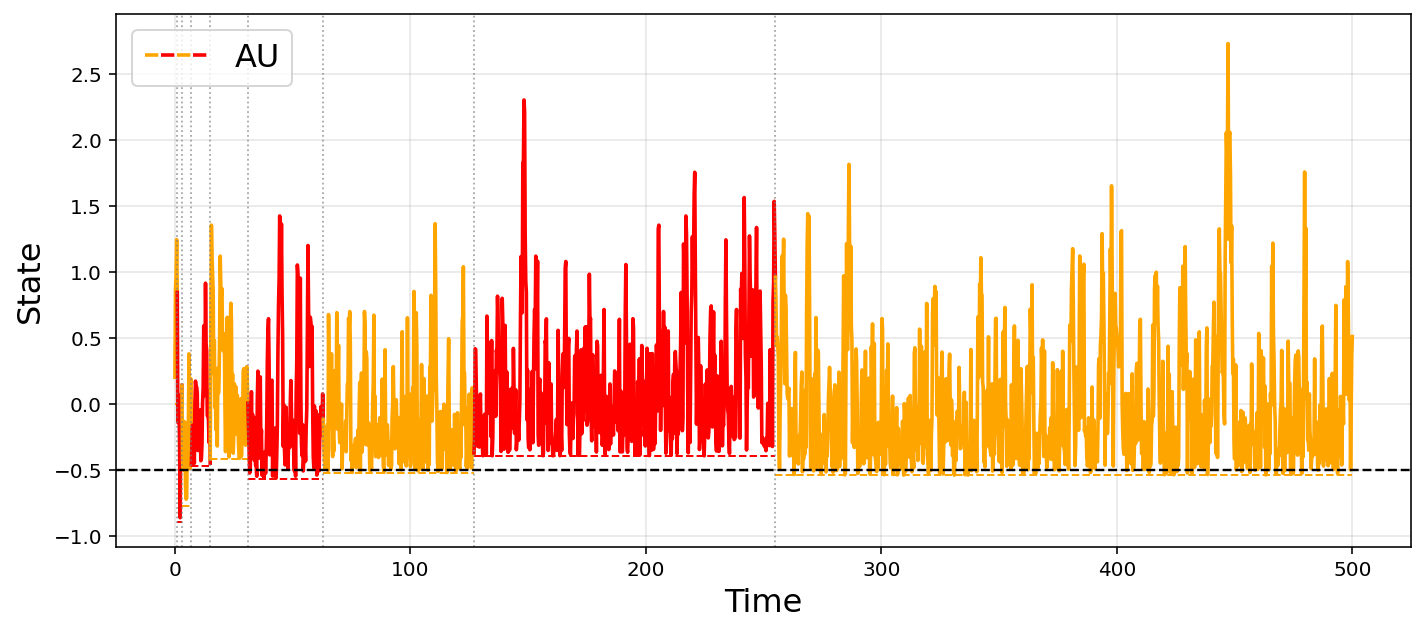}
        \caption{AU sample path}
        \label{fig:sample-au}
    \end{subfigure}
    \begin{subfigure}[b]{0.49\textwidth}
        \centering
        \includegraphics[width=\textwidth]{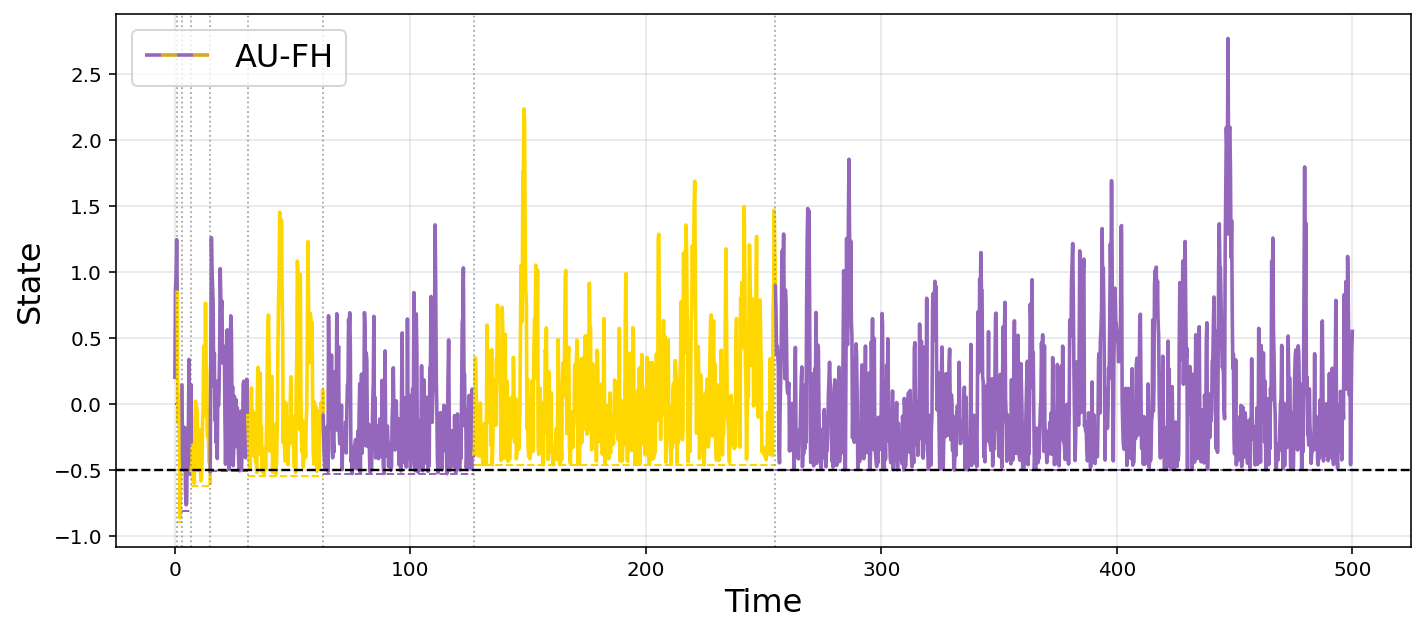}
        \caption{AU-FH sample path}
        \label{fig:sample-aufh}
    \end{subfigure}
    \caption{Sample paths under a common Brownian motion}
\end{figure}

For a sample-path realization $\{Z_{r^*}(t):t\geq0\}$, $\{Z^{\mathrm{LTO}}(t):t\geq0\}$, $\{Z^{\mathrm{AU}}(t):t\geq0\}$, and $\{Z^{\mathrm{AU}\text{-}\mathrm{FH}}(t):t\geq0\}$ driven by the same Brownian motion, the sample-path regret of the LTO, AU and AU-FH algorithms up to time $T$ are  $\int_0^T\left(Z^{\chi}(s)-Z_{r^*}(s)\right)\mathrm{d}s$, with $\chi$ = LTO, AU, AU-FH, respectively.
Figures~\ref{fig:regret-lto}, \ref{fig:regret-au} and~\ref{fig:regret-aufh} show the sample-path regrets of the LTO algorithm, the AU algorithm, and the AU-FH algorithm.
Due to randomness, on some sample paths in the early period, the cost generated by both algorithms may be lower than that of the full-information optimal policy.

\begin{figure}[htbp]
    \centering
    \begin{subfigure}[b]{0.49\textwidth}
        \centering
        \includegraphics[width=\textwidth]{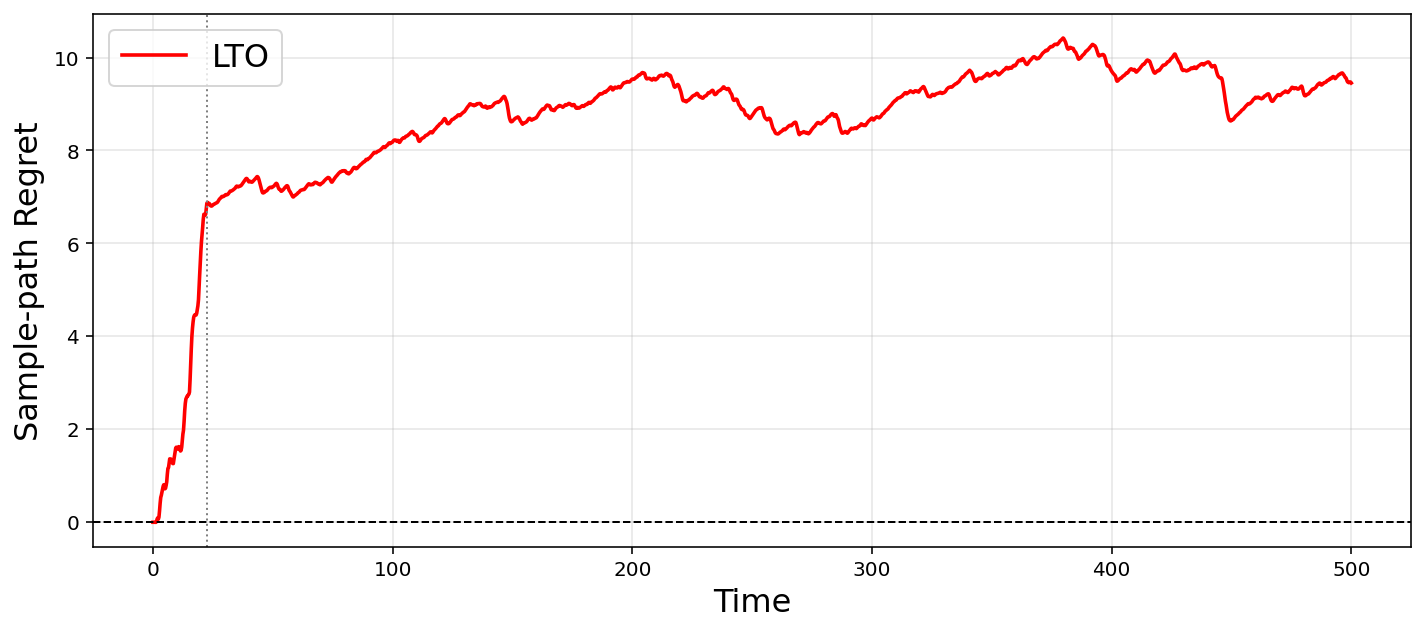}
        \caption{LTO sample-path regret}
        \label{fig:regret-lto}
    \end{subfigure}
    \hfill
    \begin{subfigure}[b]{0.49\textwidth}
        \centering
        \includegraphics[width=\textwidth]{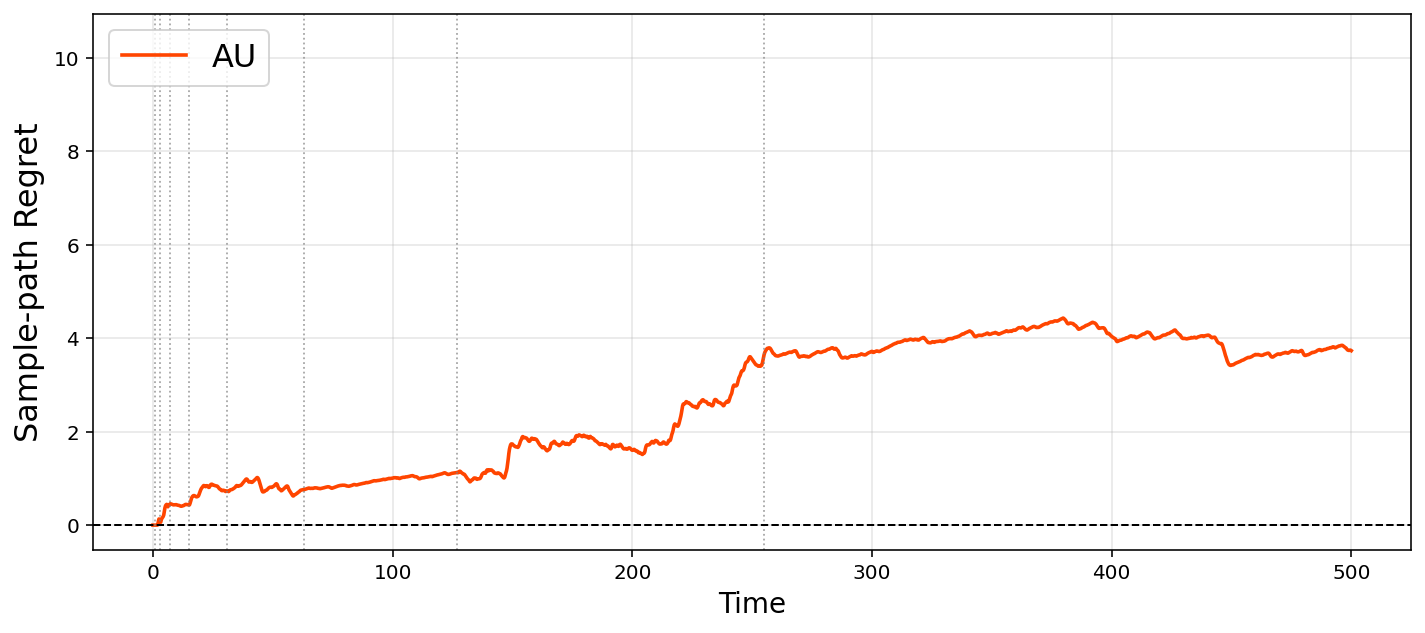}
        \caption{AU sample-path regret}
        \label{fig:regret-au}
    \end{subfigure}
    \\
    \begin{subfigure}[b]{0.49\textwidth}
        \centering
        \includegraphics[width=\textwidth]{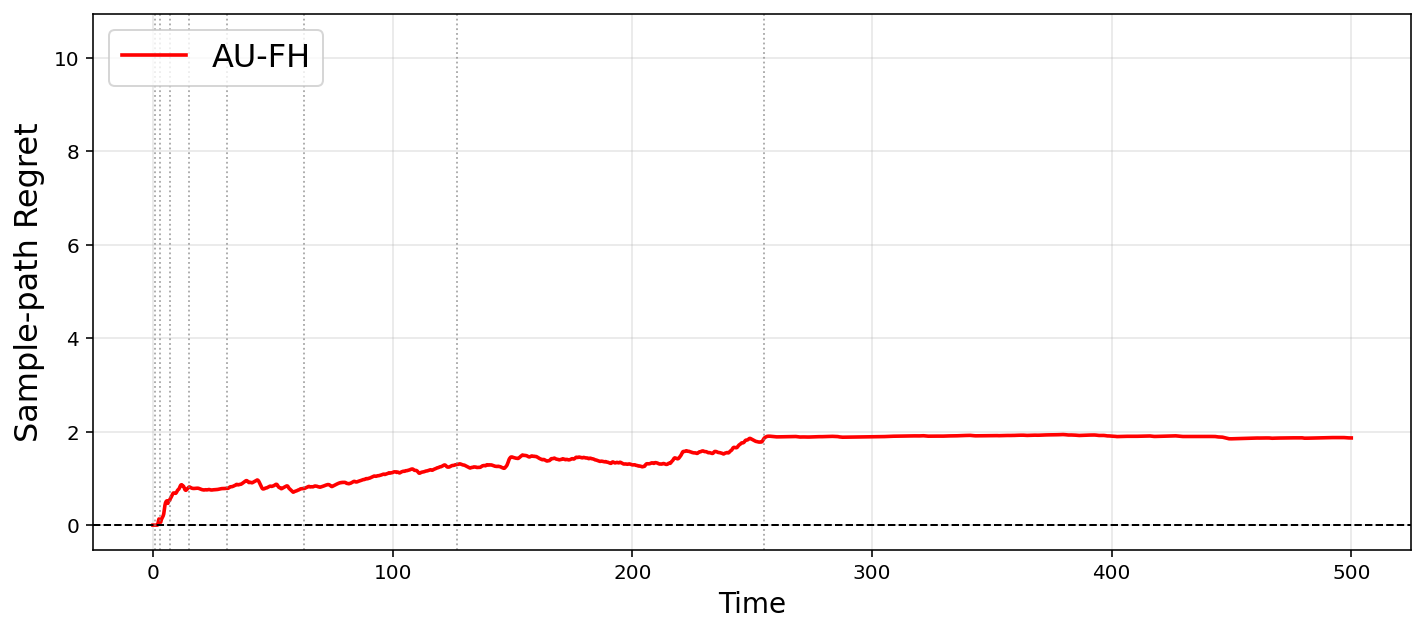}
        \caption{AU-FH sample-path regret}
        \label{fig:regret-aufh}
    \end{subfigure}
    \caption{Sample-path regret}
\end{figure}

\subsection{REINFORCE Algorithm}\label{app:C.2}
We use REINFORCE, a classical model-free policy-gradient RL algorithm, as a benchmark for our reflection control learning problem. We present an episodic implementation in Algorithm~\ref{alg:5}. The reflection policy is parameterized by  $\phi_m$. In each episode, a reflection level is sampled from the current policy and implemented throughout the episode. The trajectory data collected during the episode are then used to evaluate the episode cost and update the policy parameter via the policy-gradient step in \eqref{eq:policy-gradient}. The baseline $b_m$ is included to reduce the variance of the gradient estimator. In the implementation, we choose $H=5$ and $\Delta t=0.002$, so each episode contains $2,500$ simulated transitions.

\begin{algorithm}[htbp]
\caption{Episodic REINFORCE algorithm}
\label{alg:5}
\begin{algorithmic}[1]
    \STATE Choose the episode length $H$, time step $\Delta t$,
    and set $N_{\mathrm{ep}}:=H/\Delta t$. 
    Choose a feasible interval $[r_{\min},0]$ for the reflection level, where $r_{\min}<0$. Initialize the policy parameter $\phi_0$, policy standard deviation $\varsigma>0$, baseline $b_0=0$, and state $Z_0=x$.
    \FOR{$m=0,1,\ldots,M-1$}
        \STATE Sample $U_m\sim\mathcal{N}(\phi_m,\varsigma^2)$.
        \STATE Transform $U_m$ into a feasible reflection level: $R_m:=r_{\min}+(-r_{\min})\frac{1}{1+\exp(-U_m)}$.
        \STATE Set the initial state of episode $m$ as $Z_{m,0}:=Z_m\vee R_m$.
        \STATE Initialize the learning rate $\alpha_m$, $\beta_m\in(0,1]$ and the cumulative episode cost $G_m=0$.
        \FOR{$n=0,1,\ldots,N_{\mathrm{ep}}-1$}
            \STATE Generate $Z_{m,n+1}$ using policy $R_m$.
            \STATE Accumulate the holding cost: $G_m\leftarrow G_m+h(Z_{m,n})\Delta t$.
        \ENDFOR
        \STATE Update the policy parameter:
        \begin{align}\label{eq:policy-gradient}
        \phi_{m+1}:=\phi_m-\alpha_m(G_m-b_m)\nabla_{\phi_m}\log p_{\phi_m}(U_m),
        \end{align}
        where $p_{\phi_m}(\cdot)$ is the probability density function of $\mathcal{N}(\phi_m,\varsigma^2)$.
        \STATE Update the baseline: $b_{m+1}:=(1-\beta_m)b_m+\beta_m G_m$.
        \STATE Set the state for the next episode: $Z_{m+1}:=Z_{m,N_{\mathrm{ep}}}$.
    \ENDFOR
\end{algorithmic}
\end{algorithm}

\subsection{Sensitivity Analysis of Volatility Parameter}\label{app:C.3}
For fixed $\theta=-1$, Figures~\ref{fig:sigma-level} and~\ref{fig:sigma-cost} show the optimal reflection level $r^*$ and optimal long-run average cost $C(\gamma,r^*)$ as a function of $\sigma$, respectively.
The figures highlight that the volatility parameter has a substantial effect on the full-information benchmark. Thus, it is important to learn under the unknown $\sigma$.

\begin{figure}[htbp]
    \centering
    \begin{subfigure}[b]{0.48\textwidth}
        \centering
        \includegraphics[width=\textwidth]{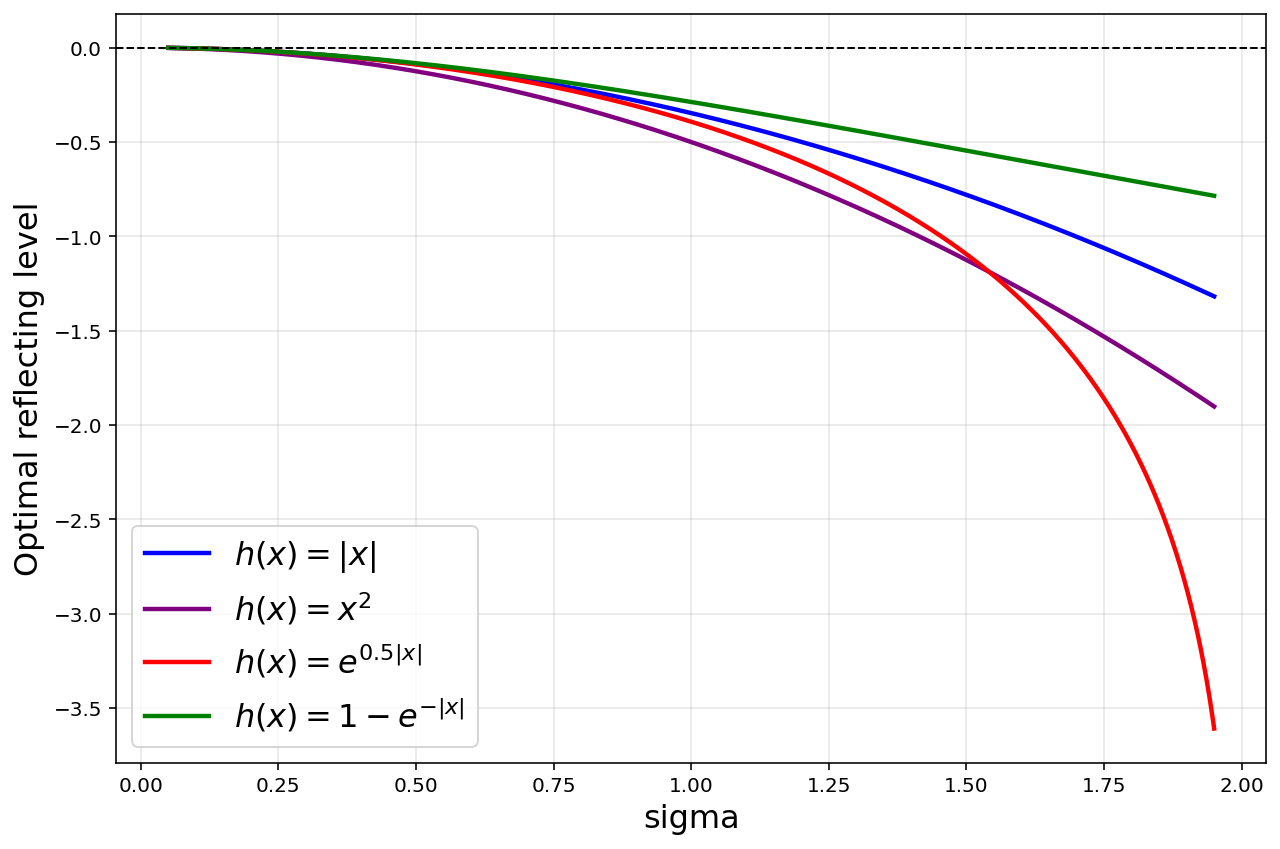}
        \caption{Effect on optimal reflection level}
        \label{fig:sigma-level}
    \end{subfigure}
    \hfill
    \begin{subfigure}[b]{0.48\textwidth}
        \centering
        \includegraphics[width=\textwidth]{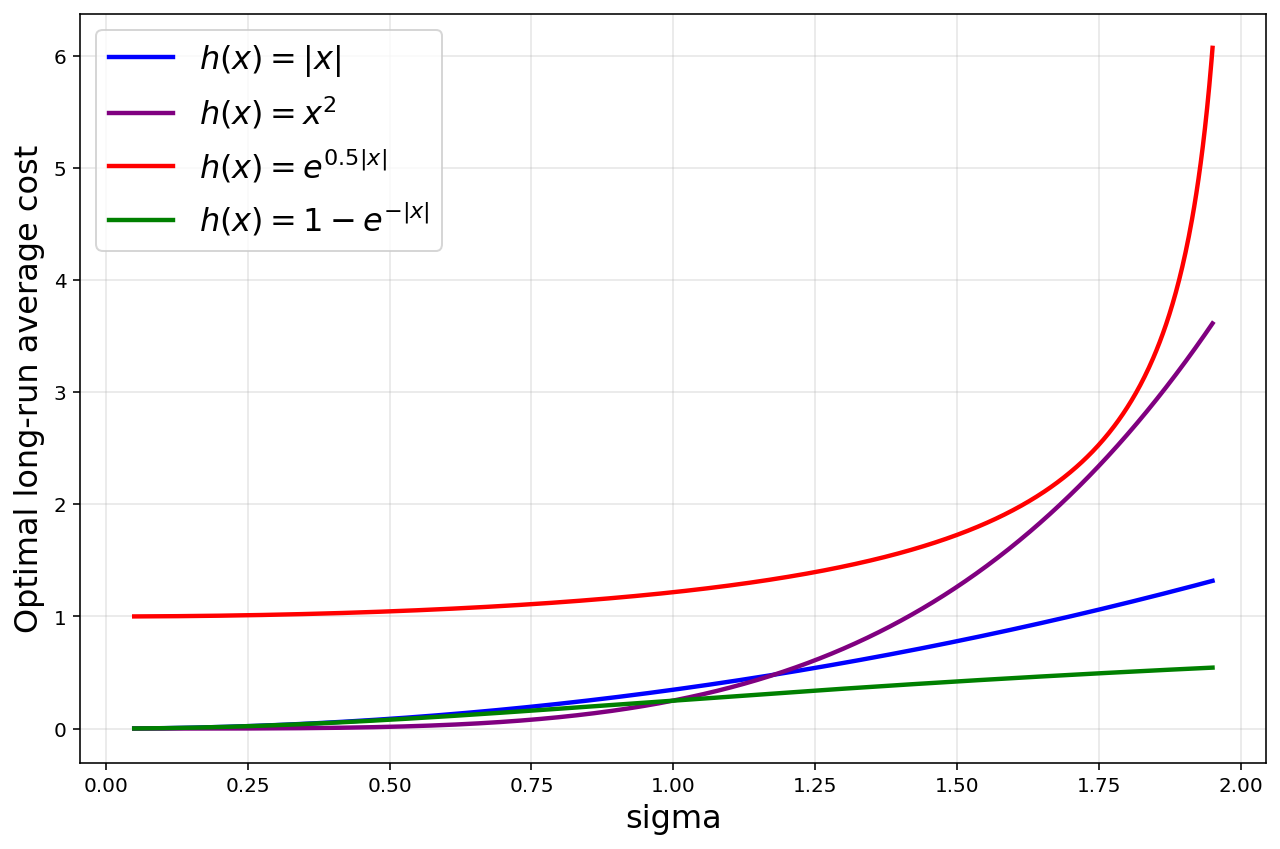}
        \caption{Effect on optimal cost}
        \label{fig:sigma-cost}
    \end{subfigure}
    \caption{}
\end{figure}

\subsection{Supplementary Figures}\label{app:C.4}
The effect of the diffusion scale $\sigma$ on AU regret under other cost functions are in Figures~\ref{fig:change-sigma-qua}, \ref{fig:change-sigma-exp} and~\ref{fig:change-sigma-bound}.
\begin{figure}[htbp]
    \centering
    \begin{subfigure}[b]{0.45\textwidth}
        \centering
        \includegraphics[width=\textwidth]{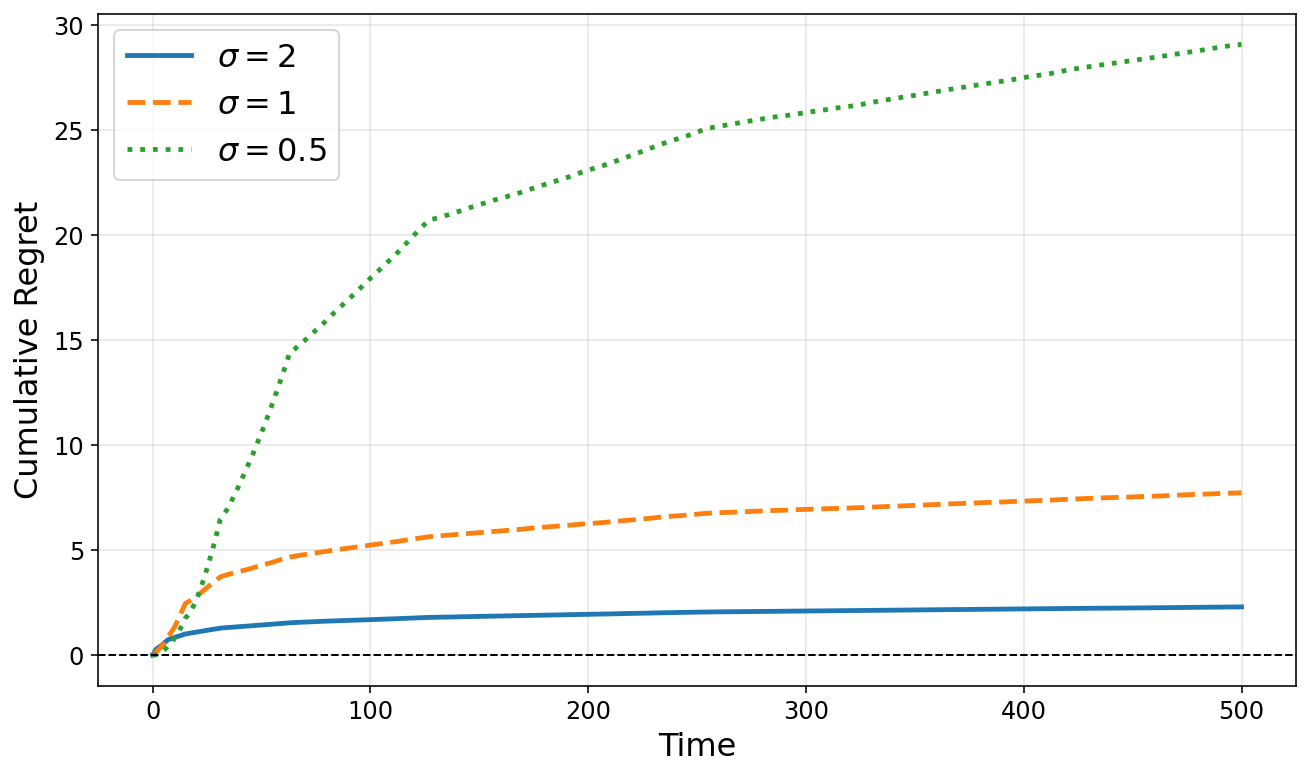}
        \caption{$h(x)=x^2$}
        \label{fig:change-sigma-qua}
    \end{subfigure} 
    \hfill
    \begin{subfigure}[b]{0.45\textwidth}
        \centering
        \includegraphics[width=\textwidth]{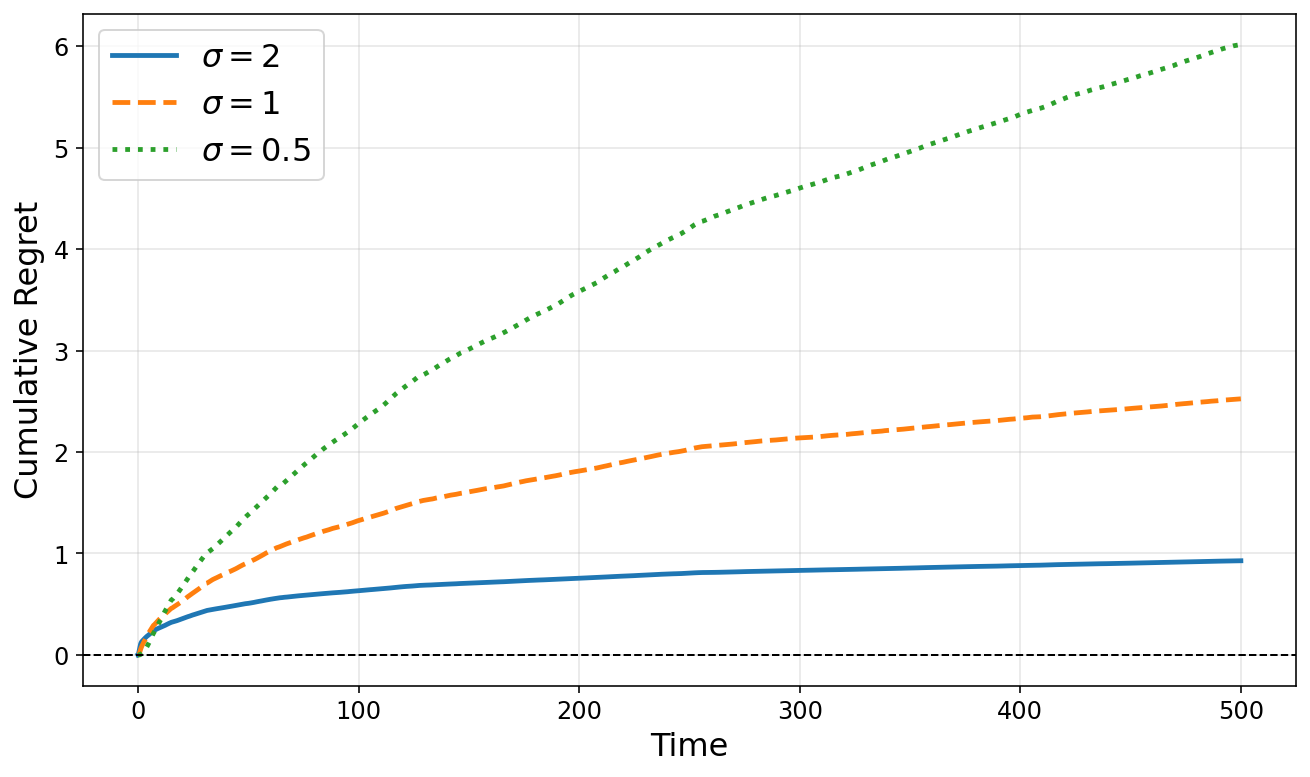}
        \caption{$h(x)=e^{0.5\lvert x\rvert}$}
        \label{fig:change-sigma-exp}
    \end{subfigure}
    \\
    \begin{subfigure}[b]{0.45\textwidth}
        \centering
        \includegraphics[width=\textwidth]{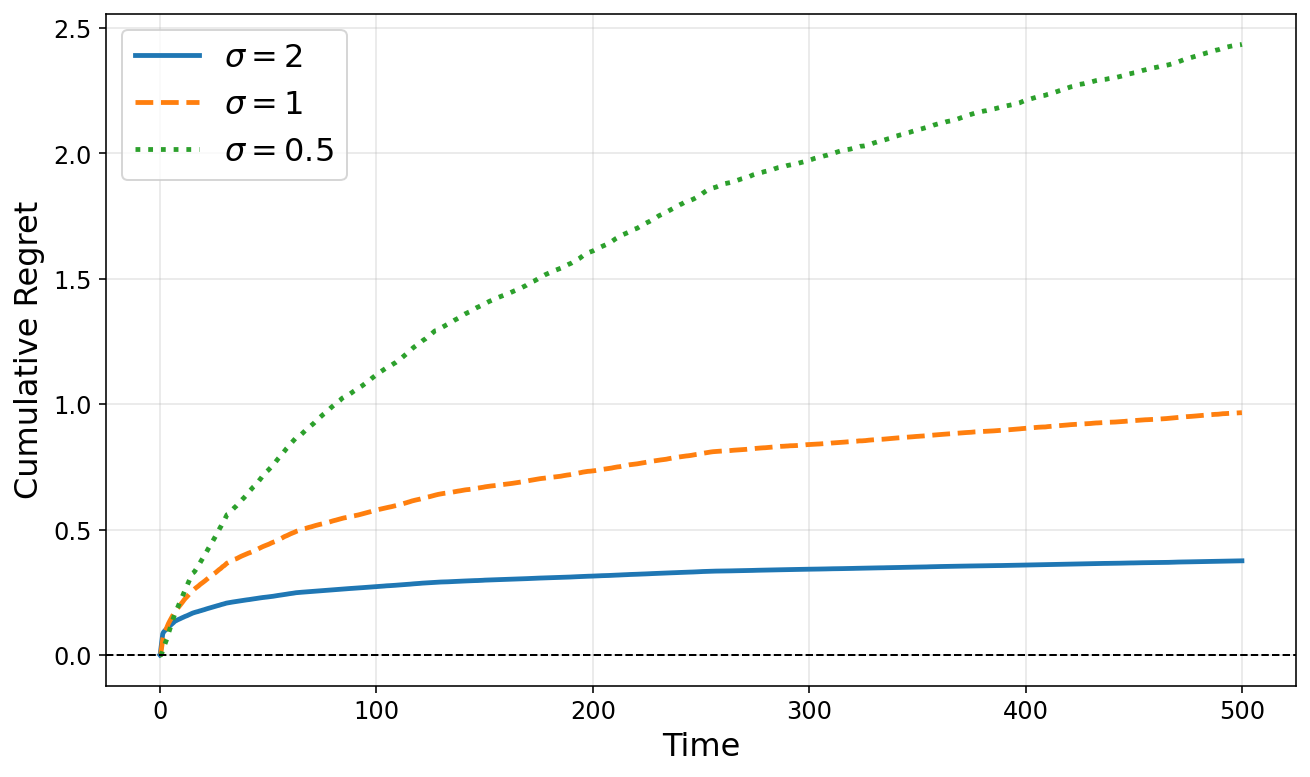}
        \caption{$h(x)=1-e^{-\lvert x\rvert}$}
        \label{fig:change-sigma-bound}
    \end{subfigure}
    \caption{Effect of diffusion scale on AU regret under various cost functions}
\end{figure}

For the AU-FH algorithm, there is also a similar effect of diffusion scale on regret, see Figure~\ref{fig:change-sigma-abs-aufh} for $h(x)=|x|$.
\begin{figure}[htbp]
    \centering
    \includegraphics[scale=0.45]{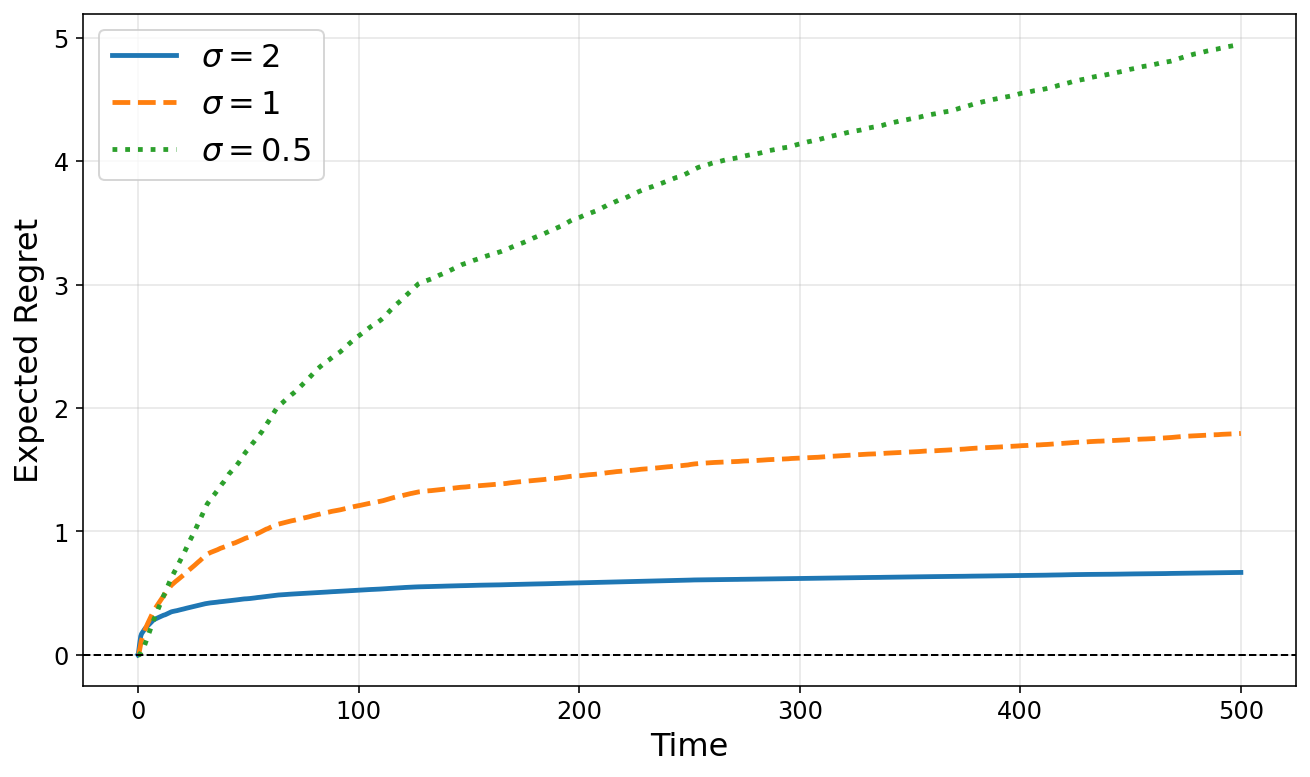}
    \caption{Effect of diffusion scale on AU-FH regret, $h(x)=|x|$}
    \label{fig:change-sigma-abs-aufh}
\end{figure}

\end{document}